\documentclass[11pt,a4paper]{amsart}

\usepackage{latexsym}
\usepackage{amsmath}
\usepackage{amsthm}
\usepackage{amssymb}
\usepackage{vmargin}
\usepackage{amscd}
\usepackage{stmaryrd}
\usepackage{euscript}
\usepackage{mathrsfs}
\usepackage{amscd}
\usepackage[all]{xy}
\usepackage{xr}
\DeclareMathAlphabet{\mathpzc}{OT1}{pzc}{m}{it}

\newcommand\da{\!\downarrow\!}

\newcommand\la{\leftarrow}
\newcommand\lra{\longrightarrow}
\newcommand\lla{\longleftarrow}
\newcommand\id{\mathrm{id}}

\newcommand\ten{\otimes}

\newcommand\CC{\mathrm{C}}

\newcommand\EE{\mathrm{E}}

\newcommand\Ru{\mathrm{R_u}}
\newcommand\Th{\mathrm{Th}\,}

\renewcommand\H{\mathrm{H}}
\newcommand\z{\mathrm{Z}}

\newcommand\N{\mathbb{N}}
\newcommand\Z{\mathbb{Z}}
\newcommand\Q{\mathbb{Q}}
\newcommand\Ql{\mathbb{Q}_{\ell}}
\newcommand\Zl{\mathbb{Z}_{\ell}}
\newcommand\R{\mathbb{R}}
\newcommand\Cx{\mathbb{C}}

\newcommand\vv{\mathbb{V}}

\newcommand\bA{\mathbb{A}}

\newcommand\bF{\mathbb{F}}
\newcommand\bG{\mathbb{G}}
\newcommand\bH{\mathbb{H}}
\newcommand\bI{\mathbb{I}}

\newcommand\bL{\mathbb{L}}

\newcommand\bO{\mathbb{O}}

\newcommand\bS{\mathbb{S}}

\newcommand\bU{\mathbb{U}}

\newcommand\cC{\mathcal{C}}
\newcommand\cD{\mathcal{D}}
\newcommand\cE{\mathcal{E}}
\newcommand\cF{\mathcal{F}}
\newcommand\cG{\mathcal{G}}

\newcommand\cN{\mathcal{N}}

\newcommand\cP{\mathcal{P}}

\newcommand\cT{\mathcal{T}}

\newcommand\cW{\mathcal{W}}

\renewcommand\O{\mathscr{O}}

\newcommand\sA{\mathscr{A}}

\newcommand\sC{\mathscr{C}}

\newcommand\sE{\mathscr{E}}
\newcommand\sF{\mathscr{F}}

\newcommand\sS{\mathscr{S}}

\renewcommand\L{\Lambda}

\newcommand\m{\mathfrak{m}}

\newcommand\g{\mathfrak{g}}

\newcommand\fh{\mathfrak{h}}

\newcommand\fp{\mathfrak{p}}

\newcommand\fu{\mathfrak{u}}
\newcommand\fv{\mathfrak{v}}

\renewcommand\hom{\mathscr{H}\!\mathit{om}}
\newcommand\cHom{\mathcal{H}\!\mathit{om}}

\newcommand\Ho{\mathrm{Ho}}

\newcommand\Alg{\mathrm{Alg}}

\newcommand\Mod{\mathrm{Mod}}

\newcommand\Hom{\mathrm{Hom}}

\newcommand\HHom{\underline{\mathrm{Hom}}}

\newcommand\Aut{\mathrm{Aut}}

\newcommand\ROut{\mathrm{ROut}}
\newcommand\RAut{\mathrm{RAut}}

\newcommand\Iso{\mathrm{Iso}}
\newcommand\Isoc{\mathrm{Isoc}}

\newcommand\Gal{\mathrm{Gal}}

\newcommand\Ob{\mathrm{Ob}\,}

\newcommand\CoLie{\mathrm{CoLie}}
\newcommand\Ind{\mathrm{Ind}}

\newcommand\Ab{\mathrm{Ab}}
\newcommand\Top{\mathrm{Top}}
\newcommand\Gp{\mathrm{Gp}}
\newcommand\agp{\mathrm{AGp}}
\newcommand\agpd{\mathrm{AGpd}}

\newcommand\mal{\mathrm{Mal}}

\newcommand\Spec{\mathrm{Spec}\,}

\newcommand\Set{\mathrm{Set}}

\newcommand\Aff{\mathrm{Aff}}

\newcommand\Sp{\mathrm{Sp}}

\newcommand\sch{\mathrm{sch}}

\newcommand\Sing{\mathrm{Sing}}
\newcommand\FD{\mathrm{FD}}

\newcommand\ad{\mathrm{ad}}

\newcommand\Lim{\varprojlim}
\newcommand\LLim{\varinjlim}

\newcommand\into{\hookrightarrow}
\newcommand\onto{\twoheadrightarrow}
\newcommand\abuts{\implies}
\newcommand\xra{\xrightarrow}

\newcommand\pr{\mathrm{pr}}

\newcommand\alg{\mathrm{alg}}

\newcommand\bt{\bullet}
\newcommand\by{\times}

\newcommand\mc{\mathrm{MC}}

\newcommand\Gg{\mathrm{Gg}}

\newcommand\Rep{\mathrm{Rep}}

\newcommand\Symm{\mathrm{Symm}}
\newcommand\SL{\mathrm{SL}}
\newcommand\GL{\mathrm{GL}}

\newcommand\et{\acute{\mathrm{e}}\mathrm{t}}

\newcommand\nr{\mathrm{nr}}
\newcommand\pnr{\mathrm{pnr}}

\newcommand\an{\mathrm{an}}

\newcommand\Tot{\mathrm{Tot}\,}
\newcommand\diag{\mathrm{diag}\,}

\newcommand\pro{\mathrm{pro}}

\newcommand\pd{\partial}

\newcommand\half{\frac{1}{2}}

\newcommand\cris{\rm{cris}}
\newcommand\pcris{\rm{pcris}}

\newcommand\Gr{\mathrm{Gr}}
\newcommand\gr{\mathrm{gr}}

\newcommand\ab{\mathrm{ab}}
\newcommand\cts{\mathrm{cts}}
\newcommand\gp{\mathrm{Gp}}
\newcommand\gpd{\mathrm{Gpd}}
\newcommand\Gpd{\mathrm{Gpd}}

\newcommand\Fil{\mathrm{Fil}}

\newcommand\algpia{\varpi_1(X,\bar{x})}

\newcommand\w{{}^W\!}

\renewcommand\alg{\mathrm{alg}}
\newcommand\red{\mathrm{red}}

\newcommand\Lie{\mathrm{Lie}}

\newcommand\sk{\mathrm{sk}}
\newcommand\cosk{\mathrm{cosk}}

\newcommand\op{\mathrm{opp}}
\newcommand\opp{\mathrm{opp}}

\newcommand\co{\colon\thinspace}

\newcommand\oR{\mathbf{R}}

\newcommand\oL{\mathbf{L}}

\newcommand\uleft\underleftarrow
\newcommand\uline\underline
\newcommand\uright\underrightarrow

\newtheorem{theorem}{Theorem}[section]
\newtheorem{proposition}[theorem]{Proposition}
\newtheorem{corollary}[theorem]{Corollary}

\newtheorem{lemma}[theorem]{Lemma}
\newtheorem*{theorem*}{Theorem}
\newtheorem*{proposition*}{Proposition}
\newtheorem*{corollary*}{Corollary}
\newtheorem*{lemma*}{Lemma}
\newtheorem*{conjecture*}{Conjecture}

\theoremstyle{definition}
\newtheorem{definition}[theorem]{Definition}

\newtheorem*{definition*}{Definition}

\theoremstyle{remark}
\newtheorem{example}[theorem]{Example}
\newtheorem{examples}[theorem]{Examples}
\newtheorem{remark}[theorem]{Remark}
\newtheorem{remarks}[theorem]{Remarks}

\newtheorem{assumption}[theorem]{Assumption}

\newtheorem*{example*}{Example}
\newtheorem*{examples*}{Examples}
\newtheorem*{remark*}{Remark}
\newtheorem*{remarks*}{Remarks}
\newtheorem*{exercise*}{Exercise}

\title{Galois actions on homotopy groups of algebraic varieties}
\author{J.P.Pridham}

\thanks{
This work was supported by Trinity College, Cambridge; and by the Engineering and Physical Sciences Research Council [grant number  EP/F043570/1].}

\begin{document}

\begin{abstract}
We study the Galois actions on the $\ell$-adic schematic and Artin--Mazur homotopy groups of algebraic varieties. For proper varieties  of good reduction over a local field $K$, we show that the $\ell$-adic schematic homotopy groups are mixed representations explicitly determined by the Galois action on cohomology of Weil sheaves, whenever $\ell$ is not equal to the residue characteristic $p$ of $K$. For quasi-projective varieties of good reduction, there is a similar characterisation involving the Gysin spectral sequence. When $\ell=p$, a slightly weaker result is proved by comparing the crystalline and $p$-adic schematic homotopy types. Under favourable conditions, a comparison theorem   transfers all these descriptions to the Artin--Mazur homotopy groups $\pi_n^{\et}(X_{\bar{K}})\ten_{\hat{\Z}} \Ql$. 
\end{abstract}

\maketitle

\section*{Introduction}

In \cite{arma}, Artin and Mazur introduced the \'etale homotopy type of an algebraic variety. This gives rise to \'etale homotopy groups $\pi_n^{\et}(X, \bar{x})$; these are pro-finite groups, abelian for $n \ge 2$, and $\pi_1^{\et}(X, \bar{x})$ is the usual \'etale fundamental group. In \cite[\S 3.5.3]{chaff}, an approach for defining $\ell$-adic schematic homotopy types was discussed, giving $\ell$-adic schematic homotopy groups $\varpi_n(X, \bar{x})$; these are (pro-finite-dimensional) $\Ql$-vector spaces when $n\ge 2$. In \cite{olssonhodge}, Olsson introduced a crystalline schematic homotopy type, and established a comparison theorem with the $p$-adic schematic homotopy type.

Thus, given a variety $X$ defined over a number field $K$, there are many notions of homotopy group:
\begin{itemize}
\item 
for each embedding $K \into \Cx$, both classical and schematic homotopy groups of the topological space $X_{\Cx}$;
\item 
the \'etale homotopy groups of $X_{\bar{K}}$;
\item 
the $\ell$-adic schematic homotopy groups of $X_{\bar{K}}$;
\item 
over localisations $K_{\fp}$ of $K$, the crystalline schematic homotopy groups of $X_{K_{\fp}}$. 
\end{itemize}
However, despite their long heritage, very little was known even about the relation between \'etale and classical homotopy groups, unless the variety is simply connected. 

The \'etale and $\ell$-adic homotopy types carry natural Galois actions, and the main aim of this paper is to study  their structure. 
In many respects, the analogous question for $X_{\Cx}$ has already been addressed, with \cite{KTP} and \cite{mhs} describing mixed Hodge structures on the classical and real schematic homotopy types.

In \cite{htpy}, a new approach to studying non-abelian cohomology and schematic homotopy types of topological spaces was introduced. Its primary application was to transfer cohomological data (in particular mixed Hodge structures) to give information about homotopy groups.  The  bulk of this paper is concerned with adapting those techniques to pro-simplicial sets. This allows us to study Artin--Mazur homotopy types of algebraic varieties, and to translate Lafforgue's Theorem and Deligne's Weil II theorems into statements about homotopy types. We 
thus establish arithmetic analogues of the results of \cite{mhs}, with Galois actions replacing mixed Hodge structures.

The main comparison results are Proposition \ref{piclasstohat} (showing when \'etale homotopy groups are pro-finite completions of classical homotopy groups), Theorem \ref{etpimal} 
(describing $\ell$-adic schematic homotopy groups in terms of \'etale homotopy groups), and Proposition \ref{crisequiv} (comparing $p$-adic and crystalline homotopy groups). 

If $X$ is smooth or proper and normal, then Corollary \ref{wgtexistspin} shows that the Galois actions on the $\ell$-adic schematic homotopy groups are mixed, with Remark \ref{wgtexistspinet} indicating when the same is true for \'etale homotopy groups.
Corollaries  \ref{formalrk} and  \ref{qformalrk} then show how to determine  $\ell$-adic schematic homotopy groups of smooth varieties over finite fields as Galois representations, by recovering them from cohomology groups of smooth Weil sheaves, thereby extending  \cite{weight1} 
from fundamental groups to higher homotopy groups, and indeed to the whole homotopy type.
Corollaries \ref{locqformalb1} and \ref{crisqformalrk} give  similar results for $\ell$-adic and $p$-adic homotopy groups of varieties over local fields. 

The structure of the paper is as follows.

In Section \ref{profinitesn}, we recall standard definitions of pro-finite  homotopy types 
and homotopy groups, and then establish some fundamental results. Proposition \ref{cohochar} shows how Kan's loop group can be used to construct the 
pro-finite
completion 
$\hat{X}$
of a space $X$, and Proposition \ref{piclasstohat} describes homotopy groups of  
$\hat{X}$.

Section \ref{review} reviews the pro-algebraic homotopy types of \cite{htpy}, with the formulation of multipointed  pro-algebraic homotopy types from \cite{mhs}, together with some new material on hypercohomology.

We adapt these results in Section \ref{algtypes}  to define non-abelian cohomology of a variety with coefficients in a simplicial algebraic group over $\mathbb{Q}_{\ell}$. The machinery developed in \cite{htpy} applies to give a pro-$\Ql$-algebraic homotopy type,  which is  a non-nilpotent generalisation of the $\mathbb{Q}_{\ell}$-homotopy type of Weil II (\cite{Weil2}). Its homotopy groups are $\ell$-adic schematic homotopy groups,  and Theorem \ref{etpimal} gives conditions for relating these to  \'etale homotopy groups. Explicitly,  if $\pi_1X$ is algebraically good (see Definition \ref{relgood2}), and the higher homotopy groups have finite rank, then the higher homotopy groups of the pro-$\Ql$-algebraic homotopy type are just $\pi_n^{\et}X\ten_{\hat{\Z}}\Ql$.
For complex varieties, we also compare the pro-algebraic homotopy types associated to the \'etale and analytic topologies. 

Section \ref{rftypes} contains technical results showing  how to extend the machinery of  Section \ref{algtypes} to relative and filtered homotopy types. The former facilitate $p$-adic Hodge theory, while the latter are developed in order to study quasi-projective varieties. We also explore what it means for a pro-discrete group to act algebraically on a homotopy type.
In Section \ref{galoisactions}, we investigate properties of homotopy types endowed with algebraic Galois actions.

In Section \ref{finite}, the techniques of \cite{weight1} for studying Galois actions on algebraic groups  then extend the finite characteristic results of \cite{gal} to  non-nilpotent and higher pro-$\Ql$-algebraic homotopy groups.  The results are  similar to \cite{mhs}, substituting  Frobenius actions for Hodge structures. Over finite fields, Theorem \ref{formal}  uses Lafforgue's Theorem and Deligne's Weil II theorems to show that the pro-$\Ql$-algebraic homotopy type of a  smooth projective variety is formal --- this means that it can be recovered from cup products on cohomology of local systems.
 For quasi-projective varieties, Corollary \ref{qformal} establishes a related property we call quasi-formality,  which is analogous to Morgan's description of the rational homotopy type (\cite{Morgan}) in terms of the Leray spectral sequence.

Section \ref{local} then addresses the same question, but over local fields.
In unequal characteristic, smooth specialisation  suffices to adapt results from finite characteristic for varieties with good reduction. In equal characteristic, we show how pro-$\Q_p$-algebraic homotopy types relate to the framework of $p$-adic Hodge theory. Proposition \ref{crisequiv} is a reworking of Olsson's non-abelian $p$-adic Hodge theory, and this has various consequences for Galois actions on Artin--Mazur homotopy types (Theorems \ref{outercris}--\ref{crisqformal}). Explicitly, the homotopy type becomes  formal as a Galois representation only after tensoring with the  ring $B_{\cris}^{\sigma}$ of Frobenius-invariant periods,  which means that the Hodge filtration is the only really new structure on the relative Malcev homotopy type (Remark \ref{finalrks}.2).

\subsection*{Acknowledgements}
I would like to record my sincere thanks to the team of referees for patiently reading through the manuscript. As well as identifying numerous errors, their suggestions have greatly improved the exposition.

\tableofcontents

\section{Pro-finite homotopy types}\label{profinitesn}

\begin{definition}
Let $\bS$ be the category of simplicial sets, and take $s\gpd$ to consist of  those simplicial objects in the category of groupoids whose spaces of objects are discrete (i.e. sets, rather than simplicial sets). 

Let $\Top$ denote the category of compactly generated Hausdorff topological spaces.  
\end{definition}

\begin{definition}
Given $G \in s\gpd$, we define $\pi_0G$ to be the groupoid with objects $\Ob G$, and morphisms $(\pi_0G)(x,y) = \pi_0G(x,y)$.
\end{definition}

\begin{definition}
A map $f\co X \to Y$ in  $\Top$ is said to be a \emph{weak equivalence} if it gives an isomorphism $\pi_0X \to \pi_0Y$ on path components, and for all $x \in X$, the maps $\pi_n(f)\co \pi_n(X,x) \to \pi_n(Y,fx)$ are all isomorphisms. 

We give $\bS$ the model structure of  \cite[Theorem V.7.6]{sht}; in particular, a map $f\co X \to Y$ in $\bS$ is said to be a  \emph{weak equivalence} if the map $|f|\co |X| \to |Y|$ of topological spaces is so, where $|\cdot |$ is the realisation functor of \cite[\S I.2]{sht}. Likewise, for $x \in X_0$ we write $\pi_n(X,x):= \pi_n(|X|,x)$.

A map $f\co G \to H$ in $s\gpd$ is a  \emph{weak equivalence} if the map  $\pi_0G \to \pi_0H$ is an equivalence, and for all objects $x \in \Ob G$, the maps $\pi_n(G(x,x))\to \pi_n(H(fx,fx))$ are all isomorphisms.  

For each of these categories, we define the corresponding  \emph{homotopy categories} $\Ho(\bS),\Ho(\Top),\Ho(s\gpd)$ by localising at weak equivalences.
\end{definition}

Note that there is a functor from $\Top$ to $\bS$ which sends $X$ to the simplicial set
$$
\Sing(X)_n= \Hom_{\Top}(|\Delta^n|, X).
$$
This is right adjoint to realisation, and these functors are a pair of Quillen equivalences, so become quasi-inverse on the corresponding homotopy categories. From now on, we will thus restrict our attention to simplicial sets.

\begin{definition}
 Given $G \in s\gpd$, define the category $\bS_G$ of $G$-spaces to consist of simplicial representations of $G$. Explicitly, $X \in \bS_G$ consists of $X(a) \in \bS$ for each $a \in \Ob G$, together with maps $G(a,b) \by X(b) \to X(a)$, satisfying the obvious associativity and unit axioms.
\end{definition}

\begin{definition}\label{wdef}
Recall from  \cite[ \S V.4]{sht} that for $G \in s\gpd$, the $G$-space $WG$ is defined by 
$$
(WG)_n(x)= \coprod_{y_n, \ldots, y_0 \in \Ob G} G_n(x,y_n)\by G_{n-1}(y_{n}, y_{n-1})\by \ldots\by G(y_1, y_0) 
$$
with operations
\begin{eqnarray*}
\pd_i(g_n, g_{n-1},\ldots ,g_0)&=& \left\{ \begin{matrix} (\pd_ig_n,\pd_{i-1}g_{n-1},\ldots, (\pd_0g_{n-i})g_{n-i-1}, g_{n-i-2}, \ldots, g_0) & i<n,\\ (\pd_ng_n,\pd_{n-1}g_{n-1}, \ldots, \pd_1g_1) & i=n, \end{matrix} \right.\\
\sigma_i(g_n, g_{n-1}, \ldots , g_0)&=& (\sigma_ig_n,\sigma_{i-1}g_{n-1},\ldots, \sigma_0g_{n-i}, \id ,  g_{n-i-1}, \ldots, g_0),
\end{eqnarray*}
and for $h \in G_n(z,x)$ and $(g_n, g_{n-1}, \ldots , g_0)\in (WG)(x)$,  
$$
h(g_n, g_{n-1}, \ldots , g_0)=(hg_n, g_{n-1}, \ldots , g_0).
$$
\end{definition}

Note  that $WG(x)$ is contractible for each $x \in \Ob G$. 

\begin{definition}\label{barwdef} As in \cite[Ch.V.7]{sht}, there is a  \emph{classifying space  functor} $\bar{W}\co s\gpd \to \bS$, given by $\bar{W}G= G\backslash WG$, the co-invariants of the $G$-action. This has a  left adjoint $G\co \bS \to s\gpd$, Dwyer and Kan's  \emph{loop groupoid functor} (\cite{pathgpd}), and   these form a pair of Quillen equivalences, so give  equivalences $\Ho(\bS)\sim \Ho(s\gpd)$.  The objects of $G(X)$ are $X_0$, and  for any $x,y \in X_0$, the geometric realisation $|G(X)(x,y)|$ is weakly equivalent to the  space of paths from $x$ to $y$ in $|X|$. 
These functors have the additional properties that  $\pi_0G(X)\cong \pi_f|X|$ (the fundamental groupoid),  $\pi_f(|\bar{W}G|)\cong \pi_0G$,   $\pi_n(G(X)(x,x))\cong\pi_{n+1}(|X|,x)$ and $\pi_{n+1}(|\bar{W}G|,x)\cong\pi_n(G(x,x))$. This allows us to study simplicial groupoids instead of topological spaces.
\end{definition}

\begin{definition}\label{localsysdef}
If $X \in \bS$, then a local system is just a representation of the groupoid $\pi_fX$, i.e. a functor $\pi_fX \to \Gp$  from the fundamental groupoid to the category of groups. As in \cite[\S VI.5]{sht},  homotopy groups form a local system $\pi_nX$, whose stalk at $x$  is  $\pi_n(X,x)$. 
\end{definition}

\subsection{Pro-simplicial $L$-groupoids}

\begin{definition}
Given a set $L$ of primes, we say that an $L$-\emph{group} is a finite group $G$ for which only primes in $L$ divide its order. We define an $L$-groupoid to be a groupoid $H$ for which $H(x,x)$ is an $L$-group  for all $x \in \Ob H$.
\end{definition}

\begin{definition}\label{prodef}
Given a category $\cC$, recall that the category $\pro(\cC)$ of pro-objects in $\cC$ has objects consisting of filtered inverse systems $\{A_{\alpha}\}$ in $\cC$, with 
$$
\Hom_{\pro(\cC)}(\{A_{\alpha}\}, \{B_{\beta}\})= \lim_{\substack{\lla \\ \beta}} \lim_{\substack{\lra \\ \alpha}} \Hom_{\cC}(A_{\alpha},B_{\beta}).
$$  
\end{definition}

\begin{remark}\label{discreterk}
 A discrete topological space is just a set. Given a pro-set $\{X_{\alpha}\}$, we can thus take the limit $\Lim_{\alpha}X_{\alpha}$ in the category of topological spaces. This functor gives a  faithful embedding of $\pro(\Set)$ into    topological spaces, so $\Lim_{\alpha}X_{\alpha}$ is discrete if and only if $\{X_{\alpha}\}$ lies in the essential image of $\Set \to \pro(\Set)$.  
We will thus refer to the essential image of $\Set \to \pro(\Set)$ as the \emph{discrete} objects.

In fact,  pro-sets endow a  topological structure which cannot be detected by weak equivalences, which is why shape theory  
 is modelled using the category $\pro(\bS)$, as in \cite{isaksen}.
\end{remark}

\begin{definition}
Given a groupoid $G$ and a set $L$ of primes, define 
$G^{\wedge_L}\in \pro(\gpd)$ by  requiring that $G^{\wedge_L}$ be the completion of $G$ with respect to all $L$-groupoids $H$. In other words, $G^{\wedge_L}$ is an inverse system of $L$-groupoids, with a canonical map $G \to G^{\wedge_L}$ inducing isomorphisms
\[
 \Hom(G^{\wedge_L},H) \to \Hom(G,H)
\]
for all $L$-groupoids $H$. 

In particular, $\Ob G^{\wedge_L} =\Ob G$ and $G^{\wedge_L}(x,x)$ is the pro-$L$ completion of the group $G(x,x)$ (in the sense of \cite[\S 6]{fried}).  If $L$ is the set of all primes, we write $\hat{G}:=G^{\wedge_L}$, so  $\hat{G}(x,x)$ is the pro-finite completion of $G(x,x)$ (in the sense of \cite[\S 1]{galoisienne}). 
\end{definition}

Note that $G^{\wedge_L}$ is a pro-$L$-groupoid in the sense of Definition \ref{prodef}. However, beware that a  pro-groupoid can be isomorphic to a pro-$L$-groupoid without actually being an inverse system of $L$-groupoids, since $\{\Gamma_{\alpha}\}_{\alpha \in I} \cong \{\Gamma_{\alpha}\}_{\alpha \ge \alpha_0}$ for any $\alpha_0 \in I$.

\begin{definition}
Say that  a simplicial groupoid $\Gamma$ is a \emph{simplicial  $L$-groupoid} if
$\Gamma_i$ is an $L$-groupoid for all $i$. Denote the category of such groupoids by $s\gpd^L$.
\end{definition}

\begin{definition}
Given a groupoid $\Gamma$, define a \emph{disconnected normal subgroupoid}  $K\lhd \Gamma$ to consist of subgroups $K(x) \le \Gamma(x,x)$ for all $x \in\Ob \Gamma$, with $a K(x) a^{-1} \in K(y)$ for all $a \in \Gamma(y,x)$. 
\end{definition}

Note that disconnected normal subgroupoids $K\lhd \Gamma$ are in one-to-one correspondence with isomorphism classes  of those surjections $f\co \Gamma \to H$ for which $\Ob f\co  \Ob \Gamma \to \Ob H$ is an  isomorphism. The equivalence is given by setting $H(x,y)= \Gamma(x,y)/K(y)= K(x)\backslash \Gamma(x,y)$, and conversely by setting $K(x):= \ker(f\co  \Gamma(x,x) \to H(fx,fx))$.

\begin{definition}
Given $\Gamma \in s\gpd$, define a \emph{simplicial disconnected normal subgroupoid} $K\lhd \Gamma$ to consist of  disconnected normal subgroupoids  $K_n\lhd \Gamma_n$, closed under the operations $\pd_i, \sigma_j$. 
\end{definition}

\begin{definition}
Given $\Gamma \in s\gpd$, define $\Gamma^{\wedge_L} \in  \pro( s\gpd^L)$ to be the inverse system $\{\Gamma/K\}_K$, where $K$ ranges over the poset of all simplicial disconnected normal subgroupoids $K\lhd \Gamma$ for which $\Gamma/K$ is a simplicial $L$-groupoid.

Given $\Gamma =\{\Gamma_{\alpha}\}_{\alpha} \in \pro(s\gpd)$, define  $\Gamma^{\wedge_L} \in  \pro( s\gpd^L)$ by 
$$
\Gamma^{\wedge_L}= \lim_{\substack{\lla \\ \alpha }} \Gamma_{\alpha}^{\wedge_L},
$$
where the limit is taken in $\pro( s\gpd^L)$. This corresponds to saying that $\Gamma^{\wedge_L}$ is the pro-object $\{\Gamma_{\alpha}/K_{\alpha}\}_{(\alpha, K_{\alpha})}$ indexed by pairs $(\alpha, K_{\alpha})$, for $K_{\alpha} \lhd \Gamma_{\alpha}$.
\end{definition}

\begin{lemma}
For $\Gamma \in \pro(s\gpd)$ and $A \in \pro(s\gpd^L)$, the canonical map
$$
\Hom_{\pro( s\gpd^L)}(\Gamma^{\wedge_L}, A) \to \Hom_{\pro( s\gpd)}(\Gamma, A)
$$
is an isomorphism.
\end{lemma}
\begin{proof}
By the definition of morphisms in pro-categories, it suffices to prove this when $A \in s\gpd^L$. Then $A$ is cofinite in both $\pro(s\gpd^L)$ and $\pro(s\gpd)$ (i.e. $\Hom(\Lim_{\alpha}\Gamma_{\alpha}, A) \cong \LLim_{\alpha}\Hom(\Gamma_{\alpha}, A)$ for filtered inverse systems $\{\Gamma_{\alpha}\}_{\alpha}$), so we may also assume that $\Gamma \in s\gpd$.

Now, for any morphism $f\co \Gamma \to A$, the image $H$ is a simplicial $L$-groupoid  of the form $H= \Gamma/K$, for $K\lhd \Gamma$ a disconnected normal subgroupoid. Therefore
$$
\Hom_{s\gpd}(\Gamma, A) = \LLim_K\Hom_{s\gpd}(\Gamma/K, A)= \Hom_{\pro(s\gpd)}( \Gamma^{\wedge_L},A),
$$
as required.
\end{proof}

\begin{lemma}\label{levelwiseproL}
For $\Gamma \in \pro(s\gpd)$, the pro-$L$-groupoid $(\Gamma^{\wedge_L})_n$ is just the pro-$L$ completion of $\Gamma_n$.
\end{lemma}
\begin{proof}
Given $A \in \gpd^L$, define $A^{\Delta_n}$  (not to be confused with $A^{\Delta^n}$) to be the simplicial groupoid on objects $\Ob A$ with 
$$
A^{\Delta_n}(x,y)_i:= A(x,y)^{\Delta^i_n},
$$
with $\pd_j\co  (A^{\Delta_n})_i \to (A^{\Delta_n})_{i-1}$ coming from  $\pd^j\co  \Delta^{i-1} \to \Delta^i$, and $\sigma_j$ coming from $\sigma^j\co  \Delta^{i+1} \to \Delta^i$. Then $A^{\Delta_n}$ is clearly an $L$-groupoid, and has the key property that
$$
\Hom_{s\gpd}(\Gamma, A^{\Delta_n}) \cong \Hom_{\gpd}(\Gamma_n, A)
$$
for all $\Gamma$.

Taking colimits extends this to all $\Gamma \in \pro(s\gpd)$, and then 
$$
\Hom_{\pro(s\gpd^L)}(\Gamma^{\wedge_L}, A^{\Delta_n}) \cong \Hom_{\pro(\gpd^L)}((\Gamma^{\wedge_L})_n, A),
$$
but the left-hand side is just
$$
\Hom_{\pro(s\gpd)}(\Gamma, A^{\Delta_n}) \cong \Hom_{\pro(\gpd)}(\Gamma_n, A),
$$
so $(\Gamma^{\wedge_L})_n$ is the pro-$L$ completion of $\Gamma_n$.
\end{proof}

\begin{definition}
 Given   $X =\{X_{\alpha}\} \in \pro(\bS)$,
 define the category of local systems on $X$ to be the direct limit (over $\alpha $) of the categories of local systems on $X_{\alpha}$ (in the sense of Definition \ref{localsysdef}).
\end{definition}

\begin{remark}
Our motivation for working with $\pro(\bS)$ comes from \cite[Definition 4.4]{fried}, which associates an object $X_{\et} \in \pro(\bS)$ to each  locally Noetherian simplicial scheme $X$. Finite local systems on $X_{\et}$ then correspond to finite locally constant \'etale sheaves on $X$.
\end{remark}

\begin{definition}\label{widetildedef}
Given a pro-simplicial set $X$, and a map $\pi_fX \to \Gamma$ to a pro-groupoid with discrete objects, define the \emph{covering system} $\widetilde{X}$ by
$$
\widetilde{X}(a):= X\by_{B\Gamma}B(\Gamma\da a) \in \pro(\bS)
$$
for $a \in \Ob \Gamma$, noting that this is equipped with a natural associative action $\Gamma(a,b)\by \widetilde{X}(a) \to \widetilde{X}(b)$ in $\pro(\bS)$. Here, $B$ is the nerve functor (equal to $\bar{W}$ in this context), and $\Gamma \da a$ denotes the slice category of morphisms in $\Gamma$ with target $a$.  
\end{definition}

\begin{definition}\label{CCdef}
Given $\pi_fX \to \Gamma$ as above, with a continuous representation $S$ of $\Gamma$ in pro-sets (i.e. $S(a)\in \pro(\Set)$ for $a \in \Ob \Gamma$, equipped with an associative action $\Gamma(a,b) \by S(a) \to S(b)$ of pro-sets),
define the cosimplicial set $\CC^{\bt}(X, S)$ by
$$
\CC^n(X,S):=\Hom_{\Gamma,\pro(\Set)}(\widetilde{X}_n, S).
$$
\end{definition}

From now on, local systems will be abelian unless stated otherwise.
\begin{definition}\label{prospacecohodef}
Given $X=\{X_{\alpha}\} \in \pro(\bS)$ and a local system $M$ on $X_{\beta}$ define cohomology groups by
$$
\H^*(X,M):= \lim_{\substack{\lra\\ \alpha}}\H^*(X_{\alpha},M),
$$
where $\H^*(X_{\alpha},-)$ is  cohomology with local coefficients, and we also write $M$ for the pullbacks of $M$ to $X_{\alpha}$ and to $X$. 
Given $G \in \pro(s\gpd)$, set $\H^*(G,-):= \H^*(\bar{W}G, -)$.
\end{definition}

Note that the cosimplicial complex $\CC^{\bt}(X,M)$ extends \cite[\S VI.4]{sht}  to pro-spaces, and that  $ \H^*(X,M)=\H^*(\CC^{\bt}(X,M))$, the  cohomology groups with local coefficients.

\begin{definition}\label{prospaceprocohodef}
Given   $X \in \pro(\bS)$  with $X_0$ discrete, and an inverse system $M=\{M_i\}_{i \in \N}$ of local systems on $X$, define the continuous cohomology groups $\H^*(X,M)$ as follows. First form the cosimplicial complex  $\CC^{\bt}(X,M):= \Lim\CC^{\bt}(X,M_i)$, for $\CC^{\bt}$ as in Definition \ref{CCdef}, then set  
$$
   \H^*(X,M) := \H^*(\CC^{\bt}(X,M)),
$$
noting that this agrees with Definition \ref{prospacecohodef} when $M_i=M$ for all $i$.
\end{definition}

\begin{remark}\label{mlworks}
Observe that there is a short exact sequence
$$
0 \to {\Lim}^1 \H^{n-1}(X,M_i) \to \H^n(X,M) \to \Lim \H^n(X, M_i) \to 0,
$$
so $\H^n(X,M) \cong \Lim \H^n(X, M_i)$ whenever the inverse system $\{\H^{n-1}(X,M_i) \}_i$ satisfies the Mittag--Leffler condition (for instance if the groups are finite).

When working with the \'etale homotopy type $X_{\et}$, we will usually apply this construction to $\Zl$-local systems $\{M_i= M/\ell^i\}_i$. In that case, the exact sequence above becomes the comparison between \'etale cohomology and Jannsen's continuous \'etale cohomology (see Example \ref{ethtpy} for details).
\end{remark}

\begin{lemma}\label{ctscohodiscrete}
Given $X \in \bS$ and an inverse system $M=\{M_i\}_{i \in \N}$ of local systems on $X$, there is an isomorphism 
$$
\H^*(X, \Lim M_i) \cong \H^*(X,M).
$$ 
\end{lemma}
\begin{proof}
As in Definition \ref{prospacecohodef}, $\H^*(X, \Lim M_i)$ is cohomology of the complex $\Lim \CC^{\bt}(X,M_i)=\CC^{\bt}(X,\Lim M_i)$, but
$$
\CC^{n}(X,\Lim M_i)=\Hom_{\Set}(X_n, \Lim M_i)= \Lim \Hom_{\Set}(X_n, M_i)=\CC^{n}(X,M),
$$ 
as required.
\end{proof}

We will occasionally refer  to groups and groupoids as ``discrete'', to distinguish them from topological (or simplicial) groups and groupoids. As in Remark \ref{discreterk}, we  regard a pro-groupoid as a kind of  topological groupoid, so ``discrete'' will indicate that both  simplicial and pro structures are trivial.

\begin{definition}
Given a set $L$ of primes, say that a pro-groupoid $G$ with discrete object set is $(L,n)$-\emph{good} if for all  $G^{\wedge_L}$-representations $M$ in abelian $L$-groups, the canonical map
$$
\phi_M:\H^i(G^{\wedge_L}, M) \to \H^i(G,M)
$$
is an isomorphism for all $i \le n$ and an inclusion for $i=n+1$. When $L$ is the set of all primes, we say that $G$ is $n$-good. Observe that any inverse system of $(L,n)$-good groupoids is $(L,n)$-good. Say that $G$ is $L$-\emph{good} if it is $(L,n)$-good for all $n$.
\end{definition}

\begin{lemma}\label{freegood}
Free groups are $L$-good for all $L$.
\end{lemma}
\begin{proof}
Let $F=F(X)$ be a free group generated by a  set $X$, and let $\Gamma:=F^{\wedge_L}$. By the argument of \cite[I\S 2.6 Ex. 1(a)]{galoisienne}, it suffices to show that $\H^*(\Gamma,M) \to \H^*(F,M)$ is surjective for all  discrete $\Gamma$-representations $M$ in abelian $L$-groups. Since $F$ is free, $\H^n(F,M)=0$ for $n> 1$, so it only remains to establish surjectivity for $n=1$.  

This amounts to showing  that every derivation $\alpha\co F \to M$ factors through $\Gamma$. The derivation gives rise to a map $\beta\co F \to M\rtimes G$, for some finite $L$-torsion quotient $G$ of $F$. Since $M\rtimes G$ is an $L$-group, $\beta$ factors through $\Gamma$.
\end{proof}

\begin{examples}\label{Lgoodegs}
\begin{enumerate}
\item $L$-groups are $L$-good.
\item\label{ffour} If $1 \to F \to \Gamma \to \Pi \to 1$ is an exact sequence of groups, with $F$ and $\Pi$ $L$-good, $F^{\wedge_L} \to \Gamma^{\wedge_L}$ injective, and $\H^a(F,M)$ finite for all finite $L$-torsion $\Gamma$-modules, then $\Gamma$ is $L$-good.

\item\label{ttwo} All finitely generated nilpotent groups are $L$-good for all $L$.

\item\label{tthree} The fundamental group of a compact Riemann surface is $L$-good for all $L$.
 \end{enumerate}
\end{examples}
\begin{proof}
\begin{enumerate}
\item[(\ref{ffour})] This is essentially \cite[I\S 2.6 Ex. 2(c)]{galoisienne}.

\item[(\ref{ttwo})] Express $\Gamma$ as a successive extension of finite groups and $\Z$, then apply (\ref{ffour}). 

\item[(\ref{tthree})]  Choose a smooth complex  projective curve $C$ of genus $g>0$, with $\pi_1(C)=\Gamma$. It suffices to show that for all finite $L$-torsion $\Gamma^{\wedge_L}$-representations $M$, the map
$$
\H^*(\Gamma^{\wedge_L}, M) \to \H^*_{\et}(C,M)
$$ 
is an isomorphism. 

Letting $\tilde{C}$ be the universal \'etale pro-$L$ cover of $C$, this is equivalent (by the Serre spectral sequence) to showing that $\H^*_{\et}(\tilde{C}, \bF_p)= \bF_p$ for all $p \in L$. 
$\tilde{C}$ is the inverse limit all finite $L$-covers $C' \to C$, giving
$$
\H^i_{\et}(\tilde{C}, \bF_p)= \lim_{\substack{\lla \\ C' }} \H^i_{\et}(C', \bF_p),
$$
which can only be non-zero for $i=0,1,2$. 

Note that $\pi_1(\tilde{C})= \ker(\hat{\Gamma} \to \Gamma^{\wedge_L})$. Thus the pro-$L$ completion $\pi_1(\tilde{C})^{\ab, L}$ of the abelianisation of $\pi_1(\tilde{C})$ must be $0$, or we would have a larger pro-$L$ quotient of $\hat{\Gamma}$ than $\Gamma^{\wedge_L}$. Hence  $\H^1_{\et}(\tilde{C}, \bF_p)= 0$ for all $p \in L$.

 We now adapt the proof of \cite[Proposition 15]{alexschmidt}. Since any curve $C'$ has a cover $C''$ of degree $p$, with the map $\H^2_{\et}(C', \bF_p)\to \H^2_{\et}(C'', \bF_p)$ thus being $0$, we deduce that $\H^2_{\et}(\tilde{C}, \bF_p)= 0$, which completes the proof.
\end{enumerate}
\end{proof}

\begin{proposition}\label{cohochar}
For any $X \in \bS$, the canonical morphism
$$
X \to \bar{W}(G(X)^{\wedge_L})
$$
in $\pro(\bS)$ induces an isomorphism $(\pi_fX)^{\wedge_L} \to \pi_f\bar{W}(G(X)^{\wedge_L})$ of pro-groupoids, and  has the property that for all finite abelian  $(\pi_fX)^{\wedge_L} $-representations $M$  in $L$-groups, the canonical map
$$
\H^*(\bar{W}(G(X)^{\wedge_L}), M)\to \H^*(X,M)
$$
is an isomorphism.
\end{proposition}
\begin{proof}
The statement about fundamental groupoids is immediate, since completion commutes with taking quotients. Now, observe that
$$
\H^n(\bar{W}(G(X)^{\wedge_L}), M)\cong \H^n(G(X)^{\wedge_L},M),
$$
tautologically from Definition \ref{prospacecohodef}.

 It thus suffices to show that the simplicial groupoid $G(X)$ is $L$-good, in the sense that $\H^*(G(X),M) \cong  \H^*(G(X)^{\wedge_L},M)$ for all $\pi_0G(X)^{\wedge_L}$-representations in abelian $L$-groups $M$.
This is equivalent to showing that for all $x \in X_0$, the simplicial groups $G(X)(x,x)$ are $L$-good. This  will follow if the groups $G_n(x,x)$ are all $L$-good, because there is a spectral sequence
$$
\H^q(G_p,M) \abuts \H^{p+q}(G, M). 
$$
Since the groups $G_n(x,x)$ are all free, this then follows from Lemma \ref{freegood}.
\end{proof}

Given a property $P$ of groups, we will say that a groupoid $\Gamma$  locally satisfies $P$ if the groups $\Gamma(x,x)$ satisfy $P$, for all $x \in \Ob \Gamma$.

\begin{definition}
Define $\pro(\bS)_{\delta}$ to be the full subcategory of $\pro(\bS)$ consisting of pro-spaces $X$ for which  $X_0$ is discrete (as in Remark \ref{discreterk}, so $X_0$ is  a set, not just a pro-set).

Define $\bS^{\wedge_L}$ to be the full subcategory of $\pro(\bS)_{\delta} $ consisting of spaces $X$ for which  the groups $\pi_n(X,x)$ are all  pro-$L$-groups. If $L$ is the set of all primes, we write $\hat{\bS}:= \bS^{\wedge_L}$. 
\end{definition}

\begin{definition}
 A morphism $f\co X \to Y$ in $\pro(\bS)_{\delta}$ is said to be an Artin--Mazur weak equivalence if $\pi_0X \to \pi_0Y$ is an isomorphism, and the maps  $\pi_n(X,x) \to \pi_n(Y,fx)$ are pro-isomorphisms  for all $n\ge 1$ and all $x \in X_0$.

Define $\Ho(\pro(\bS)_{\delta})$ and $\Ho(\bS^{\wedge_L})$ by formally inverting all Artin--Mazur weak equivalences.
\end{definition}

In \cite{isaksen}, Isaksen established a model structure on $\pro(\bS)$ with the right properties for modelling pro-homotopy types. In particular, \cite[Corollary 7.5]{isaksen}  shows  that a morphism in $\pro(\bS)_{\delta}$ is a weak equivalence in $\pro(\bS)$ if and only if it is an Artin--Mazur weak equivalence. 

\begin{proposition}\label{Lcohoweak}
Fix $N \in [1, \infty]$, and let $f\co X \to Y$ be a morphism in $\pro(\bS)_{\delta}$ such that $(\pi_fX)^{\wedge_L} \to (\pi_fY)^{\wedge_L}$ is a pro-equivalence of pro-groupoids, with the property that for all  abelian  $(\pi_fY)^{\wedge_L}$-representations $M$  in $L$-groups, the map
$$
\H^n(f)\co  \H^n(Y, M)\to \H^n(X,M)
$$
is an isomorphism for all $n \le N$ and injective for $n=N+1$. Then for all $Z \in \bS^{\wedge_L}$ with $\pi_iZ=0$ for $i>N$ (resp. $i >N+1$), the map 
$$
f^*\co \Hom_{\Ho(\pro(\bS)_{\delta}) }(Y,Z)\to \Hom_{ \Ho(\pro(\bS)_{\delta}) }(X,Z)
$$
is an isomorphism (resp. an inclusion). 
\end{proposition}
\begin{proof}
First observe that if $M$ is a $\pi_f(Y)^{\wedge_L}$-representation in abelian pro-$L$ groups, we can express it as an inverse system
$\{M_{\alpha}\}$   of $\pi_f(Y)$-representations in $L$-groups. Then the  complex $\CC^{\bt}(Y, M)$ of $M$-cochains is given by
\[
 \CC^{\bt}(Y, M) \simeq \oR\Lim_{\alpha} \CC^{\bt}(Y, M_{\alpha}).
\]
This implies that for all such $M$, the map $\H^n(f)\co  \H^n(Y, M)\to \H^n(X,M)$ is an isomorphism  for  all $n\le N$, and injective for $n=N+1$.

Now consider the Moore--Postnikov tower (\cite[Definition VI.3.4]{sht}) $P_nZ$ of a fibrant replacement for $Z$. The pro-equivalence on $\pi_f$ gives the required isomorphism if $Z=P_1Z$, and we can proceed by induction.

Assume that we have a homotopy class of maps $X \to P_nZ$, for $n<N$. The obstruction to lifting this to a homotopy class of maps $X \to P_{n+1}Z$ lies in $\H^{n+2}(X, \pi_{n+1}Z)$, and if non-empty,  the latter homotopy class is a principal $\H^{n+1}(X, \pi_{n+1}Z)$-space. 
As $\pi_{n+1}Z$ is a  pro-$L$-group,  the  isomorphism $\H^{n+1}(Y,-)\cong \H^{n+1}(X,-)$ and the inclusion $\H^{n+2}(Y,-)\into \H^{n+2}(X,-) $ (resp. the inclusion $\H^{n+1}(Y,-)\into \H^{n+1}(X,-) $)
mean that the pro-homotopy class of lifts $Y\to P_{n+1}Z$ is similarly determined (resp. embeds into the class of lifts $X \to P_{n+1}Z$), completing the inductive step.

Since the map $Z \to P_NZ$ (resp. $Z \to P_{N+1}Z$) is an Artin--Mazur weak equivalence, this completes the proof for $N< \infty$.
In the case $N= \infty$, the analysis above gives an isomorphism 
$$
f^*\co \Hom_{\Ho(\pro(\bS)_{\delta}) }(Y, \Lim_n P_nZ)\to \Hom_{ \Ho(\pro(\bS)_{\delta}) }(X,\Lim P_nZ);
$$
since the canonical map $Z \to \Lim_n P_nZ$ is an Artin--Mazur weak equivalence, this completes the proof.
\end{proof}

\begin{corollary}\label{Lloc}
The inclusion functor $\bS^{\wedge_L}\to \pro(\bS)_{\delta}$ has a homotopy left adjoint, which we denote by $X \leadsto X^{\wedge_L}$. This has the property  that for $X \in \bS^{\wedge_L}$, $X^{\wedge_L}\simeq X$.
\end{corollary}
\begin{proof}
Propositions \ref{cohochar} and \ref{Lcohoweak} imply that
for $X \in \bS$, the object $ X^{\wedge_L}:=\bar{W}(G(X)^{\wedge_L}) \in \bS^{\wedge_L}$ has the required properties. Given an inverse system $X=\{X_{\alpha}\}$, set $X^{\wedge_L}:= \Lim (X_{\alpha})^{\wedge_L}$.
\end{proof}

\begin{remarks}
Comparing  with \cite[Theorem 6.4 and Corollary 6.5]{fried}, we see that this gives a  generalisation of Artin and Mazur's pro-$L$ homotopy type (\cite{arma}) to unpointed spaces. Their context for pro-homotopy theory was formulated slightly differently, in terms of $\pro(\Ho(\bS))$, which is not very well-behaved. See \cite{isaksen} for details of the comparison. 

Since this paper was first written, an alternative pro-finite completion functor has been developed in \cite{quick}. However, the category of  pro-finite homotopy types in \cite{quick} is larger than ours, because for its pro-spaces $X$, the  pro-set $\pi_0X$ is pro-finite rather than discrete.  The   pro-finite completion functor thus differs from ours in that 
it also takes the pro-finite completion of the set $\pi_0X$. 

An important feature of \cite{quick} is the existence of a model structure for pro-finite spaces, and this raises the question of whether there is a model structure on   $\pro(s\gpd^L)$,
 and how the respective model structures compare. The most likely solution is that there is a fibrantly cogenerated model structure on $\pro(s\gpd^L_{\cF})$, where $s\gpd^L_{\cF}$ is the full subcategory of $s\gpd^L $ consisting of simplicial groupoids with finite object set. For this model structure, the cogenerating fibrations should be morphisms in $s\gpd^L_{\cF}$ which are fibrations in $s\Gpd$, possibly with some additional Artinian condition analogous to \cite[Theorem \ref{ddt1-scspmodel}]{ddt1}. The right adjoint $\pro(s\gpd^L_{\cF})\to \pro(s\gpd^L) $ should then induce a fibrantly cogenerated structure on the latter, while the functor $\bar{W}$ from $\pro(s\gpd^L_{\cF})$ to simplicial pro-finite sets  should be a right Quillen equivalence when $L$ is the set of all primes. 
\end{remarks}

\subsection{Comparing homotopy groups}

We now investigate when we can describe the homotopy groups of $X^{\wedge_L}$ in terms of the homotopy groups of $X$.

\begin{lemma}\label{kpn}
If $A$ is a finitely generated abelian group, then for $n\ge 2$, completion of the Eilenberg-Maclane space is given by $K(A,n)^{\wedge_L}= K(A^{\wedge_L},n)$.
\end{lemma}
\begin{proof}
By Proposition \ref{Lcohoweak}, we need to show that the maps 
$$
\H^*(K(A^{\wedge_L},n),M) \to \H^*(K(A,n),M)
$$
are isomorphisms for all abelian $L$-groups $M$. By considering the spectral sequence associated to a filtration, it suffices to consider only the cases $M=\bF_p$, for $p \in L$.

If $A= A'\by A''$, then $K(A,n)=K(A',n)\by K(A'',n)$, so $\H^*(K(A,n), \bF_p)=\H^*(K(A',n),\bF_p)\ten\H^*( K(A'',n),\bF_p)$. The structure theorem for finitely generated abelian groups therefore allows us to assume that $A=\Z/q$, for $q$ a prime power or $0$. 

Now, if $q$ is neither zero nor a power of $p$, then $\H^r(K(A,n),\bF_p)=0$ for $r>0$; since $A^{\wedge_L}$ is a quotient of $A$, we also get $\H^r(K(A^{\wedge_L},n),\bF_p)=0$. If $q=p^s$, then $A^{\wedge_L}=A$, making isomorphism automatic. 

If $q=0$, then $A=\Z, A^{\wedge_L}=\prod_{\ell \in L} \Zl$, and  $\H^r(K(\Zl,n),\bF_p)=0$ for $r>0$ and $\ell \ne p$. We need to show that 
$$
\H^*(K(\Z_p,n),\bF_p) \to \H^*(K(\Z,n),\bF_p)
$$
is an isomorphism, or equivalently that $K(\Z,n)^{ \wedge_p}=K(\Z_p,n)$. This follows from \cite[Theorem 1.5]{Qpf}.
\end{proof}

\begin{proposition}\label{Lcohoweak2}
Take a morphism
$f\co X \to Y$  in $\pro(\bS)_{\delta}$ such that $(\pi_fX)^{\wedge_L} \to (\pi_fY)^{\wedge_L}$ is a pro-equivalence of pro-groupoids. Then the following are equivalent:
\begin{enumerate}
 \item For all  abelian  $(\pi_fY)^{\wedge_L}$-representations $M$  in $L$-groups, the map
$$
\H^n(f)\co  \H^n(Y, M)\to \H^n(X,M)
$$
is an isomorphism for all $n \le N$ and injective for $n=N+1$.

\item The map
\[
 \pi_n(f): \pi_n(X^{\wedge_L},x) \to \pi_n(Y^{\wedge_L},fy) 
\]
is a  pro-isomorphism for $n \le N$ and a pro-surjection for $n=N+1$.
\end{enumerate}

In particular, a pro-groupoid $G$ with discrete object set is $(L,N)$-good if and only if
\[
 \pi_n( (BG)^{\wedge_L})=0
\]
for all $2 \le n \le N$.
\end{proposition}
\begin{proof}
 The key observation is that $\Hom_{\Ho(\pro(\bS)_{\delta}) }(Y,P_nZ) \cong \Hom_{\Ho(\pro(\bS)_{\delta}) }(P_nY,P_nZ) $, which is deduced from the corresponding result for $\bS$. Thus   Proposition \ref{Lcohoweak} implies that
\[
 P_N(X^{\wedge_L}) \to P_N(Y^{\wedge_L})
\]
becomes an isomorphism in $\Ho(\pro(\bS)_{\delta})$, while 
\[
 P_{N+1}(X^{\wedge_L}) \to P_{N+1}(Y^{\wedge_L})
\]
is an epimorphism. Since isomorphisms in $\Ho(\pro(\bS)_{\delta})$ are just Artin--Mazur weak equivalences, this completes the ``only if'' part.

For the converse, note that the hypothesis is equivalent to saying that the homotopy fibre $F$ of $f^{\wedge_L}\co X^{\wedge_L}\to Y^{\wedge_L}$ is $N$-connected, by looking at the long exact sequence of homotopy groups. Thus $\H^j(F, A)=0$ for all $0<j\le N$ and all abelian $L$-groups $A$. For any $\pi_fY^{\wedge_L}$-representation $M$ in abelian $L$-groups, the Leray spectral sequence 
\[
 \H^i(Y^{\wedge_L}, \H^j(F,M)) \abuts \H^{i+j}(X^{\wedge_L},f^{-1}M)
\]
 forces the maps  $\H^i(Y^{\wedge_L},M) \to \H^{i+j}(X^{\wedge_L},M) $ to be isomorphisms for $i\le N$ and injective for $i=N+1$, as required.

The final statement is given by taking $X=BG$ and $Y= B(G^{\wedge_L})$.
\end{proof}

\begin{lemma}\label{piIMlemma}
 If $f\co X \to Y$ is a morphism in $\pro(\bS)_{\delta}$ for which the map
\[
 \pi_n(f): \pi_n(X,x) \to \pi_n(Y,fy) 
\]
is a  pro-isomorphism for $n \le N$ and a pro-surjection for $n=N+1$, then the map
\[
 \pi_n(f): \pi_n(X^{\wedge_L},x) \to \pi_n(Y^{\wedge_L},fy) 
\]
is a  pro-isomorphism for $n \le N$ and a pro-surjection for $n=N+1$.
\end{lemma}
\begin{proof}
The proof of Proposition \ref{Lcohoweak2} adapts to show that for any $\pi_fY$-representation $M$,  the maps  $\H^i(Y,M) \to \H^{i}(X,M) $ are isomorphisms for $i\le N$ and injective for $i=N+1$. Thus the hypotheses of Proposition \ref{Lcohoweak2} are satisfied, giving the required results.
\end{proof}

\begin{definition}\label{semidirect}
Given a group-valued representation $H$ of a groupoid $\Gamma$ (i.e. a functor from $\Gamma$ to the category of groups), recall from \cite[Definition \ref{htpy-semidirect}]{htpy}  that the semi-direct product $H\rtimes \Gamma$ is a groupoid with objects $\Ob(H\rtimes \Gamma)=\Ob(\Gamma)$ and has $(H\rtimes \Gamma)(x,y)= H_x \rtimes \Gamma(x,y)$.
\end{definition}

\begin{proposition}\label{piclasstohat}
Fix $X \in \bS$. If  $\pi_n(X,x)$  is  finitely generated for all $n \le N$, 
and  if the image of $\pi_1(X,x)\to \Aut( \pi_n(X,x)\ten \bF_p)$ is $L$-torsion for all $n\le N$,  all $p \in L$, and all $x \in X$, then 
there is an exact sequence
$$
\xymatrix{&&\pi_{N+1}(X^{\wedge_L},x) \ar[r] &\pi_{N+1}((B\pi_1(X,x))^{\wedge_L})\ar[dll] &\\
 &\pi_N(X,x)^{\wedge_L}\ar[r] &\pi_N(X^{\wedge_L},x) \ar[r] &\pi_N((B\pi_1(X,x))^{\wedge_L})\ar[r] &\ldots\\
\ldots \ar[r] &\pi_2(X,x)^{\wedge_L}\ar[r] &\pi_2(X^{\wedge_L},x) \ar[r] &\pi_2((B\pi_1(X,x))^{\wedge_L})\ar[r] & 0.
}
$$

Hence if in addition $\pi_fX$ is $(L,N+1)$-good (resp. $(L,N)$-good),  then the natural map
$$
\pi_n(X)^{\wedge_L}\to \pi_n(X^{\wedge_L})
$$
is a pro-isomorphism for all $n\le N$ (resp. a pro-isomorphism for all $n< N$ and a pro-surjection for $n=N$).
\end{proposition}
\begin{proof}
We adapt the argument of \cite[Theorem \ref{htpy-classicalpi}]{htpy}. Let $\{X(n)\}_n$ be the Postnikov tower for $X$. We will prove the proposition inductively for the groups $X(n)$. Thanks to Lemma \ref{piIMlemma}, we may replace $X$ with $X(N)$, so may assume that the groups $\pi_n(X,x)$ are  finitely generated for all $n$.
Write $\Gamma:=\pi_fX$.

For $n=1$, $X(1)$ is weakly equivalent to $B\pi_fX$, so $(B\pi_fX)^{\wedge_L}\simeq X(1)^{\wedge_L}$ and $\pi_n(X(1),x)=0$ for all $n \ge 2$, making the  exact sequence above immediate.

Now assume that $X(n-1)$ satisfies the inductive hypothesis, and consider the fibration $X(n) \to X(n-1)$. This is determined up to homotopy  by a k-invariant (\cite[\S VI.5]{sht}) $\kappa \in \H^{n+1}(X(n-1), \pi_{n}(X))$. Since $\pi_{n}(X)\ten \bF_p$ is a finite-dimensional $\Gamma^{\wedge_L}$-representation for all $p \in L$, the group $A:=\pi_{n}(X)^{\wedge_L}$ is an inverse limit of finite $\Gamma^{\wedge_L}$-representations.
Now, the element 
$$
\kappa \in \H^{n+1}(X(n-1), A) \cong \H^{n+1}(X(n-1)^{\wedge_L}, A)
$$
comes from a map
$$
G(X(n-1))^{\wedge_L}\to (N^{-1}A[-n])\rtimes \Gamma,
$$
where $N^{-1}$ denotes the denormalisation functor (\cite[8.4.4]{W}) from chain complexes to simplicial complexes (the Dold-Kan correspondence).

Let $LA$ be the chain complex with $A$ concentrated in degrees $n,n-1$, and $d\co (LA)_n \to (LA)_{n-1}$ the identity, and define $\cG$ to be the pullback of this map along the surjection
$ N^{-1}LA \rtimes \Gamma \to(N^{-1}A[-n])\rtimes \Gamma$ of simplicial locally pro-finite $L$-torsion groupoids. This gives an extension 
$$
N^{-1}A[1-n]  \to \cG \to  G(X(n-1))^{\wedge_L}.
$$

Applying $\bar{W}$ gives the fibration
$$
\bar{W}N^{-1}A[1-n] \to \bar{W}\cG \to X(n-1)^{\wedge_L}
$$
in $\pro(\bS)$, corresponding to the k-invariant $f^*\kappa \in \H^n(X(n-1)^{\wedge_L},A)$ for $f\co X(n-1) \to X(n-1)^{\wedge_L}$. This in turn gives a map $X(n) \to \bar{W}\cG$, compatible with the fibrations. 

The long exact sequence of homotopy applied to the map $\bar{W}\cG\to X(n-1)^{\wedge_L}$ shows that $\pi_m(\bar{W}\cG,x)= \pi_m( X(n-1)^{\wedge_L})$ for all $m \ne n, n+1$, and gives an exact sequence
$$
0 \to \pi_{n+1}(\bar{W}\cG,x)  \to \pi_{n+1}( X(n-1)^{\wedge_L})\to  A(x) \to \pi_{n}(\bar{W}\cG,x)  \to \pi_{n}( X(n-1)^{\wedge_L})\to 0.
$$ 
The inductive hypothesis shows that $\pi_m( X(n-1)^{\wedge_L})= \pi_m((B\pi_1(X,x))^{\wedge_L})$ for $m \ge n+1$, so we deduce that there is a long exact sequence
$$
\xymatrix{
\ldots \ar[r] &\pi_m(X(n),x)^{\wedge_L}\ar[r] &\pi_m(\bar{W}\cG,x) \ar[r] &\pi_m((B\pi_1(X,x))^{\wedge_L})\ar[r] &\ldots\\
\ldots \ar[r] &\pi_2(X(n),x)^{\wedge_L}\ar[r] &\pi_2(\bar{W}\cG,x) \ar[r] &\pi_2((B\pi_1(X,x))^{\wedge_L})\ar[r] & 0.
}
$$

As $\bar{W}\cG \in \bS^{\wedge_L}$, it will therefore suffice to show that $F\co G(X(n))^{\wedge_L} \to \cG$ is a weak equivalence. We now apply the Hochschild--Serre spectral sequence, giving
$$
\H^p(X(n-1), \H^q( N^{-1}A[1-n],M))= \H^p(G(X(n-1))^{\wedge_L},\H^q(N^{-1}A[1-n],M))\abuts \H^{p+q}(\cG,M).
$$
Similarly
$$
\H^p(X(n-1),\H^q(E(n),V))\abuts \H^{p+q}(X(n),V),
$$
for all  $\Gamma^{\wedge_L}$-representations $M$ in abelian $L$-groups, where $E(n)$ is the fibre of $X(n) \to X(n-1)$.

Now, $E(n)$ is a $K(\pi_n(X), n)$-space, and $\bar{W}N^{-1}A[1-n]$ is a $K(A,n)$-space. By Lemma \ref{kpn}, it follows that $E(n) \to \bar{W}N^{-1}A[1-n]$ is pro-$L$ completion, giving an isomorphism of cohomology with coefficients in $M$. Thus $F$ induces isomorphisms on homology groups, hence must be a weak equivalence by Proposition \ref{Lcohoweak}.

Finally, if $\Gamma$ is $(L,m)$-good,  Corollary \ref{Lcohoweak2} shows that $\pi_n((B\Gamma)^{\wedge_L},x)=0$ for all $1< n\le m$.
\end{proof}

\section{Review of pro-algebraic homotopy types}\label{review}
Here we give a summary of the results from \cite{htpy} and \cite{mhs}.  The motivation for these is that they provide a framework to transfer information about local systems and their cohomology to statements about homotopy types. 
Fix a field $k$ of characteristic zero.

\subsection{Pro-algebraic groupoids}
Given a local system $\vv$ of finite-dimensional $k$-vector spaces on a topological space $X$, we can form the affine $k$-scheme $\Iso(\vv_x, \vv_y)$ of isomorphisms of stalks, for each pair of points $x,y \in X$. These combine to form a kind of groupoid $G$  whose  objects  are the points of $X$.
This is the motivating example of a pro-algebraic groupoid; in this case it comes equipped with a canonical groupoid homomorphism  $\pi_fX \to G(k)$.

For the general case, we now recall some definitions from \cite[\S\S \ref{htpy-gpdsn}--\ref{htpy-levisn}]{htpy}.
\begin{definition}\label{alggpddef}
Define a \emph{pro-algebraic groupoid} $G$ over a field $k$ to consist of the following data:
\begin{enumerate}
\item A discrete set $\Ob(G)$.
\item For all $x,y \in \Ob(G)$, an affine scheme $G(x,y)$ (possibly empty) over $k$.
\item A groupoid structure on $G$, consisting of a multiplication  morphism $m\co G(x,y)\by G(y,z) \to G(x,z)$, identities $\Spec k \to G(x,x)$ and inverses $G(x,y) \to G(y,x)$, satisfying associativity, identity and inverse axioms.
 \end{enumerate}
Note that a pro-algebraic group is just a pro-algebraic groupoid on one object.
We say that a pro-algebraic groupoid is reductive (resp. pro-unipotent) if the pro-algebraic groups $G(x,x)$ are so for all $x \in \Ob(G)$. An algebraic groupoid is a pro-algebraic groupoid for which the $G(x,y)$ are all of finite type.
\end{definition}
If $G$ is a pro-algebraic groupoid, let $O(G(x,y))$ denote the global sections of the structure sheaf of $G(x,y)$.

\begin{remark}
The terminology ``pro-algebraic groupoid'' follows the characterisation of pro-algebraic groups in \cite[Ch. II]{tannaka}. A linear algebraic group is an affine group scheme of finite type, and there is an equivalence of categories between affine group schemes and pro-objects in linear algebraic groups. A more accurate term for  pro-algebraic groupoids would thus be ``linear pro-algebraically enriched groupoids''.
\end{remark}

\begin{definition}
Given morphisms $f,g\co G \to H$ of pro-algebraic groupoids, define a \emph{natural isomorphism} $\eta$ between $f$ and $g$ to consist of morphisms
$$
\eta_x\co  \Spec k \to H(f(x),g(x))
$$
for all $ x\in \Ob(G)$, such that the following diagram commutes, for all $x,y \in \Ob(G)$:
$$
\begin{CD}
G(x,y) @>f(x,y)>> H(f(x),f(y))\\
@Vg(x,y)VV  @VV{\cdot\eta_y}V \\
 H(g(x),g(y)) @>{\eta_x\cdot}>>  H(f(x),g(y)).
\end{CD}
$$
[If we reversed our order of composition in Definition \ref{alggpddef}, this would be the same as a natural transformation of functors of categories enriched in affine $k$-schemes.]
 
A morphism $f\co G \to H$ of pro-algebraic groupoids is said to be an equivalence if there exists a morphism $g\co H \to G$ such that $fg$ and $gf$ are both naturally isomorphic to identity morphisms. This is the same as saying that for all $y \in \Ob(H)$, there exists $x \in \Ob(G)$ such that $H(f(x),y)(k)$ is non-empty (essential surjectivity), and that for all $x_1,x_2 \in \Ob(G)$, $G(x_1,x_2) \to H(f(x_1),f(x_2) )$ is an isomorphism.
\end{definition}

\begin{definition}\label{gpdrep}
Given a pro-algebraic groupoid $G$, define a \emph{finite-dimensional linear} $G$-\emph{representation} to be a functor $\rho$ from $G$ to the category of finite-dimensional $k$-vector spaces,  respecting the algebraic structure. Explicitly, this  consists of a set $\{V_x\}_{x \in \Ob(G)}$ of finite-dimensional $k$-vector spaces, together with morphisms $\rho_{xy}\co G(x,y) \to \Hom(V_y,V_x)$ of affine schemes, respecting the multiplication and identities. 

A morphism $f\co (V,\rho)\to (W,\varrho)$ of $G$-representations consists of $f_x \in \Hom(V_x,W_x)$ such that 
$$
f_x\circ\varrho_{xy}=\rho_{xy}\circ f_y\co G(x,y) \to \Hom(V_x,W_y).
$$
\end{definition}

\begin{definition}
Given a pro-algebraic groupoid $G$, define the \emph{reductive quotient} $G^{\red}$ of $G$ by setting $\Ob(G^{\red})=\Ob(G)$, and
$$
G^{\red}(x,y)=G(x,y)/\Ru(G(y,y))= \Ru(G(x,x))\backslash G(x,y),
$$
where $\Ru(G(x,x))$ is the pro-unipotent radical of the pro-algebraic group $G(x,x)$.
The equality arises since if $f\in G(x,y),\, g \in \Ru(G(y,y))$, then $f gf^{-1}\in \Ru(G(x,x))$, so both equivalence relations are the same. Multiplication and inversion descend similarly. Observe that $G^{\red}$ is then a reductive pro-algebraic groupoid. Representations of $G^{\red}$ correspond to semisimple representations of $G$, since $k$ is of characteristic $0$.
\end{definition}

\begin{definition}\label{tann}
Recall from \cite[Definition  II.1.7]{tannaka} that a tensor category $\cC$ is said to be \emph{rigid} if it has an internal $\Hom$-functor $\cHom$, satisfying 
\begin{itemize}
\item $\cHom(X,Y)\ten \cHom(X',Y') \cong \cHom(X\ten X', Y\ten Y')$ and 

\item $(X^{\vee})^{\vee}\cong X$ for all $X \in \cC$,
\end{itemize}
where $X^{\vee}= \cHom(X,1)$, with $1$ the unit for $\ten$.
\end{definition}


\begin{definition}\label{tannaka}
Recall from  \cite[\S II.2]{tannaka}  that a \emph{neutral Tannakian category} over $k$ is a $k$-linear rigid abelian tensor category $\cC$, equipped with a faithful exact tensor functor $\omega$ (the fibre functor) from $\cC$ to the category of finite-dimensional $k$-vector spaces.

In \cite[\S \ref{htpy-gpdsn}]{htpy}, this was extended to \emph{multifibred Tannakian categories}, which have several exact tensor functors $\{\omega_x\}_{x \in S}$, jointly faithful in the sense that $\Hom(U, V) \into \prod_{x \in S}\Hom(\omega_xU, \omega_xV)$.

A \emph{Tannakian subcategory} $\cD \subset \cC$ is a full subcategory closed under the formation of subquotients, direct sums, tensor products, and
duals.
\end{definition}

Tannakian duality (\cite[Theorem II.2.11]{tannaka}) then states that for any neutral Tannakian category $(\cC, \omega)$ over a field, there is a canonical equivalence between $\cC$ and the category of finite-dimensional representations of a unique affine group scheme $G$. Explicitly, $G$ is the scheme of tensor automorphisms of $\omega$. 

If $\cC$ is multifibred, with a  set $S$ of  fibre functors, we form a pro-algebraic groupoid $G$ on objects $S$ by setting $G(x,y)$ to be the affine scheme of tensor isomorphisms from $\omega_x$ to $\omega_y$. This gives a canonical equivalence between $\cC$ and the category of finite-dimensional $G$-representations, with $\omega_x$ being pullback along the inclusion $\{x\} \into G$.

\begin{definition}
Let $\agpd$ denote the category of pro-algebraic groupoids over $k$, and observe that this category contains all  limits. 
\end{definition}

\begin{lemma}
Consider the functor $G \mapsto G(k)$ from $\agpd$ to $\gpd$, the category of abstract groupoids. This has a left adjoint,  the \emph{algebraisation functor}, denoted $\Gamma \mapsto \Gamma^{\alg}$, which is determined by  the finite-dimensional linear representations of $\Gamma$.
\end{lemma}
\begin{proof}
 The algebraisation functor can be given explicitly by setting $\Ob(\Gamma)^{\alg}=\Ob(\Gamma)$,  and 
$$
\Gamma^{\alg}(x,y)=\Gamma(x,x)^{\alg}\by^{\Gamma(x,x)}\Gamma(x,y),
$$
where $\Gamma(x,x)^{\alg}$ is the pro-algebraic (or Hochschild--Mostow) completion of the group $\Gamma(x,x)$ (\cite{Levi}), and $X\by^GY$ is the quotient of $X\by Y$ by the relation $(gx,y) \sim (x,gy)$ for $g \in G$.

Alternatively, 
the finite-dimensional linear representations of $\Gamma$ (as in Definition \ref{gpdrep}) correspond to those of $\Gamma^{\alg}$ (if the latter exists). These form a multifibred Tannakian category (with one fibre functor for each object of $\Gamma$), so Tannakian duality provides unique pro-algebraic groupoid $G$ with the same finite-dimensional  representations as $\Gamma$. For any pro-algebraic groupoid $H$ and any groupoid homomorphism $\Gamma \to H(k)$, we then have a functor from $H$-representations to $\Gamma$ representations, and thus a unique compatible morphism $G\to H$, so  
 $\Gamma^{\alg}\cong G$.
\end{proof}

\begin{example}\label{algeg}
 The motivating example for this setup is when $\Gamma=\pi_fX$, the fundamental groupoid of a topological space. Then $(\pi_fX)^{\alg}$ is the pro-algebraic groupoid corresponding to the multifibred Tannakian category of local systems of finite-dimensional $k$-vector spaces on $X$. The fibre functors are given by $\vv \mapsto \vv_x$. Likewise, $(\pi_fX)^{\red}$ is the object corresponding to the Tannakian category of semisimple local systems.
\end{example}

\begin{definition}
Given a pro-algebraic groupoid $G$, and $U=\{U_x\}_{x \in \Ob(G)}$ a collection of pro-algebraic groups parametrised by $\Ob(G)$, we say that $G$ \emph{acts on} $U$ if there are morphisms $ U_x\by G(x,y) \xra{*} U_y$ of affine schemes, satisfying the following conditions:
\begin{enumerate}
\item $(uv)*g= (u*g)(v*g)$, $1*g=1$ and $(u^{-1})*g= (u*g)^{-1}$,  for $g \in G(x,y)$ and $u,v \in U_x$.

\item $u*(gh)=(u*g)*h$ and  $u*1=u$, for $g \in G(x,y), h \in G(y,z)$ and $u \in U_x$. 
\end{enumerate}

If $G$ acts on $U$, we construct $G \ltimes U $ as in Definition \ref{semidirect}.
\end{definition}

\begin{definition}
Given a pro-algebraic groupoid $G$, define the  \emph{pro-unipotent radical} $\Ru(G)$ to be the collection $\Ru(G)_x=\Ru(G(x,x))$ of pro-unipotent pro-algebraic groups, for $x \in \Ob(G)$. $G$ then acts on $\Ru(G)$ by conjugation, i.e.
$$
u*g:= g^{-1}u g,
$$
for $u \in \Ru(G)_x$, $g \in G(x,y)$.
\end{definition}

Now assume that the field $k$ is of characteristic $0$.
\begin{proposition}\label{leviprop}
For any pro-algebraic groupoid $G$, there is a Levi decomposition $G=G^{\red} \ltimes \Ru(G)$, unique up to conjugation by $\Ru(G)$.
\end{proposition}
\begin{proof}
\cite[Proposition \ref{htpy-leviprop}]{htpy}.
\end{proof}

\subsection{The pro-algebraic homotopy type of a topological space}\label{sagpdsn}

We now recall the results from \cite[\S \ref{htpy-sagpdsn}]{htpy}. The motivation here is that we wish to study the whole homotopy type, not just fundamental groupoids. This will involve working with the loop groupoid, which is a simplicial groupoid, so we need a simplicial framework. 

\begin{definition}
Given a simplicial object $G_{\bullet}$ in the category of pro-algebraic groupoids,   with $\Ob(G_{\bt})$ constant, define the  \emph{fundamental groupoid} $\pi_0(G_{\bullet})$ of $G_{\bt}$ to have objects $\Ob(G)$, and for $x,y \in \Ob(G)$, set 
$\pi_0(G)(x,y)$ to be the coequaliser 
$$
\xymatrix@1{G_1(x,y) \ar@<1ex>[r]^{\pd_1} \ar@<-1ex>[r]_{\pd_0}& G_0(x,y) \ar[r] &\pi_0(G)(x,y) }
$$ 
in the category of affine schemes. Thus $\pi_0(G)$ is a pro-algebraic groupoid on objects $\Ob(G)$, with multiplication inherited from $G_0$.
\end{definition}

\begin{definition}
Define a \emph{ pro-algebraic simplicial  groupoid}  to consist of a simplicial complex $G_{\bullet}$ of pro-algebraic groupoids, such that 
\begin{enumerate}
 \item $\Ob(G_{\bt})$ is constant,  and
\item for all $x \in \Ob(G)$, $G(x,x)_{\bt} \in s\agp$, i.e. the maps $G_n(x,x) \to \pi_0(G)(x,x)$ are pro-unipotent extensions of pro-algebraic groups.
\end{enumerate}
  We denote the category of pro-algebraic simplicial groupoids by $s\agpd$. 
\end{definition}

For any $G_{\bt} \in s\agpd$ and $x \in \Ob (G_{\bt})$, observe that $G_{\bt}(x,x)$ is a simplicial affine group scheme, so has homotopy groups $\pi_n(G_{\bt}(x,x))$. That these are also affine group schemes follows from the standard characterisation
\[
 \pi_n(G_{\bt}(x,x))= \H_n(NG_{\bt}(x,x), \pd_0)
\]
of homotopy groups of simplicial groups.

\begin{lemma}\label{cmsagpd}
There is a model structure on  $s\agpd$ in which a  morphism $f\co G_{\bullet} \to H_{\bullet}$ is:
\begin{enumerate}
\item a  weak equivalence if the map  $\pi_0(f)\co \pi_0(G_{\bullet})\to \pi_0(H_{\bullet})$ is an equivalence of pro-algebraic groupoids, and the maps  $\pi_n(f,x)\co \pi_n(G_{\bullet}(x,x)) \to \pi_n(H_{\bullet}(fx,fx))$ are isomorphisms for all $n$ and for all $x \in \Ob(G)$;

\item a fibration if 
the morphism  $N_n(f)\co  N(G(x,x))_n \to N(H(x,x))_n$ of normalised groups is surjective  for all $n>0$ and  all $x \in \Ob(G)$, and $f$ satisfies the path-lifting condition that for all $x \in \Ob(G), y \in \Ob(H)$, and $h \in H_0(fx,y)(k)$, there exists $z \in \Ob(G)$, $g \in G_0(x,z)(k)$ with $fg=h$.   Equivalently, this says that $G(k) \to H(k)$ is a fibration in the category of simplicial groupoids.
\end{enumerate}
\end{lemma}
\begin{proof}
This is \cite[Theorem \ref{htpy-cmsagpd}]{htpy}.
\end{proof}

We define $\Ho(s\agpd)$ to be the localisation of $s\agpd$ at weak equivalences.

There is a forgetful functor $(k)\co s\agpd \to s\gpd$, given by sending $G_{\bt}$ to $G_{\bt}(k)$. This functor  has a left adjoint $G_{\bt} \mapsto (G_{\bt})^{\alg}$. We can describe $(G_{\bt})^{\alg}$ explicitly. First let $(\pi_0(G))^{\alg}$ be the pro-algebraic completion of the abstract groupoid $\pi_0(G)$, then let $(G^{\alg})_n$ be the relative Malcev completion  (defined in \cite{hainrelative} for pro-algebraic groups) of the morphism
$$
G_n \to (\pi_0(G))^{\alg}.
$$
In other words, $G_n \to (G^{\alg})_n \xra{f} (\pi_0(G))^{\alg}$ is the universal diagram with $f$ a pro-unipotent extension.

\begin{proposition}\label{algqd}
The functors $(k)$ and $(-)^{\alg}$ give rise to a pair of adjoint functors 
$$
\xymatrix@1{\Ho(s\gpd) \ar@<1ex>[r]^{\bL^{\alg}} & \Ho(s\agpd) \ar@<1ex>[l]^{(k)}_{\bot} },
$$
with $\bL^{\alg}G(X)=G(X)^{\alg}$, for any $X \in \bS$ and $G$ as in Definition \ref{barwdef}.
\end{proposition} 
\begin{proof}
\cite[Proposition \ref{htpy-algqd}]{htpy}  shows that the functors are a Quillen pair, so the statement follows from the observation that all objects in $s\agpd$ are fibrant, making $(k)$ its own derived right Quillen functor. Since $G(X)$ is cofibrant, $\bL^{\alg}G(X)=G(X)^{\alg}$.
\end{proof}

The reason that we need to take $\bL^{\alg}$ in the Proposition is that $(-)^{\alg}$ is not an exact functor, so only preserves weak equivalences between cofibrant objects (which roughly correspond to free simplicial groupoids). In Examples \ref{goodexamples}, we will see examples of discrete groups $\Gamma$ for which the map $\bL^{\alg}\Gamma \to \Gamma^{\alg}$ is not a weak equivalence. 

\begin{definition}
Given a  simplicial set (or equivalently a  topological space), define the  \emph{pro-algebraic homotopy type} of $X$ over $k$ to be the object
$$
G(X)^{\alg}
$$ 
in $\Ho(s\agpd)$, where $G(X)$ is the loop groupoid of Definition \ref{barwdef}. Define the  \emph{pro-algebraic fundamental groupoid} by $\varpi_f(X):=\pi_0(G(X)^{\alg})$. Note that  $\pi_0(G^{\alg})$ is the pro-algebraic completion of the fundamental groupoid $\pi_0(G)$.

We then define the higher  \emph{pro-algebraic homotopy groups} $\varpi_n(X)$  (as $\varpi_fX$-representations) by 
$$
\varpi_n(X):=\pi_{n-1}(G(X)^{\alg}),
$$
where $\pi_n(G)$ is the representation $x \mapsto \pi_n(G(x,x))$, for $x \in \Ob(G)$.
\end{definition}

\begin{remark}\label{nonab} 
We can interpret $G(X)^{\alg} $ as the classifying object for non-abelian cohomology. Given $G \in s\agpd$, we can define $\H^1(X,G)$ to be the homotopy class of maps $G(X)^{\alg} \to G$, which is just $[X, \bar{W}G(k)]$. When $G$ is just a linear algebraic group, this recovers the usual definition of the set $\H^1(X,G)$ of classes of $G$-torsors on $X$. When $A$ is a simplicial finite-dimensional vector space (regarded as a simplicial algebraic group), this definition gives
\[
 \H^1(X,A)= \bH^1(X,NA),
\]
hypercohomology of the normalised complex associated to $A$.
\end{remark}

\subsection{Relative Malcev homotopy types}\label{malcevsn}

\begin{definition}\label{malcevdef}
Assume we have an abstract groupoid $G$, a reductive pro-algebraic groupoid $R$, and a representation $\rho\co G \to R(k)$ which is an isomorphism on objects and Zariski-dense on morphisms (i.e. $\rho\co G(x,y) \to R(k)(\rho x, \rho y)$ is Zariski-dense for all $x,y \in \Ob G$). Define the  \emph{Malcev completion} $(G,\rho)^{\mal}$ (or $G^{\rho, \mal}$, or $G^{R, \mal}$)   of $G$  \emph{relative to} $\rho$  to be the universal diagram
$$
G \to (G,\rho)^{\mal} \xra{p} R,
$$
with $p$  a pro-unipotent extension, and the composition equal to $\rho$. Explicitly, $\Ob(G,\rho)^{\mal}=\Ob G$ and 
$$
(G,\rho)^{\mal}(x,y)=(G(x,x), \rho)^{\mal}\by^{G(x,x)}G(x,y).
$$
If $G$ and  $R$ are groups, observe that this agrees with the usual definition (of \cite{hainrelative}).  

If $\varrho\co G \to R(k)$  is any any  Zariski-dense representation (i.e. essentially surjective on objects and Zariski-dense on morphisms) to a reductive pro-algebraic groupoid (in most examples, we take $R$ to be a group), we can define another reductive groupoid $\tilde{R}$ by setting $\Ob \tilde{R}=\Ob G$, and $ \tilde{R}(x,y)=R(\varrho x, \varrho y)$. This gives a representation $\rho\co \pi_fX \xra{\rho} \tilde{R}$ satisfying the above hypotheses, and we define the Malcev completion of $G$ relative to $\varrho$ to be the Malcev completion of $G$ relative to $\rho$. Note that $\tilde{R} \to R$ is an equivalence of pro-algebraic groupoids.
\end{definition}

\begin{definition}
Given a Zariski-dense morphism $\rho\co \pi_fX \to R(k)$, let the  \emph{Malcev completion} $G(X,\rho)^{\mal}$ of $X$  \emph{relative to} $\rho$ be the pro-algebraic simplicial group $(G(X), \rho)^{\mal}$.  Observe that the Malcev completion of $X$ relative to $(\pi_fX)^{\red}$ is just $G(X)^{\alg}$. Let $\varpi_f(X,\rho)^{\mal}=\pi_0G(X,\rho)^{\mal}$ and $\varpi_n(X,\rho)^{\mal}=\pi_{n-1}G(X,\rho)^{\mal}$. Note that $\pi_f((X,\rho)^{\mal})$ is the relative Malcev completion of $\rho\co \pi_fX \to R(k)$.  
\end{definition} 

Beware that the relative Malcev completion of $X$ is defined by completing a loop space for $X$, rather than $X$ itself. However, Theorem \ref{bigequiv} will give other equivalent formulations of the homotopy type, effectively by completing a covering space for $X$. 

\begin{lemma}\label{algpiIMlemma}
 Let $f\co X \to Y$ be a morphism in $\bS$ for which the map
\[
 \pi_n(f): \pi_n(X) \to \pi_n(Y) 
\]
is an isomorphism for $n \le N$ and a surjection for $n=N+1$, and take a Zariski-dense morphism $\rho\co \pi_fY \to R(k)$.
Then the map
\[
 \varpi_n(f): \varpi_n(X,\rho\circ f)^{\mal}  \to \varpi_n(Y, \rho )^{\mal}
\]
is an  isomorphism for $n \le N$ and a surjection for $n=N+1$.
\end{lemma}
\begin{proof}
As in the proof of Lemma \ref{piIMlemma}, for any $\pi_fY$-representation $M$,  the maps  $\H^i(Y,M) \to \H^{i}(X,M) $ are isomorphisms for $i\le N$ and injective for $i=N+1$. 

Now,  \cite[Proposition \ref{htpy-spectralh}]{htpy} gives a convergent  Adams spectral sequence 
$$
E^1_{pq}(X)= (\Lie_{-p} (\tilde{\H}^{*+1}(X,\bO(R) )^{\vee}))_{p+q} \abuts \varpi_{p+q+1}(X,\rho\circ f)^{\mal} ,
$$
in the category of pro-finite-dimensional vector spaces,
where  $\tilde{\H}$ denotes reduced cohomology, $\Lie_*$ is the free graded Lie algebra functor, and $\bO(R)$ is the local system of Definition \ref{OR}. Since $E^1_{pq}(X)\to E^1_{pq}(Y)$ is an isomorphism for $p+q< N$ and  surjective for $p+q=N$, the result follows.
\end{proof}

\begin{definition}\label{relgood}
Say that a groupoid $\Gamma$ is  $n$-\emph{good} with respect to a Zariski-dense representation $\rho\co  \Gamma \to R(k)$ to a reductive pro-algebraic groupoid if for all finite-dimensional $\Gamma^{\rho, \mal}$-representations $V$, the map
$$
\H^i(\Gamma^{\rho, \mal}, V) \to \H^i(\Gamma, V)
$$  
is an isomorphism for all $i\le n$ and an inclusion for $i=n+1$. Say that $\Gamma$ is \emph{good} with respect to $\rho$ if it is $n$-good for all $n$.
\end{definition}

See Lemma \ref{goodh} for alternative criteria to determine when a groupoid is $n$-good.

\begin{examples}\label{goodexamples}
By \cite[Examples \ref{goodexamples}]{htpy}, finite groups, free groups, finitely generated  nilpotent groups and fundamental groups of compact Riemann surfaces are all good with respect to all Zariski-dense representations. Super rigid  groups (such as $\SL_3(\Z)$) give examples of groups which are not good  with respect to any real (or complex) representations. This is because  $\Gamma^{R, \mal}=R$ in these cases, but $\H^*(\Gamma, \R)\ne \R$.
\end{examples}

\begin{theorem}\label{classicalpimal}
If $X$ is a  topological space with fundamental groupoid $\Gamma$, equipped with a Zariski-dense representation $\rho\co  \Gamma \to R(k)$ to a reductive pro-algebraic groupoid  for which: 
\begin{enumerate}
\item $\Gamma$ is  $(N+1)$-good with respect to $\rho$,
\item $\pi_n(X,-)$ is of finite rank for all $1<n \le N$, and
\item  the $\Gamma$-representation  $\pi_n(X,-)\ten_{\Z} k$ is an extension of $R$-representations (i.e. a $\Gamma^{\rho, \mal}$-representation) for all $1<n \le N$,
\end{enumerate}
then the canonical map
$$
  \pi_n(X,-)\ten_{\Z} k \to \varpi_{n}(X^{\rho, \mal},-) 
$$
is an isomorphism for all $1<n \le N$.
\end{theorem}
\begin{proof}
When $N=\infty$, this is \cite[Theorem \ref{htpy-classicalpimal}]{htpy}, but the same proof gives the conclusion above if we only assume that $\Gamma$ is  $(N+1)$-good (while still requiring the other conditions to hold for all $n$). For arbitrary $N$, and $X$ as above, this means that the $N$th stage $X(N)$ in the Postnikov tower for $X$ gives isomorphisms
\[
  \pi_n(X,-)\ten_{\Z} k \to \varpi_{n}(X(N)^{\rho, \mal},-)
\]
for all $1<n \le N$, since $\pi_iX(N)=0$ for $i>N$, while $\pi_iX(N)=\pi_iX$ for $i\le N$.

Applying Lemma \ref{algpiIMlemma} to the morphism $X \to X(N)$ now completes the proof.
\end{proof}

\subsection{Cohomology and hypercohomology}\label{gpcoho}

\subsubsection{Simplicial groupoids}

\begin{definition}\label{crepdef}
For a simplicial groupoid $\Gamma_{\bt}$, a \emph{cosimplicial $\Gamma_{\bt}$-representation} consists of the following
\begin{enumerate}

\item
a $\Gamma_n$-representation $V^n$ for all $n$, with $g\cdot\pd^iv = \pd^i((\pd_ig)\cdot v)$, for $g \in \Gamma_{n+1}$, $v \in V^n$;

\item operations $\pd^i, \sigma^i$  making  $V^{\bt}(x)$ into  a cosimplicial complex for each $x \in \Ob \Gamma_{\bt}$, satisfying the additional conditions that
 $$
g\cdot(\pd^iv) = \pd^i((\pd_ig)\cdot v) \quad h\cdot(\sigma^iv) = \sigma^i((\sigma_ig)\cdot v)
$$
 for $g \in \Gamma_{n+1}(x,y)$, $h \in \Gamma_{n-1}(x,y)$,  $v \in V^n(y)$.
\end{enumerate}
\end{definition}

\begin{remark}
 If $\Gamma_{\bt}= G(X)$, then we can think of a cosimplicial $\Gamma_{\bt}$-representation as being a kind of hyper-local system on $X$. As we will see below, these give a sufficiently large category to recover cohomology, but objects with constant cosimplicial structure are still just local systems.
\end{remark}

\begin{definition}
Given a simplicial groupoid $\Gamma_{\bt}$ and a  cosimplicial $\Gamma_{\bt}$-representation $V$, define the cosimplicial complex $\CC^{\bt}(\Gamma_{\bt},V)$ by
$$
\CC^n(\Gamma_{\bt},V)= \Hom_{\Gamma_n}((W\Gamma_{\bt})_n, V^n),
$$
for the functor $W$ from Definition \ref{wdef},
with operations $(\pd^if)(x)= \pd^i_V (f(\pd_ix))$ for $x \in (W\Gamma_{\bt})_{n+1}$, and $(\sigma^if)(x)= \sigma^i_V (f(\sigma_ix))$ for $x \in (W\Gamma_{\bt})_{n-1}$.

Then define hypercohomology groups $\bH^i(\Gamma_{\bt},V)$ by $\bH^i(\Gamma_{\bt},V)= \H^i\CC(\Gamma_{\bt},V)$. If $V$ is a $\pi_0\Gamma_{\bt}$-representation, regard $V$ as a  cosimplicial  $\Gamma_{\bt}$-representation (with constant cosimplicial structure) and write $\H^i(\Gamma_{\bt},V):= \bH^i(\Gamma_{\bt},V)$. 
\end{definition}

\begin{lemma}\label{whypercohocoho}
If $\Gamma_{\bt}$ is a simplicial groupoid and  $V$  a $\pi_0\Gamma_{\bt}$-representation, then 
$$
\H^i(\Gamma_{\bt},V) = \H^i(\bar{W}\Gamma_{\bt},V).
$$
\end{lemma}
\begin{proof}
Observe that  $\pi_0(\Gamma_{\bt})\by^{\Gamma_{\bt}} (W\Gamma_{\bt})$ is the universal covering system of $\bar{W}\Gamma_{\bt}$.
Since $V$ is a $\pi_0\Gamma_{\bt}$-representation, 
$$
\Hom_{\Gamma_n}((W\Gamma_{\bt})_n, V)= \Hom_{\pi_0\Gamma_{\bt}}(  (\pi_0\Gamma_{\bt})\by^{\Gamma_n}(W\Gamma_{\bt})_n, V)= \Hom_{\pi_f\bar{W}\Gamma}(\widetilde{(\bar{W}\Gamma_{\bt})}_n,V),
$$
so $\CC^{\bt}(\Gamma_{\bt},V)= \CC^{\bt}(\bar{W}\Gamma_{\bt},V)$ (as defined in Definition \ref{CCdef}), giving the required result.
\end{proof}

\begin{lemma}\label{cohohypercoho}
If $\Gamma_{\bt}$ is a simplicial groupoid and  $V$  a cosimplicial $\Gamma_{\bt}$-representation, then there is a 
convergent spectral sequence
$$
\H^{i}(\Gamma_{\bt}, \H^j(V)) \abuts \bH^{i+j}(\Gamma_{\bt},V),
$$
where $H^j(V)$ is the $\pi_0\Gamma_{\bt}$-representation given by setting $\H^*(V)(x)$ to be cohomology of the cosimplicial complex $V(x)$, for all $x \in \Ob \Gamma_{\bt}$. 
\end{lemma}
\begin{proof}
Form the filtration $\{F_nV\}_n$ of $V$ by setting $F_nV$ to be the image of the  $n$-skeleton $\sk^nV \to V$; $F_nV$ is the subcomplex of $V$ generated under the operations $\pd^i$ by $V^{\le n}$, and its Dold-Kan normalisation is given by 
$$
N(F_nV)^i= \left\{\begin{matrix} N^iV & i \le n \\ d N^nV & i=n+1 \\ 0 & i \ge n+2. \end{matrix} \right.
$$

Note that the condition $g\pd^iv= \pd^i((\pd_ig)v$ implies that $F_nV$ is $\Gamma_{\bt}$-equivariant. Also note that $F_nV/ F_{n-1}V$ is quasi-isomorphic to the denormalisation $D H^n(V)[-n]$. The spectral sequence associated to this filtration is thus
$$
\bH^{i+j}(\Gamma_{\bt}, D\H^j(V)[-j]) \abuts \bH^{i+j}(\Gamma_{\bt},V).
$$

Let $K_{\bt}:= \ker(\Gamma_{\bt} \to \pi_0\Gamma_{\bt})$; since  $\H^jV$ is a $\pi_0\Gamma_{\bt}$-representation, there is a bicosimplicial complex  
 $
 C^{a,b}:= \Hom_{\pi_0(\Gamma_{\bt})}(  K_a\backslash (W\Gamma_{\bt})_a, D^b\H^j(V)[-j]),
 $  
 with $ \bH^n(\Gamma_{\bt}, D\H^j(V)[-j])= \H^n(\diag C)$. By the Eilenberg-Zilber Theorem (\cite[Theorem 8.5.1]{W}), $N\diag C$  is quasi-isomorphic to the total complex of $\Hom_{\pi_0\Gamma_{\bt}}(N\Z(K_{\bt}\backslash W\Gamma_{\bt}),   \H^j(V)[-j])$, so $\H^n(C)= \H^{n-j}(G, \H^j(V))$, and the spectral sequence becomes
$$
\bH^{i}(\Gamma_{\bt}, H^j(V)) \abuts\bH^{i+j}(\Gamma_{\bt},V).
$$
\end{proof}

\begin{lemma}\label{hypercohowell}
 Given a weak equivalence $f:\Gamma_{\bt}\to \Delta_{\bt}$ of simplicial groupoids,  and a  cosimplicial $\Gamma_{\bt}$-representation   $V$, the map
\[
 f^*\co      \bH^{*}(\Delta_{\bt},V)  \to \bH^{*}(\Gamma_{\bt},f^{-1}V)
\]
   is an isomorphism. 
\end{lemma}
\begin{proof}
 Lemma \ref{cohohypercoho} gives a morphism of convergent spectral sequences, so we may assume that $V$ is a $\pi_0\Delta_{\bt}$-representation. Since $\bar{W}f\co \bar{W} \Gamma_{\bt}\to \bar{W}\Delta_{\bt}$  is   a weak equivalence  of simplicial sets, Lemma \ref{whypercohocoho} completes the proof.
\end{proof}

\begin{lemma}\label{leraylemma}
Given a simplicial group $\Gamma_{\bt}$, a cosimplicial $\Gamma_{\bt}$-representation $V$ and a simplicial abelian group $A$, the simplicial abelian group 
$$
\Tot(V\ten A)
$$
has a canonical $\Gamma_{\bt}$-action, where $\Tot\co \bS^{\Delta}\to \bS$ is the total space functor of 
\cite[Ch. VIII]{sht}, originally defined in \cite[Ch. X]{bousfieldkan}.
\end{lemma}
\begin{proof}
Given $X \in \bS^{\Delta}$ and $K \in \bS$, define $e(X,K) \in \bS^{\Delta}$ by  $e(X,K)^n:= (X^n)^{K_n}$, with obvious cosimplicial operations. Note that $\Tot(e(X,K))= \Tot(X)^K$.

The $\Gamma_{\bt}$-action on $V$ is the same as a cosimplicial map $f\co V \to e(V, \Gamma_{\bt})$, so we have maps
$$
V\ten A \xra{f} e(V, \Gamma_{\bt})\ten A \to e(V\ten A, \Gamma_{\bt}), 
$$
and hence a map $\Tot(V\ten A) \to \Tot(V\ten A)^{\Gamma_{\bt}}$. this is equivalent to a map $\Gamma_{\bt} \by \Tot(V\ten A) \to \Tot(V\ten A)$ of simplicial sets, and the argument above adapts to show that this action is associative. 
\end{proof}

In order to simplify the definitions and exposition, we will now take  $\Gamma_{\bt}$ to be simplicial group, although everything can be extended to simplicial groupoids.

\begin{definition}
For a simplicial group $\Gamma_{\bt}$, a \emph{simplicial $\Gamma_{\bt}$-representation} consists of a simplicial abelian group $A$, together with a $\Gamma_n$-action on $A_n$ for all $n$, compatible with the simplicial operations. Let $s\Rep(\Gamma_{\bt})$ be the category of simplicial $\Gamma_{\bt}$-representations.
\end{definition}

Note that Lemma \ref{leraylemma}  provides us with examples of simplicial $\Gamma_{\bt}$-representations constructed from cosimplicial $\Gamma_{\bt}$-representations. Also note that for any simplicial $\Gamma_{\bt}$-representation $V$, taking  duals levelwise gives a cosimplicial  $\Gamma_{\bt}$-representation $V^{\vee}$ given by  $(V^{\vee})^n= (V_n)^{\vee}$.

\begin{lemma}
Given a simplicial group $\Gamma_{\bt}$, there is a cofibrantly generated  model structure on $s\Rep(\Gamma_{\bt})$, in which a morphism $f\co A \to B$ is:
\begin{enumerate}
\item a weak equivalence if the maps $\pi_i(f)\co  \pi_i(A) \to \pi_i(B)$ are isomorphisms for all $i$;
\item a fibration if the underlying map in $\bS$ is a fibration, or equivalently if the maps $N_i(f)\co  N_i(A) \to N_i(B)$ on the Dold-Kan normalisation are surjective for all $i>0$.
\end{enumerate}
 \end{lemma}
\begin{proof}
The forgetful functor from $s\Rep(\Gamma_{\bt})$ to simplicial sets  preserves filtered direct limits  and has a  left adjoint $F(S) = \Z(\Gamma_{\bt}\by S)$. Thus  for any finite object $I \in \bS$, the object $FI$ is finite in $s\Rep(\Gamma_{\bt})$, so \emph{a fortiori} permits the small object argument.
The model structure on $\bS$ is cofibrantly generated by finite objects, so  \cite[Theorem 11.3.2]{Hirschhorn}  gives the required model structure on $s\Rep(\Gamma_{\bt})$.
\end{proof}

\begin{lemma}\label{leraylemma2}
We may characterise hypercohomology groups by
$$
\bH^i(\Gamma_{\bt},V) = \Hom_{\Ho(s\Ab(\Gamma_{\bt}))}(\Z, \Tot(V\ten_{\Z} N^{-1}\Z[-i])).
$$
\end{lemma}
\begin{proof}
We first note that $\Z[W\Gamma_{\bt}]$ is a cofibrant replacement for $\Z$, so for a simplicial abelian group $A$,
$$
\oR\HHom_{s\Ab(\Gamma_{\bt})}(\Z, \Tot(V\ten_{\Z} A))\simeq \HHom_{s\Ab(\Gamma_{\bt})}(\Z[W\Gamma_{\bt}], \Tot(V\ten_{\Z} A));
$$
as observed in the proof of Lemma \ref{leraylemma}, 
$$ 
\HHom_{s\Ab}(\Z[W\Gamma_{\bt}], \Tot(V\ten_{\Z} A))\cong \Tot( e(V\ten_{\Z}A, W\Gamma_{\bt})),
$$ 
so
$$
\HHom_{s\Ab(\Gamma_{\bt})}(\Z[W\Gamma_{\bt}], \Tot(V\ten_{\Z} A))\cong \Tot( e(V\ten_{\Z}A, W\Gamma_{\bt})^{\Gamma_{\bt}}).
$$

Now, $e(V\ten_{\Z}A, W\Gamma_{\bt})^{\Gamma_{\bt}}$ is given in simplicial level $n$ by $\CC^{\bt}(\Gamma_{\bt}, V\ten_{\Z} A_n)$. When $A_n$ is free and finitely generated, this becomes  $\CC^{\bt}(\Gamma_{\bt}, V)\ten_{\Z} A_n$. Taking  $A= N^{-1}\Z[-i]$ thus gives
\begin{eqnarray*}
\Hom_{\Ho(s\Ab(\Gamma_{\bt}))}(\Z, \Tot(V\ten_{\Z} N^{-1}\Z[-i]))&\cong& \pi_0 \oR\HHom_{s\Ab(\Gamma_{\bt})}(\Z, \Tot(V\ten_{\Z}N^{-1}\Z[-i] ))\\
&\cong& \Tot( \CC^{\bt}(\Gamma_{\bt}, V)\ten_{\Z} N^{-1}\Z[-i]).
\end{eqnarray*}
Given  a cosimplicial simplicial abelian group $B$, the normalisation $N\Tot B$ is equivalent to the good truncation in non-negative chain degrees of the product total complex $\Tot^{\prod}N_cNB$ of the binormalisation of $B$ (which is a cochain chain complex). Thus
$$
\pi_0 \Tot( \CC^{\bt}(\Gamma_{\bt}, V)\ten_{\Z} N^{-1}\Z[-i]) \cong \H_0 \Tot^{\prod} ((N_c\CC^{\bt}(\Gamma_{\bt}, V))\ten_{\Z}\Z[-i]),
$$
and $\Tot^{\prod} ((N_c\CC^{\bt}(\Gamma_{\bt}, V))\ten_{\Z}\Z[-i])$ is just the complex $ N_c\CC^{\bt}(\Gamma_{\bt}, V)$ turned upside down and shifted $i$ places, so
$$
\H_0 \Tot^{\prod} ((N_c\CC^{\bt}(\Gamma_{\bt}, V))\ten_{\Z}\Z[-i])= \H^iN_c\CC^{\bt}(\Gamma_{\bt}, V)= \bH^i(\Gamma_{\bt}, V),
$$
as required.
\end{proof}

The following is an analogue of the Leray spectral sequence, and will play a key r\^ole in Theorem \ref{lfibrations}.

\begin{proposition}\label{lerayserre}
Given a surjection $\Gamma_{\bt}\to \Delta_{\bt}$ of simplicial groups with kernel $B_{\bt}$, and a cosimplicial $\Gamma_{\bt}$-representation $V$, there is a canonical convergent spectral sequence
$$
\H^i(\Delta_{\bt}, \bH^j(B_{\bt},V))\abuts  \bH^{i+j}( \Gamma_{\bt},V),
$$
which we refer to as the Hochschild--Serre spectral sequence.
\end{proposition}
\begin{proof}
Given $T \in s\Ab(\Delta_{\bt})$ and $U,W \in s\Ab(\Gamma_{\bt})$, we have an isomorphism
\[
 \HHom_{s\Ab(\Delta_{\bt})}(T, \HHom_{s\Ab(B_{\bt})}(U,W)) \cong  \HHom_{s\Ab(\Gamma_{\bt})}(T\ten U,W).
\]
This defines a right Quillen functor $ s\Ab(\Delta_{\bt})^{\op} \by s\Ab(\Gamma_{\bt})^{\op}\by s\Ab(\Gamma_{\bt}) \to \bS$; since any cofibrant $\Gamma_{\bt}$-representation is cofibrant as a $B_{\bt}$-representation, the isomorphism above gives an equivalence
\[
 \oR\Hom_{s\Ab(\Delta_{\bt})}( T, \oR\HHom_{s\Ab(B_{\bt})}(U, W)) \simeq \oR\HHom_{s\Ab(\Gamma_{\bt})}(T\ten^{\oL}U, W).
\]
In particular,
$$
\oR\HHom_{s\Ab(\Gamma_{\bt})}(\Z, W)\simeq \oR\Hom_{s\Ab(\Delta_{\bt})}( \Z, \oR\HHom_{s\Ab(B_{\bt})}(\Z, W)).
$$

Setting $W=\Tot(V\ten_{\Z} N^{-1}\Z[-n])$, this gives an isomorphism
$$
\bH^n(\Gamma_{\bt}, V) \cong \H^n\CC^{\bt}(\Delta_{\bt}, \CC^{\bt}(B_{\bt}, V)),
$$ 
so the morphism $ \CC^{\bt}(\Gamma_{\bt},V) \to \CC^{\bt}(\Delta_{\bt}, \CC^{\bt}(B_{\bt}, V))$ is a quasi-isomorphism, and the result now follows from Lemma \ref{cohohypercoho}.
\end{proof}

\subsubsection{Simplicial pro-algebraic groupoids}

\begin{definition}
Given $G \in s\agpd$, define a \emph{cosimplicial $G$-representation} to be an $O(G)$-comodule $V$ in cosimplicial $k$-vector spaces. Thus we have cosimplicial complexes $V(x)$ for all $x \in \Ob G$, together with a coassociative coaction
$V(x) \to O(G)(x,y)\ten V(y)$.
\end{definition}

Note that the category of cosimplicial $G$-representations is opposite to the category $s\widehat{\FD\Rep}(G)$ of  pro-finite-dimensional simplicial $G$-representations from \cite[\S 1.5]{htpy}.

\begin{definition}
Given  $G \in s\agpd$ and a cosimplicial $G$-representation $V$, define the cosimplicial complex $\CC^{\bt}(G,V)$ by
$$
\CC^n(G,V)= O((WG)_n)\ten^{G_n} V^n,
$$
for the functor $W$ from Definition \ref{wdef}, with operations $\pd^i\ten \pd^i$ and $\sigma^i\ten \sigma^i$. 

Then define hypercohomology groups $\bH^i(G,V)$ by $\bH^i(G,V)= \H^i\CC(G,V)$. If $V$ is a $\pi_0G$-representation, regard $V$ as a  cosimplicial  $G$-representation (with constant cosimplicial structure) and write $\H^i(G,V):= \bH^i(G,V)$. 
\end{definition}

Now, \cite[Example \ref{htpy-owres}]{htpy}  ensures that $\bH^i(G,V)^{\vee}= \H_i(G,V^{\vee})$ in the notation of \cite[Definition 1.48]{htpy}. In particular, this means that hypercohomology groups of $G$ are an invariant of the homotopy type of $G$. 

\begin{proposition}\label{detectweak}
A morphism $G \xra{f} K$ of  pro-algebraic simplicial groupoids is a weak equivalence if and only if 
\begin{enumerate}
\item
$f(\Ru(G))\le \Ru(K)$,  with the quotient map
$$
G^{\red} \to K^{\red}
$$
 an equivalence, and 
 
\item for all finite-dimensional irreducible $K$-representations $V$, the maps
$$
\H^i(f)\co  \H^i(K,V) \to \H^i(G,f^*V)
$$
are isomorphisms for all $i>0$.
\end{enumerate}
\end{proposition}
\begin{proof}
This is \cite[Corollary \ref{htpy-detectweak}]{htpy}, adapted from groups to groupoids.
\end{proof}

Note that the analogue of Lemma \ref{leraylemma} for pro-algebraic simplicial groupoids thus ensures that weak equivalences induce isomorphisms on hypercohomology.

\begin{lemma}\label{cohomaps}
For  a cofibrant pro-algebraic simplicial group $G$ (for the model structure of Lemma \ref{cmsagpd}), and a  finite-dimensional $\pi_0G$-representation $V$, the cohomology group $\H^i(G,V)$ is isomorphic to the  homotopy class of maps $G \to G \ltimes (N^{-1}V[1-i])$ in the model category $s\agpd \da  G$.
\end{lemma}
\begin{proof}
Consider the morphism $k \to O(G)$, and let the cokernel be $C$. As in the proof of \cite[Proposition 1.50]{htpy}, $C$ is fibrant as a cosimplicial $G$-representation. Likewise, $ V\ten  O(G)$ and  $V\ten C$ are both fibrant, so $\H^*(G,V)$ is cohomology of the cone complex of 
$$
V\ten^G  O(G) \to V\ten^G C.
$$
Now, $V\ten^G  O(G) $ is just $ V$, so we need to describe $ V\ten^G C$.

Letting $E:=  O(G)^{\vee}$, we see that $C^{\vee}$ is the kernel of $E \to k$. Elements $\theta$ of $ V\ten^G_n C^n$ are then just morphisms $\theta\co  (C^{\vee})_n \to V$ satisfying $\alpha (gc) = g\alpha(c)$, for $g\in G_n, c \in (C^{\vee})_n$. There is a map $E \to C^{\vee}$ given by $a \mapsto a-1$, so $\theta$ composed with this gives a linear morphism $\theta'\co  E_n \to V$, satisfying $\theta'(ga)=g\theta'(a)+ \theta'(g)$ for $g \in G_n$. 

Regarding $E_n$ as an affine scheme, there is a morphism  $G_n \to E_n$, so we see that $\theta$ corresponds to a derivation $\theta'\co  G_n \to V$.  Since derivations $G \to V$ are just morphisms $G \to G \ltimes V$ over $G$, the statement now follows from the description of the path object in $s\agpd$ from \cite[Lemma 2.29]{htpy}.
\end{proof}

\begin{lemma}\label{cohomalworks}
If $\Gamma_{\bt}$ is a cofibrant
simplicial groupoid (e.g. $G(X)$ for $X \in \bS$), and  $V$ is a finite-dimensional $\pi_0\Gamma_{\bt}^{R, \mal}$-representation, then the map
$$
\H^*(\Gamma_{\bt}^{R, \mal}, V) \to \H^*(\Gamma_{\bt}, V)
$$
is an isomorphism.
\end{lemma}
\begin{proof}
This is implicit in \cite[\S 1.5.3]{htpy}. Replacing $\Gamma_{\bt}$ with a disjoint union of simplicial groups,   Lemma \ref{cohomaps} gives that $\H^*(\Gamma_{\bt}^{R, \mal}, V)$ 
is the homotopy class of maps from $\Gamma_{\bt}^{R, \mal}$  to $\Gamma_{\bt}^{R, \mal}\ltimes (N^{-1}V[1-i])$ over $\Gamma_{\bt}^{R, \mal}$. Since any map from $ \Gamma_{\bt}^{\alg}$ to a pro-unipotent extension of $R$ factors through  $\Gamma_{\bt}^{R, \mal}$, this is the same as the   homotopy class of maps from $\Gamma_{\bt}^{\alg}$  to $\Gamma_{\bt}^{R, \mal}\ltimes (N^{-1}V[1-i])$ over $\Gamma_{\bt}^{R, \mal}$.

The Quillen adjunction of Proposition \ref{algqd} then shows that this is equivalent to the homotopy class of maps from $\Gamma_{\bt}$ to $ \Gamma_{\bt} \ltimes  (N^{-1}V[1-i])$ in the slice category $s\gpd\da \Gamma_{\bt}$, which is just $\H^i(\Gamma_{\bt}, V)$.
\end{proof}

Note that if we have  $\Gamma_{\bt} \in s\gpd$ and $G \in s\agpd$ together with a morphism $f\co \Gamma_{\bt} \to G(k)$ of simplicial groupoids, then every cosimplicial $G$-representation $V$ naturally gives rise to a cosimplicial $\Gamma_{\bt}$-representation $f^*V$. For any coalgebra $C$, every $C$-comodule is a nested union of finite-dimensional comodules. Thus every cosimplicial $G$-representation $V$ is a filtered direct limit $\LLim_{\alpha} V_{\alpha}$ of levelwise finite-dimensional cosimplicial $G$-representations,  and we tweak the construction of pullbacks slightly by regarding $f^*V$ as the ind-object (i.e. filtered direct system) $\{f^*V_{\alpha}\}$ of levelwise finite-dimensional cosimplicial $\Gamma_{\bt}$-representations. We then define $\CC^{\bt}(\Gamma_{\bt}, f^*V):= \LLim_{\alpha}\CC^{\bt}(\Gamma_{\bt}, f^*V_{\alpha})$, and $\bH^*(\Gamma_{\bt},f^*V):= \H^*\CC^{\bt}(\Gamma_{\bt}, f^*V)= \LLim_{\alpha}\bH^*(\Gamma_{\bt},f^*V_{\alpha})$.

Also note that the category of cosimplicial $G$-representations is opposite to the category $s\widehat{\FD\Rep}(G)$ of \cite[\S 1.5]{htpy}.

\begin{lemma}\label{hypercohogood}
Given a cofibrant simplicial groupoid $\Gamma_{\bt}$
and a cosimplicial $O(\Gamma_{\bt}^{R, \mal})$-comodule $V$, 
the canonical map
$$
\bH^*(\Gamma_{\bt}^{R, \mal},V) \to \bH^*(\Gamma_{\bt},V)
$$
(induced by the morphism $W\Gamma_{\bt}\to W(\Gamma_{\bt}^{R, \mal})$)
is an isomorphism.
\end{lemma}
\begin{proof}
By Lemma \ref{cohohypercoho} and its analogue for $s\agpd$, 
we have convergent 
spectral sequences 
\begin{eqnarray*}
\bH^{i}(\Gamma, H^j(V)) &\abuts& \bH^{i+j}(\Gamma_{\bt},V)\\
\bH^{i}(G, \H^j(V)) &\abuts& \bH^{i+j}(G,V).
\end{eqnarray*}

For ind-finite-dimensional $\pi_0G$-representations $U$, the maps $\H^i(G,U) \to\H^i(\Gamma_{\bt}, U)$ are isomorphisms by  Lemma \ref{cohomalworks}, so  the maps $\H^i(G,\H^j(V)) \to\H^i(\Gamma_{\bt}, \H^j(V))$ are isomorphisms, making the morphism of spectral sequences an isomorphism.
\end{proof}

\begin{theorem}\label{fibrations}
Take a fibration $f\co (X,x) \to (Y,y)$ (of pointed connected  topological spaces) with connected fibres,  and set $F:= f^{-1}(y)$.
Take a Zariski-dense representation $\rho\co  \pi_1(X,x) \to R(k)$ to a  reductive pro-algebraic group $R$, let $K$ be the closure of $\rho(\pi_1(F,x))$, and set $T:= R/K$. If the monodromy action of $\pi_1(Y,y)$ on $\H^*(F, V)$ factors through $\varpi_1(Y,y)^{T, \mal}$ for all $K$-representations $V$, then $G(F,x)^{K,\mal}$ is the homotopy fibre of $ G(X,x)^{R, \mal} \to G(Y,y)^{T, \mal}$.
\end{theorem}
\begin{proof}
This is \cite[Theorem \ref{mhs-fibrations}]{mhs}, which uses Lemma \ref{hypercohogood} to show that  $\H^*(G(F,x)^{K,\mal},O(K))$  and cohomology $\H^*(\cF,O(K))$ of the homotopy fibre $\cF$ are both   $\bH^*(G(X,x), O(R)\ten_{O(T)}O(G(Y,y)^{T, \mal}))\cong \H^j(F, O(K))$. 
\end{proof}

\subsection{Equivalent formulations}

Fix a reductive pro-algebraic groupoid $R$.


\subsubsection{Lie algebras}

\begin{definition}\label{indcon}
Recall that a Lie coalgebra $C$ is said to be \emph{conilpotent} if the iterated cobracket $\Delta_n\co  C \to C^{\otimes n}$ is $0$ for sufficiently large $n$. A Lie coalgebra $C$ is \emph{ind-conilpotent} if it is a filtered direct limit (or, equivalently, a nested union) of conilpotent Lie coalgebras.
\end{definition}

\begin{definition}
 Recall from \cite[Definition \ref{htpy-cna}]{htpy}  that for any $k$-algebra $A$, we define $\hat{\cN}_A(R)$  to be opposite to the category of $R$-representations in ind-conilpotent Lie coalgebras over $A$, and denote the contravariant equivalence by $C \mapsto C^{\vee}$.
\end{definition}

Note that  there is a continuous functor $\hat{\cN}_k(R)\to \hat{\cN}_A(R)$ given by $C^{\vee} \mapsto (C\ten_k A)^{\vee}$. We denote this by $\g \mapsto \g \hat{\ten} A$.

\begin{remark}
Observe that $\g \in  \hat{\cN}_A(R)$ can be regarded as an object of the category $\Aff_A(R)$ of $R$-representations in affine $A$-schemes, by regarding it as the functor 
$$
\g(B):= \Hom_{A,R}(\g^{\vee},B),
$$
for $B \in \Alg_A(R):= A \da \Alg(R)$. In fact, $\g(B)$ is then a Lie algebra over $B$, so the Campbell--Baker--Hausdorff formula defines a group structure on $\g(B)$, and the resulting group is denoted by $\exp(\g)(B)$. Thus $\exp(\g)$ is an $R$-representation in affine group schemes over $A$ (i.e. a group object of $\Aff_A(R)$).
\end{remark}

\begin{definition}\label{cndef}
 Write $s\hat{\cN}_A(R)$ for the category of simplicial objects in $\hat{\cN}_A(R)$. A weak equivalence in $s\hat{\cN}_A(R)$ is a map which gives isomorphisms on cohomology groups of the duals  (which are just $A$-modules). We denote by $\Ho(s\hat{\cN}_A(R))$ the localisation of  $s\hat{\cN}_A(R)$ at weak equivalences. 
\end{definition}

For $k=A$, we will usually drop the subscript, so $\hat{\cN}(R):=\hat{\cN}_k(R)$, and so on.

\begin{definition}\label{cedef}
Define $\cE(R)$ to be the full subcategory of $\agpd\da R$ consisting of those morphisms $\rho\co G\to  R$ of  pro-algebraic groupoids which are pro-unipotent extensions. Similarly, define $s\cE(R)$ to consist of the pro-unipotent extensions in $s\agpd\da R$, 
and $\Ho(s\cE(R)_*)$ to be full subcategory of $\Ho(\Ob R \da s\agpd)$ on objects $s\cE(R)$.
\end{definition}

\begin{definition}
Given a pro-algebraic groupoid $R$, define the category $s\cP_A(R)$ to have the same objects as $s\hat{\cN}_A(R)$, with morphisms given by
$$
\Hom_{s\cP(R)}(\g,\fh)= \exp(\prod_{x \in \Ob R}\pi_0\fh(x))\by^{\exp(\fh^R_0)} \Hom_{\Ho(s\hat{\cN}(R))}(\g,\fh),
$$
where $\fh_0^R$ (the Lie subalgebra of $R$-invariants in $\fh_0$)  acts by conjugation on the set of homomorphisms. Composition of morphisms is given by $(u,f) \circ (v,g)= (u\circ f(v), f\circ g)$.
\end{definition}

The following is a key comparison result, which will be used in  Proposition \ref{eqhtpy} and Theorem \ref{qleqhtpy} as a step towards reformulating Malcev homotopy types in terms of Godement resolutions.

\begin{proposition}\label{meequiv}
For any reductive pro-algebraic groupoid $R$, the categories $\Ho(s\cE(R)_*)$ and $s\cP(R)$ are equivalent.
\end{proposition}
\begin{proof}
This is part of \cite[Theorem \ref{bigequiv}]{htpy}, adapting  \cite[Proposition \ref{meequiv}]{htpy}  to the unpointed case. 
The proof just exploits the Levi decomposition of Proposition \ref{leviprop}. 

Explicitly, the functor  maps $\g \in s\cP(R)$ to the simplicial pro-algebraic group given in level $n$ by $R \ltimes \exp(\g_n)$.  Given a morphism 
$$
( u,f) \in \exp(\prod_{x \in \Ob R}\pi_0\fh(x))\by^{\exp(\fh^R)}\Hom_{\Ho(s\cN(R))}(\g, \fh),
$$ 
lift $u$ to $\tilde{u} \in \prod_{x \in \Ob R} \exp(\fh_0(x))$, and construct the morphism 
$$
 \ad_{\tilde{u}}\circ(R \ltimes \exp(f)) \co  R \ltimes \exp(\g) \to R \ltimes \exp(\fh)
$$
in  $s\cE(R)$, where for $a \in  (R \ltimes \exp(\fh))(x,y)$, we set  $\ad_{\tilde{u}} (a)= \tilde{u}(x)\cdot  a \cdot \tilde{u}(y)^{-1}$.
\end{proof}

\begin{definition}\label{hrelmaldef}
We can now define the \emph{multipointed Malcev homotopy type} of $X$ relative to $\rho$ to be the image of   $G(X,\rho)^{\mal}$ in $\Ho(s\cE(\tilde{R})_*)$, or equivalently $\Ru G(X,\rho)^{\mal}$ in  $s\cP(\tilde{R})$. Define the \emph{unpointed 
Malcev homotopy type} of $X$ relative to $\rho$ to be the image of   $G(X,\rho)^{\mal}$ in $\Ho(s\cE(\tilde{R}))$.

Since $\tilde{R} \to R$ is an equivalence of groupoids, 
there is an equivalence $\Ho(s\cE(R)) \to \Ho(s\cE(\tilde{R}))$, so may discard some basepoints to give  an object of $s\cP(R)$ (or equivalently of $\Ho(s\cE(R)_*)$) whenever $\rho$ is surjective on objects. 
\end{definition}

\subsubsection{Chain Lie algebras}

\begin{definition}Let  $dg\hat{\cN}_A$  be opposite to the category of non-negatively graded ind-conilpotent cochain Lie  coalgebras over $A$. Define $dg\hat{\cN}_A(R)$ to be the category of $R$-representations in $dg\hat{\cN}_A$.
 For $k=A$, we will usually drop the subscript, so $dg\hat{\cN}(R):=dg\hat{\cN}_k(R)$, and so on.
\end{definition}

The following is \cite[Lemma \ref{htpy-cnamod}]{htpy} :
\begin{lemma}
There is a closed model structure on $dg\hat{\cN}_A(R)$  in which a morphism $f\co \g \to \fh$ is 
\begin{enumerate}
\item
a fibration whenever  the underlying map $f^{\vee}\co  \fh^{\vee} \to \g^{\vee}$ of cochain complexes over $A$ is injective in strictly positive degrees;

\item 
a weak equivalence whenever the maps 
$\H^i(f^{\vee})\co  \H^i(\fh^{\vee}) \to \H^i(\g^{\vee})$ are isomorphisms for all $i$. 
\end{enumerate}
\end{lemma}

\begin{remark}\label{einfty}
It follows from the construction in \cite[Lemma \ref{htpy-cnamod}]{htpy}  that for  cofibrant objects $\g \in dg\hat{\cN}(R)$ (taking $A$ to be a field), $\g^{\vee}$ is freely cogenerated as a graded Lie coalgebra. 
Thus $\g^{\vee}[-1]$ is a positively graded strong homotopy commutative algebra without unit (in the sense of \cite[Lectures 8 and 15]{Kon}), and  a choice of cogenerators on $\g^{\vee}$ is the same as a positively graded $E_{\infty}$ (a.k.a. $C_{\infty}$) algebra --- this is an aspect of Koszul duality.
\end{remark}

\begin{definition}
We say that a morphism $f\co \g \to \fh$ in $dg\hat{\cN}(R)$ is \emph{free} if there exists a (pro-finite-dimensional)  sub-$R$-representation $V \subset \fh$ such that  $\fh$ is the free pro-nilpotent graded Lie algebra over $\g$ on generators $V$.
\end{definition}

\begin{proposition}[Minimal models]\label{dgminimal}
For every object $\g$ of   $dg\hat{\cN}(R)$, there exists a free chain Lie algebra $\m$  with $d=0$ on  the abelianisation $\m/[\m,\m]$,  unique up to non-unique isomorphism, together with a weak equivalence  $\m \to \g$.
\end{proposition}
\begin{proof}
 \cite[Proposition \ref{htpy-dgminimal}]{htpy}. 
\end{proof}
The significance of this result is that, together with Proposition \ref{meequiv}, it allows us to reformulate Malcev homotopy types in terms of extra structure on cohomology groups, since $(\m/[\m,\m])_n$ is dual to $\H^{n+1}(\g, k)$.

\begin{definition}
Let $dg\cP(R)$ be the category with the same objects as $dg\hat{\cN}_A(R)$, and morphisms given by
$$
\Hom_{dg\cP(R)}(\g,\fh)= \exp(\prod_{x \in \Ob R}\H_0\fh(x))\by^{\exp(\fh^R_0)}\Hom_{\Ho(dg\hat{\cN}_A(R))}(\g,\fh),
$$
where $\fh_0^R$ (the Lie subalgebra of $R$-invariants in $\fh_0$)  acts by conjugation on the set of homomorphisms. Composition of morphisms is given by $(u,f) \circ (v,g)= (u\circ f(v), f\circ g)$.
\end{definition}

\begin{proposition}\label{nequiv}
There is a normalisation functor $N\co s\hat{\cN}_A(R) \to dg\hat{\cN}_A(R)$ such that
$$
\H_i(N\g) \cong \pi_i(\g),
$$
giving equivalences $\Ho(s\hat{\cN}_A(R))\simeq \Ho(dg\hat{\cN}_A(R))$, and $s\cP_A(R) \simeq dg\cP_A(R)$.
\end{proposition}
\begin{proof}
This is essentially \cite[Propositions \ref{htpy-nequiv} and \ref{htpy-anequiv}]{htpy}, adapted as in \cite[Theorem \ref{mhs-bigequiv}]{mhs}.
\end{proof}

\subsubsection{Cosimplicial algebras}

\begin{definition}
 Let $c\Alg(R)$ be the category of of $R$-representations in cosimplicial $k$-algebras. 
\end{definition}

\begin{proposition}\label{calgmodel}
There is a   simplicial model category structure on   $c\Alg(R)$, in which a map $f\co A \to B$ is 
\begin{enumerate}
\item a  weak equivalence if  $\H^i(f)\co \H^i(A) \to \H^i(B)$ is an isomorphism in $\Rep(R)$ for all $i$; 
\item a fibration if $f^i(x)\co A^i(x) \to B^i(x)$ is a surjection for all $x \in \Ob(R)$ and all $i$.
\end{enumerate}
\end{proposition}
\begin{proof}
This is \cite[Proposition \ref{calgmodel}]{htpy}, adapting \cite[\S 2.1]{chaff}.  
\end{proof}

\begin{definition}
 Let $c\Alg(R)_*$ be the category of of $R$-representations in cosimplicial $k$-algebras, equipped with an augmentation to $\prod_{x \in \Ob R} O(R)(x,-)$. This inherits a model structure from $c\Alg(R)$.
 Denote the   opposite category by $s\Aff(R)_*=  \coprod_{x \in \Ob R} R(x,-) \da s\Aff(R)$, where the coproduct is taken in the category of affine schemes.
\end{definition}

\begin{definition}
Given representations $V,W \in \Rep(R)$, define $V\ten^RW:=\Hom_{\Rep(R)}(k,V\ten W)$.
\end{definition}

\begin{definition} Given $A \in c\Alg(R)$ and $\g \in s\hat{\cN}(R)$, define the 
 Maurer-Cartan space $\mc(A,G)$ to consist of sets $\{\omega_n\}_{n\ge 0}$, with $\omega_n \in \exp(A^{n+1}\hat{\ten}^R\g_n)$, such that 
\begin{eqnarray*}
\pd_i\omega_n &=& \left\{\begin{matrix} \pd^{i+1}\omega_{n-1}  & i>0 \\ (\pd^1\omega_{n-1})\cdot(\pd^0\omega_{n-1})^{-1} & i=0,\end{matrix} \right.\\
\sigma_i\omega_n &=& \sigma^{i+1}\omega_{n+1},\\
\sigma^0\omega_n&=& 1,
\end{eqnarray*}
where  $\exp(A^{n+1}\hat{\ten}^R\g_n)$ is the group  whose underlying set is the Lie algebra $A^{n+1}\hat{\ten}\g_{n-1}$, with multiplication given by the Campbell--Baker--Hausdorff formula.
\end{definition}

\begin{definition}
Given $A \in c\Alg(R)$ and $\g \in s\hat{\cN}(R)$, define the \emph{gauge group} $\Gg(A,\g)\le \prod_n \exp(A^n\hat{\ten}^R\g_n)$
to consist of those $g$ satisfying
\begin{eqnarray*}
\pd_ig_n &=& \pd^{i}g_{n-1}  \quad \forall i>0, \\
\sigma_ig_n &=& \sigma^{i}g_{n+1} \quad \forall i.
\end{eqnarray*}
 This has an action on $\mc(A,\g)$ given by  
$$
(g*\omega)_n= (\pd_0g_{n+1}) \cdot \omega_n \cdot (\pd^0g_n^{-1}).
$$
\end{definition}

\begin{definition}
Let $c\Alg(R)_{0*}$ be the full subcategory of $c\Alg(R)_*$ whose objects satisfy $\H^0(A) \cong k$. Let $\Ho(c\Alg(R)_{0*})$ be the full subcategory of $\Ho(c\Alg(R)_{0*})$ with objects in $c\Alg(R)_{0*}$. Let $s\Aff(R)_{0*}$ be the category opposite to $c\Alg(R)_{0*}$, and   $\Ho(s\Aff(R)_{0*})$ opposite to $\Ho(c\Alg(R)_{0*})$.
\end{definition}

\begin{definition}
Given a topological space $X$, and a sheaf $\sF$ on $X$, define 
$$
\CC^n(X,\sF):= \prod_{f\co |\Delta^n| \to X} \Gamma(|\Delta^n|, f^{-1}\sF).
$$
Together, these form a cosimplicial complex $\CC^{\bt}(X,\sF)$.
\end{definition}

\subsubsection{Cochain algebras}

\begin{definition}
Define $DG\Alg(R)$ to be the category of $R$-representations in  non-negatively graded cochain $k$-algebras, and let 
 $dg\Aff(R)$  be the opposite  category. 
\end{definition}

\begin{lemma}\label{dgalgmodel}
There is a closed model structure on $DG\Alg(R)$ in which a morphism $f\co A \to B$ is:
\begin{enumerate}
\item a weak equivalence if $\H^i(f)\co \H^i(A) \to \H^i(B)$ is an isomorphism in $\Rep(R)$ for all $i$;

\item a fibration if $f^i\co A^i \to B^i$ is a surjection for all $i$;

\item a cofibration if it has LLP with respect to all trivial fibrations.
\end{enumerate}
\end{lemma}
\begin{proof}
This is standard (see e.g. \cite[Proposition 4.1]{schematicv2}).
\end{proof}

\begin{definition}
Define $DG\Alg(R)_*$ to be the category of $R$-representations in  non-negatively graded cochain $k$-algebras, equipped with an augmentation to $\prod_{x \in \Ob R} O(R)(x,-)$. This inherits a model structure from $DG\Alg(R)$.
Define $ dg\Aff(R)_*$ to be the category opposite to $DG\Alg(R)_*$. 

Let $DG\Alg(R)_{0*}$ be the full subcategory of $DG\Alg(R)_*$ whose objects $A$ satisfy $\H^0(A)=k$. Let $\Ho(DG\Alg(R)_*)_0$ be the full subcategory of $\Ho(DG\Alg(R)_*)$ on the objects of  $DG\Alg(R)_0$. Let $ dg\Aff(R)_{0*}$ and $\Ho( dg\Aff(R)_*)_0$ be the opposite categories to $DG\Alg(R)_{0*}$ and $\Ho(DG\Alg(R)_*)_0$, respectively. 
\end{definition}

\begin{proposition}\label{affequiv}
There is a denormalisation functor $D\co DG\Alg(R) \to c\Alg(R)$ such that
$$
\H^i(DA) \cong \H^i(A).
$$
This is a right Quillen equivalence, with left adjoint $D^*$, so gives an equivalence  $\Ho(c\Alg(R))\simeq \Ho(DG\Alg(R))$.
\end{proposition}
\begin{proof}
This is \cite[Proposition \ref{htpy-affequiv}]{htpy}.
\end{proof}

\begin{definition}
Given a cochain algebra $A \in DG\Alg(R)$, and a chain Lie algebra $\g \in dg\hat{\cN}(R)$, define the Maurer-Cartan space by 
$$
\mc(A,\g):=\{\omega \in \bigoplus_n A^{n+1}\hat{\ten}^R \g_n \,|\,d\omega+\half[\omega,\omega]=0\}.
$$
\end{definition}

\begin{definition}\label{dgdef}\label{dgdefgauge}
Given $A \in DG\Alg(R)$ and $\g \in dg\hat{\cN}(R)$, we define the gauge group by
$$
\Gg(A,\g):= \exp(\prod_n A^n\hat{\ten}^R\g_n).
$$ 
Define a gauge action of $\Gg(A,\g)$ on $\mc(A,\g)$ by 
$$
g(\omega):= g\cdot \omega \cdot g^{-1} -(dg)\cdot g^{-1}.
$$
\end{definition}

\begin{definition}\label{Th}
Recall that the Thom-Sullivan (or Thom-Whitney) functor $\Th$ from cosimplicial algebras to DG algebras is defined as follows. Let $\Omega(|\Delta^n|)$ be the DG algebra of rational polynomial forms on the $n$-simplex, so
$$
\Omega(|\Delta^n|)=\Q[t_0, \ldots, t_n, dt_0, \ldots, dt_n]/(1-\sum_i t_i, \sum_i dt_i),
$$ for $t_i$ of degree $0$. The usual face and degeneracy maps for simplices yield $\pd_i\co    \Omega(|\Delta^n|) \to \Omega(|\Delta^{n-1}|)$ and  $\sigma_i\co    \Omega(|\Delta^n|) \to \Omega(|\Delta^{n-1}|)$, giving a simplicial complex of DGAs. Given a cosimplicial algebra $A$, we then set
$$
\Th(A):= \{a \in  \prod_n A^n\ten\Omega(|\Delta^n|)\,:\, \pd^i_A a_n = \pd_ia_{n+1},\, \sigma^j_Aa_n= \sigma_ja_{n-1}\, \forall i,j\}.
$$
\end{definition}

The following is a major comparison result, which will be used in  Theorem \ref{qleqhtpy} as the main  step towards reformulating Malcev homotopy types in terms of Godement resolutions.
\begin{theorem}\label{bigequiv}
We have the following commutative diagram of equivalences of categories:
$$
\xymatrix{
\Ho(dg\Aff(R)_*)_0 \ar@<1ex>[r]^{\Spec D} \ar@<1ex>[d]^{\bar{G}}& Ho(s\Aff(R)_*)_0 \ar@<1ex>[l]^{\Spec \Th} \ar@<1ex>[d]^{\bar{G}} \\
dg\cP(R) \ar@<1ex>[u]^{\bar{W}}  & s\cP(R) \ar@<1ex>[u]^{\bar{W}} \ar@<1ex>[l]^{N},
}
$$
with the pair
$$
\xymatrix@1{ \Ho(dg\Aff(R)_*)_0  \ar@<1ex>[r]^-{\bar{G}}  & dg\cP(R) \ar@<1ex>[l]^-{\bar{W}}},
$$
characterised by the property that 
$$
\Hom_{\Ho(dg\Aff(R)_*)}(\Spec A,\bar{W}\g)=\Hom_{dg\cP(R)}(\bar{G}(A), \g)=\mc(A,\g)\by^{\Gg(A,\g)}\prod_{x \in \Ob R}\exp(\H_0\g(x)).
$$
\end{theorem}
\begin{proof}
This is \cite[Theorem \ref{mhs-bigequiv}]{mhs}, which adapts \cite[Corollary \ref{htpy-bigequiv}]{htpy}  to the pointed case. The vertical equivalences come from \cite[Proposition \ref{htpy-wequiv}]{htpy}, while the horizontal equivalences are from \cite[Theorems \ref{htpy-qs}]{htpy}   and Theorem \ref{htpy-defqs} or, for a shorter and more conceptual proof, \cite[Theorem \ref{monad-cfexp}]{monad}.
  The results  of \cite[4.1]{HinSch}  imply that $D$ and $\Th$ are homotopy inverses.
\end{proof}

\begin{definition}\label{OR}
Recall that  $O(R)$ has the natural structure of an $R\by R$-representation. Since  every $R$-representation has an associated semisimple local system on $|BR(k)|$, we will also write $O(R)$ for the $R$-representation in semisimple local systems on $|BR(k)|$  corresponding to the $R\by R$-representation $O(R)$.
We then define the $R$-representation $\bO(R)$ in semisimple local systems on $X$ by $\bO(R):=\rho^{-1}O(R)$.
\end{definition}

\begin{proposition}\label{eqhtpy}
Under the equivalences of Theorem \ref{bigequiv}, the relative Malcev homotopy type $G(X)^{\rho, \mal}$ of a topological space $X$ corresponds to 
$$
\CC^{\bt}(X,\bO(R))\in c\Alg(R),
$$
equipped with its augmentation to $\prod_{x \in X} \CC^{\bt}(x,\bO(R)) \cong \prod_{x \in X} O(R)(x,-)$.
\end{proposition}
\begin{proof}
This is essentially the same as \cite[Theorem \ref{htpy-eqhtpy}]{htpy} (which considers the unpointed case).
\end{proof}

\begin{corollary}\label{eqtoen}
Pro-algebraic homotopy types are equivalent to the schematic homotopy types of \cite{chaff}, in the sense that the full subcategory of the homotopy category $\Ho(s\mathrm{Pr})$ on objects $X^{\sch}$ is equivalent to the full subcategory of $\Ho(s\agpd)$ on objects $G(X)^{\alg}$. Under this equivalence, $X^{\sch}$ is represented by the simplicial scheme $\bar{W}G(X)^{\alg}$, and
pro-algebraic homotopy groups are isomorphic to schematic homotopy groups.
\end{corollary}
\begin{proof}
\cite[Corollary \ref{htpy-eqtoen}]{htpy}.
\end{proof}

\begin{definition}
Given a manifold $X$, denote the sheaf  of real $n$-forms on $X$ by $\sA^n$. Given a real sheaf $\sF$ on $X$, write
$$
A^n(X,\sF):=\Gamma(X,\sF\ten_{\R} \sA^n).
$$ 
\end{definition}

\begin{proposition}\label{propforms}
The real Malcev  homotopy type of a manifold $X$ relative to $\rho\co \pi_fX \to R(\R)$ is given in $DG\Alg(R)$ by the de Rham complex
$
A^{\bt}(X, \bO(R)),
$
equipped with its augmentation to $\prod_{x \in X} A^{\bt}(x,\bO(R)) \cong \prod_{x \in X} O(R)(x,-)$.
\end{proposition}
\begin{proof}
\cite[Proposition \ref{htpy-propforms}]{htpy}.
\end{proof}

\section{Pro-$\Ql$-algebraic homotopy types}\label{algtypes}
The purpose if this section is to transfer the framework of Section \ref{review} to an $\ell$-adic setting, replacing topological spaces with pro-finite spaces (and hence \'etale homotopy types of algebraic varieties).

Fix a prime $\ell$. Although all results here will be stated for the local field $\Ql$, they hold for any  of its algebraic extensions. 
\subsection{Algebraisation of locally pro-finite groupoids}

\begin{definition}\label{alg} 
Given a  pro-groupoid $\Gamma$ with $\Ob(\Gamma)$ a discrete set (in the sense of  Remark \ref{discreterk}), we define the \emph{pro-algebraic completion} $\Gamma^{\alg}$ to be the pro-$\Ql$-algebraic groupoid 
pro-representing the functor
\begin{eqnarray*}
\agpd  &\to& \Set\\
H &\mapsto &\Hom_{\Top\gpd}(\Gamma, H(\Ql)),
\end{eqnarray*}
where $\Top\gpd$ denotes the category of topological groupoids, and $H(\Ql)$ is endowed with the topology induced from $\Ql$. Note that this exists by the Special Adjoint Functor Theorem (\cite[Theorem V.8.2]{mac}), with the algebraic groups $\GL_n$ providing the data for the solution set condition (by Tannakian duality). Given a set of primes $L$, define the $L$-algebraic completion $\Gamma^{L,\alg}$ to be $(\Gamma^{\wedge_L})^{\alg}$. If $P$ is the set of all primes, we simply write $\hat{\Gamma}^{\alg}:=\Gamma^{P, \alg}$.
\end{definition}

\begin{remarks}
Since representations with finite monodromy are algebraic there is a canonical retraction $\Gamma^{L,\alg}\to \Gamma^{\wedge_L}$ of pro-algebraic groupoids. 

The motivating example for this definition is when $\Gamma = \pi_f^{\et}(X)$, the \'etale fundamental groupoid of an algebraic variety.
\end{remarks}

The following definition is a slight generalisation of  \cite[Definition 2.1]{weight1}, and extends Definition \ref{malcevdef} to pro-groupoids:
\begin{definition}\label{malcev}
Given a pro-groupoid $\Gamma$ with $\Ob(\Gamma)$  discrete,
a reductive pro-algebraic groupoid $R$ over $\Ql$, and a Zariski-dense (i.e. essentially surjective on objects and Zariski-dense on morphisms) continuous map
$$
\rho\co \Gamma^{\wedge_L} \to R(\Ql),
$$
where the latter is given the $\ell$-adic topology, we define the \emph{relative Malcev completion} $\Gamma^{L,\rho,\mal}$ (or $\Gamma^{L, R, \mal}$)
   to be the universal diagram
$$
\Gamma^{\wedge_L} \xra{g} \Gamma^{L,\rho,\mal}(\Ql) \xra{f} \tilde{R}(\Ql),
$$
where $\tilde{R}$ is the groupoid equivalent to $R$ on objects $\Ob \Gamma$ (as in Definition \ref{malcevdef}),
with $f\co \Gamma^{L,\rho,\mal}\to \tilde{R}$  a pro-unipotent extension of pro-$\Ql$-algebraic groupoids, $g$ a continuous map of topological groupoids, and their composition equal to $\rho$. 

To see  that this universal object exists, we note that this description determines the linear representations of $\Gamma^{L,\rho,\mal}$  (as described in Remarks \ref{savemalcev}). Since these form a multi-fibred tensor category,  Tannakian duality (\cite[Remark 2.6]{htpy}) then gives a construction of    $\Gamma^{L,\rho,\mal}$.
\end{definition}

\begin{remarks}\label{savemalcev} By considering  groupoid homomorphisms $\Gamma^{\wedge_L} \to \coprod_n \GL_n(\Ql)$, observe that finite-dimensional linear representations of $\Gamma^{L,\alg}$ are just continuous $\Ql$-representations of $\Gamma^{\wedge_L}$. 

Finite-dimensional representations of $\Gamma^{L,\rho,\mal}$ are only those continuous $\Ql$-representations whose semisimplifications are $R$-representations. Moreover, if we let $R$ be the reductive quotient $\Gamma^{L,\red}$ of $\Gamma^{L,\alg}$, then  $\Gamma^{L,\alg}=\Gamma^{L,R,\mal}$.
\end{remarks}

\begin{definition}
 Given  an $n$-dimensional $\Ql$-vector space $V$, a \emph{lattice} $\L$ in $V$ is a rank $n$ $\Zl$-submodule $\L \subset V$.
\end{definition}

\begin{lemma}\label{serrelattice}
If $\Gamma$ is a  pro-finite group, $V$ an $n$-dimensional $\Ql$-vector space, and $\rho\co \Gamma \to \GL(V)$ a continuous representation (where the latter is given the $\ell$-adic topology) then there exists a lattice  $\L \subset V$ such that
$\rho$ factors through $\GL(\L)$.
\end{lemma}
\begin{proof}
Since $\Gamma$ is pro-finite, it is compact, and hence $\rho(\Gamma)\le \GL(V)$ must be compact. 
\cite[LG 4 Appendix 1 Theorems 1 and 2]{Se}  show that every compact subgroup of $\GL(V)$ is contained in a maximal compact subgroup, and that the maximal compact subgroups are of the form $\GL(\L)$. Explicitly, we choose a lattice $\L_0 \subset V$, then set $\L= \sum_{\gamma \in \Gamma} \rho(\gamma) \L_0$ (with compactness ensuring the sum is finite).
\end{proof}

\begin{remark}\label{locsyslrk}        
In particular, when $\Gamma = \pi_f^{\et}(X)$, this means that finite-dimensional  representations of $\Gamma^{\alg}$ are smooth $\Ql$-sheaves on $X$, while finite-dimensional  representations of $\Gamma^{\red}$ are semisimple $\Ql$-sheaves. The Zariski-dense map $\rho\co  \Gamma \to R(\Ql)$ identifies $R$-representations with a full tensor subcategory of semisimple $\Ql$-sheaves, and $\Gamma^{\rho,\mal}$-representations are Artinian extensions of these semisimple sheaves.
\end{remark}

\begin{proposition}\label{gzl}
Given a locally pro-finite groupoid $\Gamma$ with discrete objects (as in Remark \ref{discreterk}), and a Zariski-dense continuous map
$$
\rho\co \Gamma^{\wedge_L} \to G(\Ql)
$$
to a pro-$\Ql$-algebraic groupoid, there is a canonical model $G_{\Zl}$ for $G$ over $\Zl$ for which  $\rho$ factors through a Zariski-dense map
$$
\rho_{\Zl}\co \Gamma^{\wedge_L} \to G_{\Zl}(\Zl).
$$
\end{proposition}
\begin{proof}
Assume that $\rho$ is an isomorphism on objects (replacing $G$ by an equivalent groupoid). Let $\cC$ be the category of continuous $\Gamma$-representations in finite free $\Zl$-modules. For each $x \in \Ob\Gamma$, this gives a fibre functor $\omega_x$ from $\cC$ to finite free $\Zl$-modules. 

If we let $\cD$ be the category of  $\Gamma$-representations in finite-dimensional $\Ql$-vector spaces, with the fibre functors also denoted by $\omega_x$, then the category of $G$-representations is equivalent to a full subcategory $\cD(G)$ of $\cD$, since $\rho$ is Zariski-dense. By Tannakian duality (as in \cite[\S \ref{htpy-gpdsn}]{htpy}), there are isomorphisms
$$
G(x,y)(A) \cong \Iso^{\ten}(\omega_x|_{\cD(G)},\omega_y|_{\cD(G)})(A),
$$
where $\Iso^{\ten}$ is the set of natural isomorphisms of tensor functors.

Now, by Lemma \ref{serrelattice}, the functor $\ten \Ql\co \cC \to \cD$ is essentially surjective.
 Let $\cC(G)$ be the full subcategory of $\cC$ whose objects are those $\L$ for which $\L\ten \Ql$ is isomorphic to an object of $\cD(G)$; these are $\Gamma$-lattices in $G$-representations. Define
$$
G_{\Zl}(x,y)(A):= \Iso^{\ten}(\omega_x|_{\cC(G)},\omega_y|_{\cC(G)})(A),
$$
observing that this is an affine scheme (since it preserves all  limits), with $G_{\Zl}\ten \Ql =G$.

Equivalently, we could set $O(G_{\Zl})\subset O(G)$ to be $\{f\,:\, f(\rho(\gamma)) \in \Zl \quad \forall \gamma \in \Gamma\}$.
\end{proof}

\begin{definition}\label{admiss}
Given a finite-dimensional nilpotent Lie algebra $\fu$ over $\Ql$, equipped with the continuous action of a pro-finite group $\Gamma$ (respecting the Lie algebra structure), we say that a lattice $\L \subset \fu$ is \emph{admissible} if it satisfies the following:
\begin{enumerate}

\item $\L$ is a $\Gamma$-subrepresentation;

\item $\L$ is closed under all the monomials in the Campbell--Baker--Hausdorff formula 
\[ 
 \log(e^a\cdot e^b) = 
\sum_{n>0}\frac {(-1)^{n-1}}{n} 
\sum_{ \substack{r_i + s_i > 0 \\ 1\le i \le n}}
\frac{(\sum_{i=1}^n (r_i+s_i))^{-1}}{r_1!s_1!\cdots r_n!s_n!} 
[a^{r_1} b^{s_1} a^{r_2} b^{s_2} \ldots a^{r_n} b^{s_n} ],
\]
where
\[
[ a^{r_1} b^{s_1} \ldots a^{r_n} b^{s_n} ] = [ \overbrace{a,[a,\ldots[a}^{r_1} ,[ \overbrace{b,[b,\ldots[b}^{s_1} ,\,\ldots\, [ \overbrace{a,[a,\ldots[a}^{r_n} ,[ \overbrace{b,[b,\ldots b}^{s_n} ]]\ldots]],
\] 
understood to be $0$ if $s_n > 1$ or if $s_n = 0$ and $r_n > 1$
\end{enumerate}
\end{definition}

\begin{lemma}\label{admissprop}
If $\L \subset \fu$ is an admissible lattice and $\fu \in \cN$, then the image of $\L$ under the exponential map
$$
\exp\co  \fu \to \exp(\fu)
$$
is a pro-finite subgroup.
\end{lemma}
\begin{proof}
We may regard $\exp(\fu)$ as being the set $\fu$, with multiplication given by the Campbell-Baker-Hausdorff formula (which has only finitely many terms in this case, since $\fu$ is nilpotent). Since $\L$ is closed under all the operations in the  formula, it is closed under multiplication. As $\exp$ is a homeomorphism, $\exp(\L)$ is compact and thus pro-finite.
\end{proof}

\subsection{Pro-$\Ql$-algebraic homotopy types}
We now proceed as in \S \ref{sagpdsn}, extending to a simplicial framework in order to study the loop groupoid (and hence the whole homotopy type), rather than just the fundamental groupoid.

\begin{definition}\label{salg} 
Given a pro-simplicial groupoid $G$ with $\Ob(G)$ a discrete set, we define the \emph{pro-algebraic completion} $G^{L,\alg} \in s\agpd$ to represent the functor
\begin{eqnarray*}
s\agpd  &\to& \Set\\
H &\mapsto &\Hom_{s\Top\gpd}(G^{\wedge_L}, H(\Ql)),
\end{eqnarray*}
where $\Top\gpd$ denotes the category of topological groupoids. Note that Lemma \ref{levelwiseproL} implies that we can compute this levelwise by $(G^{L,\alg})_n = (G_n)^{L,\alg}$.
\end{definition}

\begin{remark}\label{nomodel}
It is natural to ask whether $G \mapsto G^{L, \alg}$ is left Quillen for any suitable model structure on pro-$L$ simplicial groupoids. This cannot be the case, since the functor is not even a left adjoint, essentially because $\Ql$ is not pro-finite. 
\end{remark}

\begin{definition}\label{smalcev}
Given a pro-simplicial groupoid $G$ with $\Ob(G)$  discrete,
a reductive pro-algebraic groupoid $R$ over $\Ql$, and a Zariski-dense continuous map
$$
\rho\co \pi_0(G)^{\wedge_L} \to R(\Ql),
$$
where the latter is given the $\ell$-adic topology, we define the \emph{relative Malcev completion} $G^{L,\rho,\mal} \in s\cE(R) \subset s\agpd \da R$ by $(G^{L,\rho,\mal})_n := (G_n)^{L,\rho\circ a_n,\mal}$, for $a_n\co G_n \to \pi_0G$ the canonical map.

Note that $\pi_0(G^{L,\rho,\mal})= \pi_0(G)^{L,\rho,\mal}$.
\end{definition}

\begin{lemma}\label{admissiblelattice}
If the continuous action of a pro-finite group $\Gamma$ on $\fu_{\bt} \in s\cN_{\Ql}$ is semisimple, then $\fu$ is the union of its $\Gamma$-equivariant simplicial admissible sublattices.
\end{lemma}
\begin{proof}
Since the action of $\Gamma$ is semisimple, we may take a complement $V_{\bt} \subset \fu_{\bt}$ of $[\fu_{\bt},\fu_{\bt}]$ as a simplicial $\Gamma$-representation. Given a lattice $M \subset V$, let $g(M)\subset \fu$ denote the $\Zl$-submodule generated by $M$ and the operations in the Campbell-Baker-Hausdorff formula. Since $\fu$ is nilpotent, it follows that $g(M)$ is a finitely generated $\Zl$-module, and hence a lattice in $\fu$. By semisimplicity  and Lemma \ref{serrelattice}, there exists a $\Gamma$-equivariant lattice $\L_{\bt} \subset V_{\bt}$. The lattices $\ell^{-n}\L_{\bt} \subset V_{\bt}$ are also then $\Gamma$-equivariant for $n\ge 0$, so the lattices $g(\ell^{-n}\L_{\bt})\subset \fu_{\bt}$ are all admissible.

It only remains to show that $\bigcup  g(\ell^{-n}\L) \to \fu$ is a surjective map of Lie algebras. This follows since $ \bigcup  \ell^{-n}\L \to \fu/[\fu,\fu]$ is  surjective. 
\end{proof}

\begin{lemma}\label{cladic} 
Given a compact topological 
space $K$  and a finite-dimensional nilpotent $\Ql$-Lie algebra $\fu$, the map
$$
\Hom_{\cts}(K, \Zl)\ten_{\Zl} \fu \to \Hom_{\cts}(K, \fu)
$$
is an isomorphism.
\end{lemma}
\begin{proof}
First observe that the map is clearly injective, since $\fu$ is a flat $\Zl$-module. For surjectivity, note that the image of $f\co  K \to \fu$ must be contained in an admissible sublattice $\L \subset \fu$ (by compactness and Lemma \ref{admissiblelattice}). Now,
$$
\Hom_{\cts}(K, \L)\cong \Hom_{\cts}(K, \Zl)\ten_{\Zl}\L,
$$
since $\L$ is a finite free $\Zl$-module. 
\end{proof}

\begin{definition}
Given a continuous representation $V$ of $\widehat{\pi_fX}$ in $\Ql$-vector spaces,  recall the standard definition that 
$$
\H^*(X,V):= \H^*(X,\L)\ten_{\Zl}\Ql,
$$
for any $\pi_fX$-equivariant  $\Zl$-lattice $\L \subset V$ as in Lemma \ref{serrelattice}, and  $\H^*(X,\L)$ as in Definition \ref{prospaceprocohodef}.
\end{definition}

\begin{remark}\label{discretecoho}
If $X$ is discrete, note that this is not in general the same as cohomology $\H^n(X,V^{\delta})$ of the discrete $\pi_fX$-representation $V^{\delta}$ underlying $V$. However, both will coincide if   $\H_n(G,\L^{\vee})$ has finite rank, by the Universal Coefficient Theorem and Lemma \ref{ctscohodiscrete}. 
\end{remark}

\begin{example}\label{ethtpy}
If $X$ is a locally Noetherian simplicial scheme, we may consider the \'etale topological type $X_{\et} \in \pro(\bS)$, as defined in \cite[Definition 4.4]{fried}. Since $(X_{\et})_0$ is the set of geometric points of $X_0$, we may then apply the constructions of this section. For a finite local system $M$ on $X$, we have
$$
\H^*(X_{\et}, M)\cong \H^*_{\et}(X,M),
$$
by \cite[Proposition 5.9]{fried}. For an inverse system $M=\{M_i\}$ of local systems, we have
$$
\H^*(X_{\et}, M)= \H^*(\Lim_i\CC^{\bt}_{\et}(X,M_i))= \H^*_{\et}(X, (M)),
$$
where $\CC^{\bt}_{\et}$ is a variant of the Godement resolution and $ \H^*_{\et}(X, (M))$ is Jannsen's continuous \'etale cohomology (\cite{jancts}). If the groups  $\H^*_{\et}(X, M_i)$ satisfy the Mittag-Leffler condition (in particular, if they are finite), then 
$$
\H^*(X_{\et}, M)\cong\Lim_i \H^*_{\et}(X,M_i).
$$

\cite[Theorem 7.3]{fried}  shows that $X_{\et}\in \hat{\bS}$ whenever the schemes $X_n$ are connected and geometrically unibranched. It seems that this result can be extended to simplicial schemes (or even simplicial algebraic spaces) for which the homotopy groups $\pi_m^{\et}(X_n)$ satisfy the $\pi_*$-Kan condition (\cite[\S IV.4]{sht}), provided the simplicial set $\pi(X)_{\bt}$, given by $\pi(X)_n:= \pi(X_n)$, the set of connected components of $X_n$,  has finite homotopy groups.  
\end{example}

\begin{proposition}\label{algcoho}
Take $X \in \pro(\bS)$ with $X_0$ discrete, and a Zariski-dense continuous map
$$
\rho\co \pi_f(X)^{\wedge_L} \to R(\Ql),
$$
for $\ell \in L$, with $\Ob R  = \Ob \pi_f(X)$. Then $G(X)^{L, \rho, \mal}$ is cofibrant (for the model structure of Lemma \ref{cmsagpd}), the map  $G(X)^{L', \rho, \mal}\to G(X)^{L, \rho, \mal}$ is an isomorphism for all $L \subset L'$, and 
$$
\H^*(G(X)^{L, \rho, \mal},V) \cong \H^*(X, \rho^*V).
$$
\end{proposition}
\begin{proof}
Let $\Delta \le R(\Ql)$ be the image of $\rho$. Write $\{X_{\alpha}\}_{\alpha \in \bI}$ for the inverse system $X$.
For $\fu \in s\cN(R)$, 
$$
\Hom_{s\Top\gpd}(G(X)^{\wedge_L}, \exp(\fu)\rtimes R)_{R}= \Hom_{s\Top\gpd}(G(X)^{\wedge_L}, \exp(\fu)\rtimes \Delta)_{\Delta}.
$$
Since $\fu \in s\cN(R)$, the normalisation $N\fu$ is bounded in degrees $\le n$, say. This implies that $\fu= \cosk_{n+1}\fu$, the $n+1$-coskeleton, or equivalently that any simplicial morphism $Y \to \fu$ is determined by the maps $Y_i \to \fu_i$ for $i \le n+1$.

Thus any  morphism $f\co  G(X)^{\wedge_L}\to \exp(\fu)\rtimes \Delta $ is determined by the maps $f_i \co  G(X)^{\wedge_L}_i\to \exp(\fu_i)\rtimes \Delta$ for $i \le n+1$. 
 Now, by Lemma \ref{admissiblelattice}, $\exp(\fu)\rtimes \Delta$ is the union over all admissible $\Delta$-equivariant sublattices $\L \subset \fu$ of $\exp(\L)\rtimes \Delta $.
Since each $G(X)^{\wedge_L}_i$ is compact, its image in  $\exp(\fu_i)\rtimes \Delta$ must be contained in $\exp(\L_i)\rtimes \Delta$ for some admissible $\L\subset \fu$.  By choosing $\L$ large enough that this holds for all $i \le n+1$, we see that
\begin{eqnarray*}
\Hom_{s\Top\gpd}(G(X)^{\wedge_L}, \exp(\fu)\rtimes R)_{R}&=& \lim_{\substack{\lra \\ \L \subset \fu  \text{ admissible}}}\Hom_{s\Top\gpd}(G(X)^{\wedge_L}, \exp(\L)\rtimes \Delta)_{\Delta}\\
&=& \lim_{\substack{\lra \\ \L \subset \fu  \text{ admissible}}}\Hom_{s\pro(\gpd^L)}(G(X)^{\wedge_L}, \exp(\L)\rtimes \Delta)_{\Delta},
\end{eqnarray*}
because $\pro(\gpd^L)$ is a full subcategory of $\Top\gpd$. Here, $\exp(\L)\rtimes \Delta \in \pro(s\gpd^L) $ denotes the pro-object $\{ (\exp(\L)/\exp(\ell^m\L))\rtimes \Delta\}_m$. From now on, we will abuse notation by writing $\exp(\L/\ell^n\L)$ or even $\exp(\L/\ell^n)$ for the finite group $\exp(\L)/ \exp(\ell^n\L)$.

Now, since $\L= \cosk_{n+1}\L$, any morphism $f\co H \to \exp(\L/\ell^m\L )\rtimes \Delta $ is determined by the maps $f_i$ for $i \le n+1$. As  $\exp(\L/\ell^m\L)$ is levelwise finite, and filtered colimits commute with finite limits, this means that 
$$
\Hom_{s\pro(\gpd^L)}(G(X)^{\wedge_L}, \exp(\L)\rtimes \Delta)_{\Delta}= \Hom_{\pro(s\gpd^L)}(G(X)^{\wedge_L}, \exp(\L)\rtimes \Delta)_{\Delta}.
$$
Hence
$$
\Hom_{s\Top\gpd}(G(X)^{\wedge_L}, \exp(\fu)\rtimes R)_{R}= \lim_{\substack{\lra \\ \L \subset \fu  \text{ admissible}}}\Hom_{\pro(s\gpd^L)}(G(X), \exp(\L)\rtimes \Delta)_{\Delta}.
$$

Under the adjunction $G \dashv \bar{W}$, 
this becomes  
$$
\lim_{\substack{\lra \\ \L \subset \fu  \text{ admissible}}}\Hom_{\pro(\bS)}(X, \bar{W}(\exp(\L)\rtimes \Delta))_{\bar{W}\Delta}.
$$ 
This expression is independent of $L$, so we have shown that $G(X)^{L', \rho, \mal}\to G(X)^{L, \rho, \mal}$ is an isomorphism for all $L \subset L'$.

For $p\co \fu \to \fv$ an acyclic small extension with kernel $I$ in $s\cN(R)$, and an admissible lattice $\L' < \fu$, consider the map $\L' \to p(\L')$. This is surjective, and $\H_*(\L'\cap I)\ten \Ql=0$, since $ (\L'\cap I)\ten \Ql \cong I$. As $\H_*(I)=0$, we may choose a $\Delta$-equivariant lattice $\L'\cap I < M < I$ such that $\H_*(M/\ell M)=0$. Let $\L:=\L'+M$, noting that this is an admissible lattice ($p$ being small), with the maps $\L/\ell^n \to p(\L)/\ell^n$ all acyclic.

In order to show that $G(X)^{L, \rho, \mal}$ is cofibrant, take an arbitrary map 
$f\co G(X)^{\wedge_L}\to \exp(\fv)\rtimes \Delta$ over $\rho$; this must factor through $\exp(p(\L'))$ for some admissible lattice $\L' < \fu$, and we may replace $\L$ by $\L'$ as above. It therefore suffices to show that the corresponding map
$$
f\co  X \to \bar{W}( \exp(p(\L))\rtimes \Delta)
$$ 
in $\pro(\bS)$ lifts to $\bar{W}( \exp(\L)\rtimes \Delta)$. For each $n\in \N$, we have a map
$$
f_n\co  X_{\alpha(n)} \to  \bar{W}( \exp(p(\L)/\ell^n)\rtimes \Delta),
$$
and these are compatible with the structural morphisms. 

We now prove existence of the lift by induction on $n$. Assume we have $g_n\co  X_{\alpha(n)} \to  \bar{W}( \exp(\L/l^n)\rtimes \Delta)$, such that $p \circ g_n =f_n$. This gives a map
$$
(f_{n+1}, g_n)\co  X_{\alpha(n)} \to\bar{W}(\exp((p(\L)/\ell^{n+1})\by_{p(\L)/\ell^n}(\L/\ell^n))\rtimes \Delta).
$$
However, $\L/\ell^{n+1} \to (p(\L)/\ell^{n+1})\by_{p(\L)/\ell^n}(\L/\ell^n)$ is an acyclic small extension, so 
$$
\bar{W}(\exp(\L/\ell^{n+1})\rtimes \Delta)\to   \bar{W}(\exp((p(\L)/\ell^{n+1})\by_{p(\L)/\ell^n}(\L/\ell^n))\rtimes \Delta)
$$
is a trivial fibration, allowing us to construct a lift $g_{n+1}\co  X_{\alpha(n+1)} \to  \bar{W}( \exp(\L/\ell^{n+1})\rtimes \Delta)$. This completes the proof that $G(X)^{L, \rho, \mal}$ is cofibrant.

Finally, if $V$ is an $R$-representation then $\H^{n+1}(G(X)^{L,\rho, \mal},V)$ is  the coequaliser of the diagram
$$
\xymatrix@1{
\Hom_{s\agpd\da R}(G(X)^{L,\rho, \mal}, (N^{-1}V[-n])^{\Delta^1})\ar@<.5ex>[r] \ar@<-.5ex>[r] & \Hom_{s\agpd\da R}(G(X)^{L,\rho, \mal}, N^{-1}V[-n]).
} 
$$
For a $\Delta$-equivariant lattice $\L \subset V$, this is the direct limit over $m$ of 
$$
\xymatrix@1{
\Hom_{\pro(\bS\da \bar{W}R)}(X, \bar{W}((N^{-1}\ell^{-m}\L[-n])^{\Delta^1} \rtimes R)) \ar@<.5ex>[r] \ar@<-.5ex>[r] &  \Hom_{\pro(\bS\da \bar{W}R)}(X, \bar{W}(N^{-1} \ell^{-m}\L[-n] \rtimes R)).
} 
$$

Hence 
$$
\H^{n+1}(G(X)^{L,\rho, \mal},V)\cong \lim_{\substack{\lra \\ m}}\H^{n+1}(X,l^{-m}\L)= \H^{n+1}(X,\L)\ten \Ql= \H^{n+1}(X, V),
$$
as required.
\end{proof}

\begin{definition}\label{relmaldef}
Given $X$ and $\rho$ as above, define the \emph{relative Malcev homotopy type}
$$
X^{\rho, \mal}:= G(X)^{P,\rho, \mal},
$$
where $P$ is the set of all primes, noting that this is isomorphic to $G(X)^{L,\rho, \mal}$ for all $L \ni \ell$, by Proposition \ref{algcoho} and Lemma \ref{detectweak}.

Define
$$
X^{L, \alg}:= G(X)^{L,\alg}.
$$
\end{definition}

\begin{remark}
Note that if $X \in \bS$, this definition of Malcev completion differs slightly  from the Malcev homotopy type $X^{\rho, \mal}$ of Definition \ref{hrelmaldef}, which is given by $G(X)^{\rho, \mal}$. However, the following lemma rectifies the situation.
\end{remark}

\begin{lemma}\label{malagrees}
For $X \in \bS$ and $\rho\co  \pi_f(X)^{\wedge_L} \to R(\Ql)$ Zariski-dense and continuous, there is a canonical map
$$
G(X)^{\rho,\mal} \to G(X)^{L, \rho,\mal};
$$
this is a quasi-isomorphism whenever the groups $\H^n(X, V)$ are finite-dimensional for all finite-dimensional $R$-representations $V$.
\end{lemma}
\begin{proof}
Existence of the map is immediate. To see that it gives a quasi-isomorphism, Lemma \ref{detectweak} shows that we need only look at cohomology groups. Given an $R$-representation $V$ corresponding to a  local system $\vv$ over $\Ql$ on $X$, the map on cohomology groups is
$$
\H^*(X^{\wedge_L}, \vv) \to \H^*(X, \vv);
$$ 
 this is an isomorphism by Remark \ref{discretecoho}.
\end{proof}

\begin{definition}\label{varpi}
Define \emph{pro-algebraic} (or \emph{schematic}) and \emph{relative homotopy groups} by $\varpi_n(X^{\wedge_L}):= \pi_{n-1}(G(X)^{L, \alg})$ and $\varpi_n(X^{\rho, \mal}):= \pi_{n-1}(G(X)^{P,\rho, \mal})$.

Define pro-algebraic (or schematic) and relative fundamental groupoids by
$\varpi_f(X^{\wedge_L}):= \pi_f(X)^{L, \alg}$   and $\varpi_f(X^{\rho, \mal}):= \widehat{\pi_fX}^{\rho, \mal}$.

Define $\varpi_f(\hat{X}), \varpi_n(\hat{X})$ by the convention that $\hat{X}=X^{\wedge_{P}}$, for $P$ the set of all primes.
\end{definition}

Note that Lemma \ref{serrelattice} implies that for a locally Noetherian scheme $X$, finite-dimensional $\varpi_f(X_{\et}^{\wedge_L})$-representations correspond to smooth $\Ql$-sheaves on $X$.

The following now follow immediately from Lemma \ref{detectweak}
\begin{corollary}\label{detectweaket}
A map $f\co X \to Y$ in $\pro(\bS)$, with $X_0,Y_0$ discrete, induces an isomorphism
$$
f^{L, \alg}\co  X^{L, \alg}\to Y^{L,\alg}
$$ 
of homotopy types if and only if the following conditions hold:
\begin{enumerate}
\item $f^*$ induces an equivalence between the categories of finite-dimensional semisimple continuous $\Ql$-representations of $(\pi_fX)^{\wedge_L}$ and $(\pi_fY)^{\wedge_L}$;
\item for all finite-dimensional semisimple continuous $\Ql$-representations $V$ of $\pi_fY$, the maps
$$
f^*\co  \H^*(Y, V)\to \H^*(X,f^*V)
$$
are  isomorphisms.
\end{enumerate}
\end{corollary}

\begin{corollary}
Take a map $f\co X \to Y$ in $\pro(\bS)$, with $X_0,Y_0$ discrete, and with a Zariski-dense morphism $\rho\co  (\pi_f Y)^{\wedge_L} \to R(\Ql)$  such that $\rho \circ f \co (\pi_fX)^{\wedge_{L}}  \to R(\Ql)$ is also Zariski-dense.  Then $f$ induces an isomorphism
$$
f^{R, \mal}\co  X^{R, \mal}\to Y^{R, \mal}
$$ 
of homotopy types if and only if  for all $R$-representations $V$,  the maps
$$
f^*\co  \H^*(Y, \rho^*V)\to \H^*(X,f^*\rho^*V)
$$
are  isomorphisms.
\end{corollary}

\subsection{Equivariant cochains}
Proposition \ref{propforms} showed how the schematic homotopy type of a manifold can be recovered from the de Rham complex with local system coefficients. We will now establish an analogue for algebraic varieties, involving an \'etale Godement resolution with coefficients in smooth $\Ql$-sheaves.

\begin{lemma}\label{wworks}
If $\L$ is a 
$\Gamma$-representation in pro-simplicial groups such that $\L \rtimes \Gamma \in \pro(s\gpd)$, 
then
$$
\Hom_{\Gamma,\pro(\bS)}(\widetilde{X}, \bar{W}\L) \cong \Hom_{\pro(\bS) \da B\Gamma}(X, \bar{W}(\L \rtimes \Gamma)),
$$
for $\widetilde{X} $ as in Definition \ref{widetildedef}.
\end{lemma}
\begin{proof}
The calculation is essentially the same as for \cite[Lemma \ref{htpy-wworks}]{htpy}.
\end{proof}

\begin{definition}\label{CCdefind}
Given an ind-finite rank  $\Zl$-local system (i.e. a  filtered direct system in the category of finite rank $\Zl$-local systems) $\vv= \{\vv_{\alpha}\}_{\alpha}$, define 
$$
\CC^{\bt}(X, \vv):= \LLim_{\alpha} \CC^{\bt}(X, \vv_{\alpha}),
$$
where the right-hand side is given in Definition \ref{CCdef}.
\end{definition}

\begin{definition}
Given a pro-algebraic groupoid $G$ over $\Zl$, define $O(G)$ to be the $G\by G$-representation given by global sections of the structure sheaf of $G$, equipped with its left and right $G$-actions.

Given  a representation $\rho\co \pi_fX \to G(\Zl)$, let $\bO(G)$ be the $G$-representation in (ind-finite rank)  $\Zl$-local systems on $X$ given by pulling $O(G)$ back along its right $G$-action.
\end{definition}

\begin{definition}
Given $X, L,\rho, R$ as in Proposition \ref{algcoho}, let $R_{\Zl}$ be the $\Zl$-model for $R$ constructed in Proposition \ref{gzl}, and set
$$
\CC^{\bt}(X, \bO(R)):= \CC^{\bt}(X, \bO(R_{\Zl}))\ten_{\Zl} \Ql.
$$
\end{definition}

\begin{theorem}\label{qleqhtpy}
For $X, L,\rho, R$ as in Proposition \ref{algcoho}, the relative Malcev homotopy type
$$
G(X)^{L, \rho, \mal} \in s\agpd \da R
$$
corresponds under the equivalences of Proposition \ref{meequiv} and Theorem \ref{bigequiv} to the  $R$-representation
$$
\CC^{\bt}(X, \bO(R))
$$
in cosimplicial $k$-algebras, equipped with its natural augmentation to $\prod_{x \in X_0}\CC^{\bt}(x, \bO(R)) = \prod_{x \in \Ob R} O(R)(x,-)$.
\end{theorem}
\begin{proof}
We need to show that, for $\fu \in s\cN(R)$, 
$$
\Hom_{s\agpd \da R}(G(X)^{L, \rho, \mal}, \exp(\fu) \rtimes R) \cong \Hom_{s\Aff(R)}(\Spec \CC^{\bt}(X, \bO(R)), \bar{W}(\exp(\fu))).
$$

Adapting the proof of Proposition \ref{algcoho}, we know that
$$
\Hom_{s\agpd \da R}(G(X)^{L, \rho, \mal}, \exp(\fu) \rtimes R)\cong 
\lim_{\substack{\lra \\ \L }}  \Hom_{\pro(\bS)}(X, \bar{W}(\exp(\L)\rtimes R_{\Zl}(\Zl)))_{BR_{\Zl}(\Zl)},
$$
where the limit is taken over $\L\subset \fu  \text{ admissible}$.
By Lemma \ref{wworks},
$$
\Hom_{\pro(\bS)}(X, \bar{W}(\exp(\L)\rtimes R_{\Zl}(\Zl)))_{BR_{\Zl}(\Zl)}\cong \Hom_{R_{\Zl}(\Zl), \pro(\bS)}(\widetilde{X}, \bar{W}\exp(\L)).
$$

If we regard $\exp(\L)$ as the $\Zl$-valued points of the group scheme $\exp(\L)(A):= \exp(\L\ten A)$, then this is an affine space, so
$$
\Hom_{ \pro(\bS)}(\widetilde{X}, \bar{W}\exp(\L))\cong \Hom_{s\Aff_{\Zl}}(\Spec \CC^{\bt}(\widetilde{X}, \Zl), \bar{W}\exp(\L)).
$$
Since  $\L \cong \L \ten^R_{\Zl}O(R_{\Zl})$, we then have
$$
\Hom_{R_{\Zl}(\Zl), \pro(\bS)}(\widetilde{X}, \bar{W}\exp(\L))\cong \Hom_{s\Aff(R_{\Zl})}(\Spec \CC^{\bt}(X, \bO(R_{\Zl})), \bar{W}\exp(\L)).
$$

The map 
\begin{eqnarray*}
\lim_{\substack{\lra \\ \L }}\Hom_{s\Aff(R_{\Zl})}(\Spec \CC^{\bt}(X, \bO(R_{\Zl})), \bar{W}\exp(\L))\\
\to \lim_{\substack{\lra \\ \L }}\Hom_{s\Aff(R_{\Zl})}(\Spec \CC^{\bt}(X, \bO(R_{\Zl}))\ten \Ql, \bar{W}\exp(\L))
\end{eqnarray*}
is clearly injective. However, since there exists an admissible lattice $\L'$ with $l^{-n} \L \subset \L'$, the map must also be surjective. Finally, note that
\begin{eqnarray*}
\Hom_{s\Aff(R_{\Zl})}(\Spec \CC^{\bt}(X, \bO(R_{\Zl}))\ten \Ql, \bar{W}\exp(\L))\\
= \Hom_{s\Aff(R)}(\Spec \CC^{\bt}(X, \bO(R)), \bar{W}\exp(\L\ten\Ql)),
\end{eqnarray*}
as required.
\end{proof}

\begin{remarks}\label{cfolsson}
 We could use  Proposition \ref{affequiv} to replace $\CC^{\bt}(X, \bO(R))$ with a DG algebra, giving a more reassuring  analogue of the de Rham algebra used in Proposition \ref{propforms} to govern relative Malcev homotopy types of manifolds. This is the approach taken in \cite{olssonhodge}, and when $R=1$, it corresponds  to Deligne's $\Ql$-homotopy type (\cite[\S V]{Weil2}). However, in the sequel we will work systematically with cosimplicial  rather than DG objects --- both approaches being equivalent, the transfer can add unnecessary complication.     

Note that if we take a scheme $X$, then Proposition \ref{eqhtpy} adapts to show that  $\CC^{\bt}(X_{\et}, \vv)$ is a Godement resolution for the continuous \'etale cohomology of $\vv$.
Under the comparison  of Corollary \ref{eqtoen}, this shows that for an algebraic variety $X$,  $\widehat{G(X_{\et})}^{\alg}$  agrees with the $\ell$-adic homotopy type discussed in \cite[\S 3.5.3]{chaff}.

Given any morphism $\rho\co \varpi_f(\widehat{X_{\et}})^{\red} \to R$ to a reductive group, there is a forgetful functor $\rho^{\sharp}\co  s\hat{\cN}(R)\to s\hat{\cN}(\varpi_f(\widehat{X_{\et}})^{\red})$. If we write $\bL\rho_{\sharp}$ for the derived left adjoint and $\rho$ is surjective, then $\Ru( \widehat{G(X_{\et})}^{\rho, \mal})=\bL\rho_{\sharp}\Ru( \widehat{G(X_{\et})}^{\alg})$.
Note that for $\cC$ a Tannakian subcategory (see Definition \ref{tann}) of $\FD\Rep(\varpi_f(\widehat{X_{\et}})^{\red})$, with corresponding groupoid $G$,   the homotopy type $X_{\cC_{\et}}$ of  \cite[1.5]{olssonhodge}  is equivalent to   $ \bL\rho_{\sharp}\Ru( \widehat{G(X)}^{\alg}) $, for  $\rho\co   \varpi_f(\widehat{X_{\et}})^{\red}\to G$. 
\end{remarks}

\subsection{Completing fibrations}

Observe that the definitions and results of  \S \ref{gpcoho} extend naturally to pro-groupoids and pro-spaces; we will make use of this extension without further comment.

\begin{theorem}\label{lfibrations}
Take a pro-fibration $f\co (X,x) \to (Y,y)$ of connected objects in $\pro(\bS)$ with connected fibres,  and set $F:= f^{-1}(y)$.
Take a Zariski-dense representation $\rho\co  \pi_1(X,x) \to R(\Ql)$ to a  reductive pro-algebraic group $R$, let $K$ be the Zariski closure of $\rho(\pi_1(F,x))$, and set $T:= R/K$. If the monodromy action of $\pi_1(Y,y)$ on $\H^*(F, V)$ factors through $\varpi_1(Y,y)^{T, \mal}$ for all $K$-representations $V$, then $G(F,x)^{K,\mal}$ is the homotopy fibre of $ G(X,x)^{R, \mal} \to G(Y,y)^{T, \mal}$.

In particular, there is a long exact sequence
\begin{eqnarray*}
\ldots \to \varpi_n(F,x)^{K,\mal} \to \varpi_n(X,x)^{R,\mal}\to \varpi_n(Y,y)^{T,\mal} \to \varpi_{n-1}(F,x)^{K,\mal}\to \\
\ldots \to \varpi_1(F,x)^{K,\mal} \to \varpi_1(X,x)^{R,\mal}\to \varpi_1(Y,y)^{T,\mal} \to 1.
\end{eqnarray*}
\end{theorem}
\begin{proof}
We adapt the proof of Theorem \ref{fibrations}.

First observe that  $\rho(\pi_1(F,x))$ is normal in $\pi_1(X,x)$, so $K$ is normal in $R$, and $T$ is therefore a reductive pro-algebraic group, so $(Y,y)^{T, \mal}$ is well-defined. Next, observe that since $K$ is normal in $R$, 
 $\Ru(K)$ is also normal in $R$, and is therefore $1$, ensuring that $K$ is reductive, so $(F,x)^{K,\mal}$ is also well-defined. 

Consider the complex $O(R)\ten_{O(T)}O(G(Y,y)^{T, \mal})$ of $ G(X,x)^{R, \mal}$-representations, regarded as a  cosimplicial $G(X,x)$-representation. Since $G(F,x)\to \ker(G(X,x) \to G(Y,y))$ is a weak equivalence,  
the Hochschild-Serre spectral sequence for $f$ (Proposition \ref{lerayserre})  
 with coefficients in this complex is 
 \begin{eqnarray*}
E_2^{i,j}=&\bH^i(G(Y,y), \H^j(F, O(R))\ten_{O(T)}O(G(Y,y)^{T, \mal}))\\
&\abuts \bH^{i+j}(G(X,x), O(R)\ten_{O(T)}O(G(Y,y)^{T, \mal})).
 \end{eqnarray*}
 
Regarding $O(R)$ as a $K$-representation,   $\H^*(F, O(R))$ is a $\varpi_1(Y,y)^{T, \mal}$-representation by hypothesis. Hence
$
 \H^*(F, O(R))\ten_{O(T)}O(G(Y,y)^{T, \mal})
$
is a cosimplicial $G(Y,y)^{T, \mal}$-representation, so
$$
\bH^i(G(Y,y), \H^j(F, O(R))\ten_{O(T)}O(G(Y,y)^{T, \mal}))\cong \bH^i(G(Y,y)^{T, \mal},  \H^j(F, O(R))\ten_{O(T)}O(G(Y,y)^{T, \mal})),
$$
by Lemma \ref{hypercohogood}.

Now, $\H^*(F, O(R))\ten_{O(T)}O(G(Y,y)^{T, \mal})$ is a fibrant cosimplicial $G(Y,y)^{T, \mal}$-representation, so
\begin{eqnarray*}
&&\bH^i(G(Y,y)^{T, \mal},  \H^j(F, O(R))\ten_{O(T)}O(G(Y,y)^{T, \mal}))\\
&&\cong  \H^i\Gamma(G(Y,y)^{T, \mal},  \H^j(F, O(R))\ten_{O(T)}O(G(Y,y)^{T, \mal}))\\
&&=  \left\{ \begin{matrix} \H^j(F, O(R))\ten_{O(T)}k = \H^j(F, O(K))& i=0 \\ 0 & i \ne 0, \end{matrix} \right.\\
\end{eqnarray*}
so
$$
\bH^j(G(X,x), O(R)\ten_{O(T)}O(G(Y,y)^{T, \mal}))\cong \H^j(F, O(K)).
$$

Now, let $\cF$ be the homotopy fibre of $ G(X,x)^{R, \mal} \to G(Y,y)^{T, \mal}$  (which is just the kernel as this map is surjective), noting that there is a natural map $G(F,x)^{K,\mal} \to \cF$.  Lemma \ref{hypercohogood} implies that 
$$
\bH^j(G(X,x), O(R)\ten_{O(T)}O(G(Y,y)^{T, \mal}))=\bH^j(G(X,x)^{R, \mal}, O(R)\ten_{O(T)}O(G(Y,y)^{T, \mal})), 
$$
and \cite[Theorem 1.51]{htpy}  gives a Hochschild--Serre spectral sequence
\begin{eqnarray*}
& \bH^i(G(Y,y)^{T, \mal}, \H^j(\cF, O(R))\ten_{O(T)}O(G(Y,y)^{T, \mal}))\\ &\abuts \bH^{i+j}(G(X,x)^{R, \mal}, O(R)\ten_{O(T)}O(G(Y,y)^{T, \mal})).
\end{eqnarray*}

The reasoning above adapts to show that this spectral sequence also collapses, yielding
$$
\H^j(\cF, O(K)) = \bH^j(G(X,x), O(R)\ten_{O(T)}O(G(Y,y)^{T, \mal})).
$$

We have therefore shown that the map $G(F,x)^{K,\mal} \to \cF $ gives an isomorphism
$$
\H^*(\cF, O(K)) \to \H^*(G(F,x)^{K,\mal}, O(K)),
$$
and hence isomorphisms $\H^*(\cF, V) \to \H^*(G(F,x)^{K,\mal}, V) $ for all $K$-representations $V$.
Since this is a morphism of simplicial pro-unipotent extensions of $K$, \cite[Corollary \ref{htpy-detectweak}]{htpy}  implies that $G(F,x)^{K,\mal} \to \cF$ is a weak equivalence.
\end{proof}

\begin{examples}\label{fibegs}
 Note that we can apply this theorem to $f_{\et}\co X_{\et} \to Y_{\et}$ whenever  $f\co X \to Y$ is  geometric fibration in the sense of \cite[Definition 11.4]{fried}. This  includes smooth projective morphisms, as well as smooth quasi-projective morphisms where the divisor is transverse to $f$. The fibre of $f_{\et}$ over $y$ will then be equivalent to  $(f^{-1}\{y\})_{\et}$.  

Another source of examples comes from nerves of pro-finite groups. Any surjection $g\co \Gamma \to \Delta$ of pro-finite groups gives a pro-fibration $B\Gamma \to B\Delta$, with fibre $B(\ker g)$.

Of course, even if $f\co  X \to Y$ is not a pro-fibration, we can take a fibrant replacement. This will have connected fibres if and only if $\pi_1(X,x) \to \pi_1(Y,y)$ is surjective, and the theorem then describes the homotopy fibre of $f$.
\end{examples}

\subsection{Comparison with Artin--Mazur homotopy groups}\label{arma}

\begin{lemma}\label{algetpiIMlemma}
 Let $f\co X \to Y$ be a morphism in $\pro(\bS)_{\delta}$ for which the map
\[
 \pi_n(f): \pi_n(X) \to \pi_n(Y) 
\]
is a pro-isomorphism for $n \le N$ and a pro-surjection for $n=N+1$, and take a continuous Zariski-dense morphism $\rho\co \pi_fY \to R(\Ql)$.
Then the map
\[
 \varpi_n(f): \varpi_n(X,\rho\circ f)^{\mal}  \to \varpi_n(Y, \rho )^{\mal}
\]
is an  isomorphism for $n \le N$ and a surjection for $n=N+1$.
\end{lemma}
\begin{proof}
 The proof of Lemma \ref{algpiIMlemma} carries over to this generality.
\end{proof}

\begin{definition}\label{relgood2}
By analogy with Definition \ref{relgood}, say that a locally pro-discrete groupoid $\Gamma$ 
is $n$-\emph{good} with respect to a continuous Zariski-dense representation $\rho\co  \Gamma \to R(\Ql)$ to a reductive pro-algebraic groupoid if for all finite-dimensional $\Gamma^{\rho, \mal}$-representations $V$,
the map
$$
\H^i(\Gamma^{\rho, \mal}, V) \to \H^i(\Gamma, V)
$$  
is an isomorphism for all $i\le n$ and and an inclusion for $i=n+1$. Say that $\Gamma$ is \emph{good} with respect to $\rho$ if it is $n$-good for all $n$.

If $\Gamma$ is ($n$-)good relative to $\Gamma^{\red}$, then we say that $\Gamma$ is algebraically ($n$-)good.
\end{definition}

\begin{lemma}\label{goodh}
%
A pro-groupoid $\Gamma$ is $N$-good with respect to $\rho$ if and only if for any finite-dimensional $\Gamma^{\rho, \mal}$-representation $V$, and $\alpha \in \H^n(\Gamma, V)$ for $n\le N$, there exists an injection $f\co V \to W_{\alpha}$ of finite-dimensional $\Gamma^{\rho, \mal}$-representations, with $f(\alpha)=0 \in \H^n(\Gamma, W_{\alpha})$.
\end{lemma}
\begin{proof}
This is a special case of the results of \cite[\S \ref{heid-relgoodsn}]{heid}, which adapt directly from groups to groupoids. 
\end{proof}
  
\begin{lemma}\label{goodtest}
Let $\Gamma$ be a locally finitely presented $(L,N)$-good groupoid and $\rho\co  \Gamma^{\wedge_L} \to R(\Ql)$ a Zariski-dense representation, with $\ell \in L$. Then $\Gamma$ is $N$-good relative to $\rho\co \Gamma \to R(\Ql)$ 
if and only if $\Gamma^{\wedge_L}$ is $N$-good relative to $\rho$.  
\end{lemma}
\begin{proof}
Take a finite-dimensional $R$-representation $V$. By Lemma \ref{malagrees}, $(B\Gamma)^{\rho, \mal} \simeq (B\Gamma)^{L,\rho, \mal}$. Since $\Gamma$ is $L$-good, Lemma \ref{Lcohoweak2} gives that $\pi_n((B\Gamma)^{\wedge_L})= 0$ for all $1<n\le N$. Applying Lemma \ref{algetpiIMlemma} to the morphism $(B\Gamma)^{\wedge_L} \to B(\Gamma^{\wedge_L}) $, the observations above   show that
\[
 \varpi_n(B\Gamma)^{\rho,\mal} \to \varpi_n(B(\Gamma^{\wedge_L}))^{L,\rho,\mal}
\]
is an isomorphism for $n \le N$ and a surjection for $n=N+1$.

Now, \cite[\S \ref{heid-relgoodsn}]{heid} shows that a pro-group $G$ is $N$-good relative to $\rho$ if and only if $\varpi_n(BG)^{ L,\rho,\mal}=0$ for $1<n\le N$, and the same proof adapts to groupoids. Thus $\Gamma$ is $N$-good relative to $\rho$ if and only if 
$\Gamma^{\wedge_L}$ is so.
\end{proof}

\begin{examples}\label{rgoodexamples}
A pro-finite group $\Gamma$ is good with respect to a representation $\rho\co  \Gamma^{\wedge_L} \to R$ whenever any of the following holds:

\begin{enumerate}
\item  $\Gamma$ is finite, or  $\Gamma^{\wedge_L} \cong\Delta^{\wedge_L}$, for $\Delta$ a finitely generated free discrete group. 

\item\label{rttwo} $\Gamma^{\wedge_L} \cong\Delta^{\wedge_L}$, for $\Delta$ a finitely generated nilpotent discrete group.

\item\label{rtthree} $\Gamma^{\wedge_L} \cong\Delta^{\wedge_L}$, for $\Delta$ the  fundamental group of a compact Riemann surface. In particular, this applies if $\Gamma$ is the fundamental group of a smooth projective curve $C/k$, for $k$ a separably closed field whose characteristic is not in $L$.

\item\label{rffour} If $1 \to F \to \Gamma \to \Pi \to 1$ is an exact sequence of groups, with $F$ finite and $F^{\wedge_L} \to \Gamma^{\wedge_L}$ injective,  assume that $\Pi^{\wedge_L}$ is good relative to $R/\overline{\rho(F)}$, where $\overline{\phantom{F}}$ denotes Zariski closure. Then $\Gamma$ is good relative to $\rho$.
 \end{enumerate}
\end{examples}
\begin{proof}
Combine Lemma \ref{goodtest} with Examples \ref{Lgoodegs} and \cite[Examples \ref{htpy-goodexamples}]{htpy}.
\end{proof}

\begin{remark}
For an example of an important pro-finite group which is not good with respect to a representation, note that $\Sp_g(\Zl)$ is not good with respect to the natural map $\rho\co  \Sp_g(\Zl) \to \Sp_g(\Ql)$ for $g \ge 2$. In fact, $\varpi_2((B\Sp_g(\Zl))^{\rho, \mal})\cong \bG_a$. This issue arises  in \cite{hainmatrelative}, considering the pro-finite mapping class group $\Gamma_g$ acting on a genus $g$ curve. The action on cohomology gives a map $\rho\co  \Gamma_g \to \Sp_g(\Zl)$ with kernel $T_g$, the Torelli subgroup,  and the map $T_g^{1,\mal} \to \ker(\Gamma^{\rho, \mal}_g \to \Sp_g)$ has kernel $\bG_a$. Theorem \ref{lfibrations} allows us to interpret this copy of $\bG_a$ as the image of  the connecting homomorphism $\varpi_2((B\Sp_g(\Zl))^{\rho, \mal})\to T_g^{1,\mal}$.
\end{remark}

\begin{theorem}\label{etpimal}
Let $L$ be a set of primes containing $\ell$, and take $X \in \pro(\bS)_{\delta}$ with fundamental groupoid  $\pi_fX=\Gamma$, equipped with a continuous Zariski-dense representation $\rho\co  \Gamma^{\wedge_L} \to R(\Ql)$ to a reductive pro-algebraic groupoid.  If  
\begin{enumerate}
\item $\pi_n(X^{\wedge_L},-)\ten_{\hat{\Z}}\Ql$ is  finite-dimensional for all $1<n\le N$,  and
\item the $\Gamma^{\wedge_L}$-representation  $\pi_n(X^{\wedge_L},-)\ten_{\hat{\Z}} \Ql$ is an extension of $R$-representations (i.e. a $\Gamma^{L,\rho, \mal}$-representation) for all $1<n\le N$,
\end{enumerate}
then for each $x \in X$ there is an exact sequence
$$
\xymatrix{ & &\varpi_{N+1}(X^{L,\rho,\mal},x) \ar[r] &\varpi_{N+1}((B\Gamma)^{L,\rho,\mal })\ar[dll] &\\
 \ar[r] & \pi_N(X^{\wedge_L},x)\ten_{\hat{\Z}}\Ql \ar[r] &\varpi_{N}(X^{L,\rho,\mal},x) \ar[r] &\varpi_N((B\Gamma)^{L,\rho,\mal })\ar[r] &\ldots\\
\ldots \ar[r] &\pi_2(X^{\wedge_L},x)\ten_{\hat{\Z}}\Ql \ar[r] &  \varpi_{2}(X^{L,\rho,\mal},x) \ar[r] &\varpi_2((B\Gamma)^{L,\rho,\mal })\ar[r] & 0.
}
$$

In particular, if in addition $\Gamma^{\wedge_L}$ is  $(N+1)$-good (resp. $N$-good) with respect to $\rho$, then 
the canonical map
$$
  \pi_n(X^{\wedge_L},-)\ten_{\hat{\Z}} \Ql \to \varpi_{n}(X^{L,\rho,\mal}) 
$$
is an isomorphism for all $n\le N$ (resp. an isomorphism for all $n< N$ and a surjection for $n=N$).
\end{theorem}
\begin{proof}
Without loss of generality, we may assume that $X$ is connected, choose a point $x \in X$, and replace $R$ with the group $R(x,x)$. Let $(\tilde{X},x)$ be the universal cover of $(X,x)$, and note that 
we have a homotopy fibration sequence $(\tilde{X},x) \to (X,x) \to B\pi_1(X,x)$, which means that we can apply Theorem \ref{lfibrations} (after taking a fibrant replacement for $(X,x) \to B\pi_1(X,x)$). This immediately gives the long exact sequence 
\begin{eqnarray*}
\ldots \to \varpi_n(\tilde{X},x) \to \varpi_n(X,x)^{R,\mal}\to \varpi_n(B\pi_1(X,x) )^{R,\mal} \to \varpi_{n-1}(\tilde{X},x)^{K,\mal}\to \\
\ldots \to \varpi_2(\tilde{X},x) \to \varpi_2(X,x)^{R,\mal}\to \varpi_2(B\pi_1(X,x) )^{R,\mal} \to 0.
\end{eqnarray*}

It therefore suffices to show that
$$
  \pi_n(\tilde{X}^{\wedge_L},x)\ten_{\hat{\Z}} \Ql \to \varpi_{n}(\tilde{X}^{L,\alg},x) 
$$
is an isomorphism for $n \le N$. 

We may assume that $\tilde{X}=\{\tilde{X}_{\alpha}\}_{\alpha}$ is an inverse system of fibrant simplicial sets, and then form the tower $\{\tilde{X}(n)\}_n$  by setting $\tilde{X}(n)= \{\tilde{X}_{\alpha}(n)\}_{\alpha} $, where $\{\tilde{X}_{\alpha}(n)\}_n $ the Moore--Postnikov tower  of $\tilde{X}_{\alpha}$. 

Note that if $\tilde{X}(N)$ satisfies the theorem, then we can apply Lemma \ref{algetpiIMlemma} to the morphism $\tilde{X} \to \tilde{X}(N)$, so $\tilde{X}$ will also satisfy the theorem.  We now prove  by induction on $n$ that $\tilde{X}(n)$ satisfies the theorem for $n \le N$. 

For $n=1$, $\tilde{X}(1)$ is contractible, making the long exact sequence automatic.
Now, assume that $\tilde{X}(n-1)$ satisfies the inductive hypothesis, and 
consider  the pro-fibration $\tilde{X}(n) \to \tilde{X}(n-1)$, with fibre $E(n)$ over $x$. Properties of the Postnikov tower give that $\pi_i\tilde{X}(n)= \pi_i\tilde{X}$ for all $i\le n$, with $E(n)$ being a $K(\pi_nX, n)$-space.

The long exact sequence of   Theorem \ref{lfibrations}  gives  $\varpi_{i}(E(n)^{\alg})\cong \varpi_{i}(\tilde{X}(n)^{L,\rho,\mal})$ for $i\ge n$, and exact sequences
$$
\varpi_{i}(E(n)^{\alg})\to \varpi_{i}(\tilde{X}(n)^{L,\alg})\to \pi_i(\tilde{X}^{\wedge_L})\ten_{\hat{\Z}} \Ql \to \varpi_{i-1}(E(n)^{\alg}).
$$

Since $E(n)$ is a $K(\pi_nX, n)$-space, the problem thus reduces to establishing the theorem for the case when $X$ is a $K(\pi,n)$ space (for $n\ge 2$), and $R=1$. Unlike \cite[Theorem \ref{htpy-classicalpi}]{htpy},  we cannot now immediately appeal to the Curtis convergence theorem to show that for any pro-discrete abelian group $\pi$ and $n \ge 2$, the map 
$$
G(K(\pi,n))^{L,\alg}\to N^{-1}(\hat{\pi}\ten_{\hat{\Z}}\Ql[1-n])
$$
is a weak equivalence of simplicial unipotent groups. 

Instead,  observe that we may replace $\pi$ by $\pi^{\hat{\ell}}$, so assume that $\pi$ is a pro-$\ell$ group.
Since $\pi\ten_{\Zl}\Ql$ is finite-dimensional, we may write $\pi= \nu^{\hat{l}}$, for $\nu$ an abelian group of finite rank. On cohomology, we have maps
$$
\H^*(N^{-1}(\pi\ten_{\Zl} \Ql[1-n]), \Ql) \to \H^*(K(\pi,n), \Ql) \to\H^*(K(\nu,n), \Ql).\quad \quad (\dagger)
$$

By \cite[Theorem I.3.4]{QRat}, the Lie algebra $\nu\ten_{\Z} \Ql[1-n]$ is the $\Ql$-homotopy type of $K(\nu, n)$. Since $\pi\ten_{\Zl} \Ql= \nu\ten_{\Z}\Ql$,
the composite is an isomorphism in (\textdagger), while the second map is an isomorphism by Lemma \ref{kpn}. Thus the first map is also an isomorphism, as required.

For the final part, we just note that \cite[\S \ref{heid-relgoodsn}]{heid} shows that  $\Gamma$ is $N$-good relative to $\rho$ if and only if $\varpi_n((B\Gamma)^{L,\rho,\mal})=0$ for $1<n\le N$.
\end{proof}

\subsection{Comparison of homotopy types for complex varieties}

Let $X_{\bt}$ be a simplicial scheme of finite type over $\Cx$. To this we may associate the \'etale homotopy type $X_{\et} \in \pro(\bS)$ (as in Example \ref{ethtpy}). There is also an analytic homotopy type $X_{\an} \co = \diag \Sing(X_{\bt}(\Cx)) \in \bS$, where  $\diag$ is the diagonal functor on bisimplicial sets. We now compare the corresponding schematic homotopy types.

\begin{lemma}
If $G$ is a pro-algebraic group over $\Ql$, and $\rho\co \pi_f(X_{\an})  \to G(\Ql)$ a representation with compact image (for the $\ell$-adic topology on $G(\Ql)$), then $\rho$ factorises canonically through $\widehat{\pi_f(X_{\et})}$, giving a continuous representation
$$
\rho\co \widehat{\pi_f(X_{\et})}\to G(\Ql).
$$
\end{lemma}
\begin{proof}
It follows from \cite[Theorem 8.4]{fried}  that
$$
\widehat{\pi_f(X_{\et})}\cong \widehat{\pi_f(X_{\an})}.
$$
Since $G(\Ql)$ is totally disconnected, any compact subgroup is pro-finite, completing the proof.
\end{proof}

Now, given a reductive pro-algebraic groupoid $R$, and $\rho\co \pi_f(X_{\Cx}) \to R(\Ql)$ with compact Zariski-dense image, we may  compare the relative Malcev homotopy type $X_{\an}^{\rho, \mal}$ of \cite[Definition \ref{htpy-relmaldef}]{htpy}  with the relative Malcev homotopy type $X_{\et}^{\rho, \mal}$ of Definition \ref{relmaldef}, since both are objects of $\Ho(s\cE(R))$.

\begin{theorem}
For $X, \rho$ as above, there is a canonical isomorphism
$$
X_{\an}^{\rho, \mal}\cong X_{\et}^{\rho, \mal}.
$$
\end{theorem}
\begin{proof}
We adapt \cite[Theorem 8.4]{fried}, which constructs a new homotopy type $X_{s.\et}$, and gives morphisms
$$
X_{\et} \la X_{s.\et} \to X_{\an}
$$
in $\pro(\bS)_{\delta}$, inducing weak equivalences on pro-finite completions. By Lemma \ref{malagrees}, $X_{\an}^{\rho, \mal}$ is quasi-isomorphic to $\widehat{X_{\an}}^{ \rho, \mal}$. By Lemma \ref{piIMlemma}, the maps 
\[
 \widehat{X_{\et}}\la \widehat{X_{s.\et}} \to \widehat{X_{\an}}
\]
are weak equivalences in $\hat{\bS}$. Lemma \ref{algetpiIMlemma} then implies that the maps
\[
 \widehat{X_{\et}}^{\rho, \mal}\la \widehat{X_{s.\et}}^{ \rho, \mal} \to \widehat{X_{\an}}^{\rho, \mal}
\]
 are  quasi-isomorphisms, as required. 
\end{proof}

\begin{remarks}
In particular, this shows that there is an action of the Galois group $\Gal(\Cx/K)$ on the relative Malcev homotopy groups $\varpi_n(X_{\an}^{\rho, \mal})$ whenever $X$ is defined over a number field $K$ and $\rho$ is Galois-equivariant. The question of when this action is continuous will be addressed in \S \ref{galoisactions}. 

It seems possible that the conditions of Theorem \ref{classicalpimal} might be satisfied in some cases where those of Theorem \ref{etpimal} do not hold, giving $\varpi_n(X_{\an}^{\rho, \mal})\cong \pi_n(X_{\an})\ten_{\Z}\Ql$, but no such examples are known to the author.
\end{remarks}

\section{Relative and filtered homotopy types}\label{rftypes} 

The aims of this section are twofold. Firstly, we adapt some of the framework of pro-algebraic homotopy types to work over a base ring, rather than a base field. This is motivated by the need in \S \ref{crissn} to phrase the \'etale-crystalline comparison  over variants of  Fontaine's ring $B_{\cris}$ of $p$-adic periods, rather than just over $\Q_p$. Secondly, \S \ref{filteredtypes} develops techniques for transferring filtrations systematically from cochains to homotopy types. These will be used in \S\S \ref{finite} and \ref{local} to determine the structure of homotopy types of quasi-projective varieties. This is possible because the Gysin filtration on   homotopy groups (unlike that on cohomology) is not determined by weights of Frobenius, so imposes further restrictions.

\subsection{Actions on pro-algebraic homotopy types}

Fix a $\Ql$-algebra $A$, and a reductive pro-algebraic groupoid $R$ over $\Ql$.

\begin{definition}
Define $c\Alg_A(R)$ (resp. $DG\Alg_A(R)$) to be the comma category $A \da c\Alg(R)$ (resp. $A\da DG\Alg(R)$), with  model structure induced by Proposition \ref{calgmodel} (resp. Proposition \ref{dgalgmodel}). Denote the opposite category  by $s\Aff_A(R)$ (resp. $dg\Aff_A(R)$). Likewise, define 
\begin{eqnarray*}
c\Alg_A(R)_*&:=& c\Alg_A(R)\da\prod_{x \in \Ob R} A\ten O(R)(x,-),\\
 DG\Alg_A(R)_*&:=& DG\Alg_A(R)\da\prod_{x \in \Ob R} A\ten O(R)(x,-),
\end{eqnarray*}
 and so on.
\end{definition}

Observe that the Quillen equivalence of Proposition \ref{affequiv} induces Quillen equivalences between $dg\Aff_A(R)_*$ and $s\Aff_A(R)_* $, so gives the following equivalence of categories:
$$
\xymatrix@1{
\Ho(dg\Aff_A(R)_*) \ar@<1ex>[r]^{\Spec D} & Ho(s\Aff_A(R)_*) \ar@<1ex>[l]^{\oR(\Spec D^*)}.}
$$

Although we do not have a precise analogue of Theorem \ref{bigequiv} for $ \Ho(dg\Aff_A(R)_*)_0$, we have the following:
\begin{lemma}\label{puny}
Given $X \in dg\Aff(R)_{0*}$ and $\g \in dg\hat{\cN}(R)$, 
\begin{eqnarray*}
&&\Hom_{\Ho(dg\Aff_A(R)_*)}(X\ten A, \bar{W}\g \ten A) \\
&\cong& \Hom_{\Ho(dg\hat{\cN}_A(R))}(\bar{G}(X)\hat{\ten}A,\g\hat{\ten} A)\by^{\exp(\g^R_0\hat{\ten} A)}\prod_{x \in \Ob R}\exp(\H_0\g(x)\hat{\ten} A) .
\end{eqnarray*}
\end{lemma}
\begin{proof}
The proof of \cite[Proposition \ref{htpy-wequiv}]{htpy}  adapts to this context.
\end{proof}

\subsection{Homotopy  actions}

\begin{definition}\label{raut}
Given $\g \in s\hat{\cP}_R$, define a group-valued functor $\Aut_R(\g)$  on the category of $\Ql$-algebras by setting
$$
\Aut_R(\g)(A):= \Aut_{s\cP_A(R)}(\g\hat{\ten} A).
$$

Given $G \in s\cE(R)$, define $\RAut(G):=\Aut_R(\Ru(G))$, noting that
$$
\RAut(G)(\Ql) \cong \Aut_{\Ho(s\cE(R))_*}(G).
$$
 For $G \in s\agpd$, set $\RAut(G):=\Aut_{G^{\red}}(\Ru(G))$.
\end{definition}

\begin{lemma}\label{auto}
If  $G \in s\cE(R)$ is such that $\H^i(G,V)$ is finite-dimensional for all $i$ and all finite-dimensional irreducible $R$-representations $V$, then the group-valued functor
$$
\RAut(G)
$$
is represented by a  pro-algebraic group over $\Ql$. The map
$$
\RAut(G) \to \prod_{x \in \Ob R} \prod_i  \Aut_R(\H^i(G,O(R)(x,-)))
$$
of pro-algebraic groups has pro-unipotent kernel.
\end{lemma}
\begin{proof}
This is a consequence of \cite[Theorem \ref{htpy-auto}]{htpy}, which proves the corresponding statement for the group
$
\ROut(G):= \RAut(G)/ \prod_{x \in \Ob G}\Ru(G)(x). 
$
Since $ \prod_{x \in \Ob G}\Ru(G)(x)$ is a pro-unipotent pro-algebraic group, the result follows.
\end{proof}

\begin{definition}\label{algaut}
Given a pro-algebraic groupoid $G$, we may extend the automorphism group $\Aut(G)$ to a group presheaf  over $\Ql$, by setting 
$$
\Aut(G)(A):= \Aut_A(G\by_{\Spec \Ql} \Spec A).
$$
\end{definition}

\begin{lemma}\label{outdef}
For $G \in s\cE(R)$, there is a group presheaf $\Aut^h(G)$ over $\Ql$,  with  the properties that $\Aut^h(G)(\Ql)$ is the group of automorphisms of $G$ in $\Ho( \Ob G \da s\agpd)$, and that there is an  exact sequence 
$$
1 \to \RAut(G) \to \Aut^h(G) \xra{\alpha} \Aut(R) \to 1,
$$
where $\Aut(R) $ is given the algebraic structure of Definition \ref{algaut}.

If $\H^i(G,V)$ is finite-dimensional for all $i$ and all finite-dimensional irreducible $R$-representations $V$, then $\alpha$ is fibred in affine schemes.
\end{lemma}
\begin{proof}
Let $R=G^{\red}$, take $Y \in \Ho(dg\Aff(R)_*)$ corresponding to $G$ under the equivalence of Theorem \ref{bigequiv} and define 
$$
\Aut^h(G)(A):= \{(f,\theta) \,:\, f\in \Aut(R)(A), \theta \in \Iso_{\Ho(dg\Aff_A(R)_*)}(Y\ten A, f^{\sharp} Y\ten A)\}.
$$ 
We may now take a minimal model $\m$ for $\bar{G}(Y)\in dg\hat{\cN}(R)$, and observe that Lemma \ref{puny} then gives
\begin{eqnarray*}
&&\Hom_{\Ho(dg\Aff_A(R)_*)}(Y\ten A, f^{\sharp} Y\ten A)\\
 &\cong& \Hom_{\Ho(dg\Aff_A(R)_*)}(Y\ten A, f^{\sharp} \bar{W}\m\ten A)\\
&\cong& \Hom_{\Ho(dg\hat{\cN}_A(R)_*)}(\bar{G}(Y)\hat{\ten}A,\m\hat{\ten} A)\by^{\exp(\m^R_0\hat{\ten} A)}\prod_{x \in \Ob R}\exp(\H_0\m(x)\hat{\ten} A) \\
&\cong & \Hom_{\Ho(dg\hat{\cN}_A(R))}(\m\hat{\ten}A,\m\hat{\ten} A)\by^{\exp(\m^R_0\hat{\ten} A)}\prod_{x \in \Ob R}\exp(\H_0\m(x)\hat{\ten} A).
\end{eqnarray*}

The proof that $\alpha$ is fibred in affine schemes is now essentially the same as  Lemma \ref{auto}, which deals with the fibre over $1 \in \Aut(R)$.
\end{proof}

\begin{definition}
Given a pro-discrete group $\Gamma$, we say that a morphism $\Gamma \to \Aut^h(G)(\Ql)$ is \emph{algebraic} if it factors through a morphism $\Gamma^{\alg} \to \Aut^h(G)$ of presheaves of groups.
\end{definition}

\begin{corollary}\label{redalgebraic}
If $\H^i(G,V)$ is finite-dimensional for all $i$ and all finite-dimensional irreducible $R$-representations $V$,   then a morphism $\Gamma \to \Aut^h(G)(\Ql)$ is algebraic whenever $\Gamma \to \Aut(G^{\red})$ is so.
\end{corollary}
\begin{proof}
We have $\Gamma^{\alg}\to \Aut(G^{\red})$, so $\theta\co \Gamma \to (\Gamma^{\alg}\by_{\Aut(G^{\red})}\Aut^h(G))(\Ql)$. Since $\Aut^h(G) \to \Aut(G^{\red}) $ is fibred in affine schemes, the group on the right is pro-algebraic, so $\theta$ factors through $\Gamma^{\alg}$, as required.
\end{proof}

If $R=G^{\red}$, observe that there is canonical action of $\Aut^h(G)$ on $\bigoplus_{x \in \Ob R}\H^*(G, O(R)(x,-))$. In fact, we have a homomorphism
$$
\beta\co \Aut^h(G)\to \Aut(R) \by \Aut(\bigoplus_{x \in \Ob R}\H^*(G, O(R)(x,-)))
$$
of presheaves of groups. 

\begin{lemma}\label{cohohelps}
If $\H^i(G,V)$ is finite-dimensional for all $i$ and all finite-dimensional irreducible $R$-representations $V$, then the kernel of $\beta$ is a pro-unipotent pro-algebraic group.
\end{lemma}
\begin{proof}
The kernel of $\beta$ is just the kernel of 
$$
\RAut(G) \to  \prod_{x \in \Ob R} \prod_i  \Aut_R\H^i(G,O(R)(x,-)),
$$
which is pro-unipotent by Lemma \ref{auto}.
\end{proof}

\subsection{Filtered homotopy types}\label{filteredtypes}

\subsubsection{Commutative algebras}

\begin{definition}
Given a $\Ql$-algebra $A$ and a reductive pro-algebraic groupoid $R$ over $\Ql$, define $FDG\Alg_A(R)$ (resp. $Fc\Alg_A(R)$) to consist of $R$-representations  $B$ in non-negatively graded cochain (resp. cosimplicial) algebras over $A$, equipped with an increasing exhaustive filtration $J_0B \subset J_1B \subset \ldots$ of $B$ as a DG (resp. cosimplicial) $(R,A)$-module, with the property that $(J_mB)\cdot  (J_nB)\subset J_{m+n} B$. Morphisms are required to respect the filtration, and we assume that $1 \in J_0B$.

Write 
\begin{eqnarray*}
Fc\Alg(R)_*&:=& Fc\Alg(R) \da \prod_{x \in \Ob R} O(R)(x,-),\\
FDG\Alg(R)_*&:=& FDG\Alg(R) \da \prod_{x \in \Ob R} O(R)(x,-),
\end{eqnarray*}
 where $O(R)(x,-)=J_0O(R)(x,-)$. 
\end{definition}

Given $(B,J) \in FDG\Alg_A(R)$ or $Fc\Alg_A(R)$, there is a spectral sequence ${}_J\!\EE^{*,*}_*(B)$ associated to the filtration $J$, with 
$$
{}_J\!\EE_1^{a,b}(B)= \H^{a+b}(\Gr^J_{-a}B).
$$
\begin{definition}\label{efil}
We regard ${}_J\!\EE^{*,*}_1(B)$ as an object of $FDG\Alg_A(R)$, with 
$$
J_m( {}_J\!\EE_1^{*,*}(B))^n= \bigoplus_{r \le m} {}_J\!\EE_1^{-r,n+r}(B),
$$
noting that $d(J_m (E_1)^n) \subset J_{m-1}(E_1)^{n+1}$.
\end{definition}

\begin{definition}\label{deffilmodel}
Define a map $f\co B \to C$ to be a \emph{fibration} if the maps $J_nf\co  J_nB \to J_nC$ are all surjective. A map $f$ is a \emph{weak equivalence} if the maps $ {}_J\!\EE_1^{*,*}(f)\co {}_J\!\EE_1^{*,*}(B)\to{}_J\!\EE_1^{*,*}(C)$ are all isomorphisms. 
\end{definition}

\begin{lemma}\label{fdgmod}
There are cofibrantly generated model structures on the categories $Fc\Mod_A(R)$ and $FDG\Mod_A(R)$, with the classes of  fibrations  and weak equivalences above.
\end{lemma}
\begin{proof}
First, note that normalisation gives an equivalence $Fc\Mod_A(R)\to FDG\Mod_A(R)$ of categories,  preserving and reflecting fibrations  and weak equivalences. It thus suffices only to consider $ FDG\Mod_A(R) $

 Let $S_{n,m}$ denote the cochain complex  consisting of  $A$ concentrated in degree $n$, with $J_mS_{n,m}=S_{n,m}$ and $J_{m-1}S_{n,m}=0$. Let $D_{n,m}$ denote the cochain complex  consisting of  $A$ concentrated in degrees $n, n-1$ with differential $d^{n-1}$ the identity, $J_mD_{n,m}=D_{n,m}$ and $J_{m-1}D_{n,m}=0$. By convention, $D_{0,m}=0$. Note that there are natural maps $S_{n,m} \to D_{n,m}$.

For a set $\{V\}$ of representatives of irreducible $R$-representations in $\Ql$-vector spaces, define $I$ to be the set of morphisms $A \ten S_{n,m}\ten V \to A \ten D_{n,m}\ten V$, for $n \ge 0$. Define $J$ to be the set of morphisms $0 \to A \ten D_{n,m}\ten V$, for $n \ge 0$. 

Now, 
\begin{eqnarray*}
\Hom_{FDG\Mod_A(R)}(A \ten S_{n,m}\ten V, M) &=& \Hom_{R}(V, J_m\z^nM)\\
\Hom_{FDG\Mod_A(R)}(A \ten D_{n,m}\ten V, M) &=& \Hom_{R}(V, J_m M^{n-1}),
\end{eqnarray*}
so a  map $f\co  M \to N$ in $ FDG\Mod_A(R)$ is then $I$-injective when 
$$
J_mM^{n-1} \xra{f,d} J_mN^{n-1}\by_{d,J_m\z^nN,f}\z^nM
$$ 
is surjective, for all $m,n$, and $J$-injective when 
$$
J_mM^{n-1} \xra{f} J_mN^{n-1}
$$
is surjective for all $n$. Thus $I$-injectives are trivial fibrations, and $J$-injectives are fibrations. 

Since $S_{n,m}= D_{n+1,m}/S_{n+1,m}$, the map $0 \to D_{n,m}$ is a  composition $0 \to S_{n,m} \to D_{n,m}$ of pushouts of maps in $I$, so maps in $J$  are all $I$-cofibrations. Since maps in $J$ are all weak equivalences, we have satisfied  the conditions of \cite[Theorem 2.1.19,]{Hovey}  giving the model structure claimed. 
\end{proof}

\begin{lemma}\label{fdgmodcof}
In the category $FDG\Mod(R)=FDG\Mod_{\Ql}(R)$, all objects $V$ are cofibrant.
\end{lemma}
\begin{proof}
Given $V \in FDG\Mod_{\Ql}(R)$, it will suffice to show that $J_0V$ is cofibrant, and  that all the maps $J_{m-1}V \to J_{m}V$ are cofibrations, since $V = \varinjlim J_m V$. 
To do this, we will show that these maps are transfinite compositions of pushouts of generating cofibrations.

Now, since all $R$-representations are semisimple, we may choose  decompositions $\gr^J_mV^n = M^n \oplus N^n\oplus dN^{n-1}$, with $dM^n=0$. By semisimplicity, we may also lift the $R$-modules $M^i, N^i$ to $\tilde{M}^i, \tilde{N}^i \subset J_mV$. Now $d\tilde{M} \subset J_{m-1}V$, so the map $ J_{m-1}V \to  J_mV$ is a pushout of $\bigoplus_n (S_{n+1,m} \ten \tilde{M}^n) \to  \bigoplus_n (D_{n+1,m} \ten \tilde{M}^n) \oplus \bigoplus_n (D_{n+1,m} \ten \tilde{N}^n)$, and hence a cofibration. Since this argument also applies to $0 \to J_0V$, we deduce that $V$ is cofibrant.
\end{proof}

\begin{proposition}\label{filalgmod}
There is a cofibrantly generated model structure on $FDG\Alg_A(R)$ (resp. $Fc\Alg_A(R)$), for which a morphism is a  fibration or  weak equivalence whenever the underlying morphism in $FDG\Mod_A(R)$ (resp. $Fc\Mod_A(R)$) is so (in the model structure of Lemma \ref{fdgmod}).
\end{proposition}
\begin{proof}
The forgetful functor $FDG\Alg_A(R)\to FDG\Mod_A(R)$ (resp. $Fc\Alg_A(R) \to Fc\Mod_A(R)$) preserves filtered colimits and has a left adjoint, the free algebra functor. Since the free algebra functor maps trivial generating cofibrations to weak equivalences,
we may  apply \cite[Theorem 11.3.2]{Hirschhorn}, which gives the required cofibrantly generated  model structure. The generating cofibrations and trivial cofibrations are given by the images under the free algebra functor of the generating cofibrations and trivial cofibrations in $FDG\Mod_A(R)$ (resp. $Fc\Mod_A(R)$).
\end{proof}

\subsubsection{Lie algebras}

\begin{definition}
 Define  $F\hat{\cN}_A(R)$ to be  opposite to the category $F\hat{\cN}_A(R)^{\opp}$ of $R$-representations in ind-conilpotent (see Definition \ref{indcon}) Lie coalgebras $C$ over $A$, equipped with an exhaustive increasing filtration $J_0C \subset J_1C \subset \ldots$, of $C$ as an $(R,A)$-module, with the property that $\nabla(J_{r} C)\subset \sum_{m+n=r} (J_m C)\otimes  (J_nC)$, for $\nabla$ the cobracket. Morphisms are required to respect the filtration.

Similarly,  $Fdg\hat{\cN}_A(R)$ is opposite to the category of $R$-representations in  non-negatively filtered ind-conilpotent $\N_0$-graded cochain Lie coalgebras over $A$.   $Fs\hat{\cN}_A(R)$ is the category of simplicial objects in $F\hat{\cN}_A(R)$. When $A=\Ql$, we will usually drop the subscript $A$.
\end{definition}

\begin{proposition}\label{filliemod}
There is a closed model structure on $Fdg\hat{\cN}_A(R)$ (resp. $Fs\hat{\cN}_A(R)$),  in which a morphism $f\co \g \to \fh$ is a fibration or a weak equivalence whenever the underlying map 
$f^{\vee}\co  \fh^{\vee} \to \g^{\vee}$ in $FDG\Mod_A(R)$ 
(resp. $Fc\Mod_A(R)$)  is a cofibration or a weak equivalence.
\end{proposition}
\begin{proof}
The proof of \cite[Lemma \ref{htpy-cnamod}]{htpy}  carries over to this context.
\end{proof}

\subsubsection{Equivalences}

\begin{definition}
Define $Fc\Alg(R)_{00_*}$ (resp. $FDG\Alg(R)_{00_*}$) to be the full subcategory of $Fc\Alg_A(R)_*$ (resp. $FDG\Alg_A(R)_*$) consisting of objects $B$ with $B^0=\Ql$. Let $Fc\Alg(R)_{0*}$ (resp. $FDG\Alg(R)_{0_*}$) be the full subcategory consisting of objects weakly equivalent to objects of $Fc\Alg(R)_{00_*}$ (resp. $FDG\Alg(R)_{00_*}$). Let $\Ho(Fc\Alg(R)_*)_{0}$ (resp. $\Ho(FDG\Alg(R)_*)_0$) be the full subcategory of $\Ho(Fc\Alg(R)_*)$ (resp. $\Ho(FDG\Alg(R)_*)$) on objects $Fc\Alg(R)_{0_*}$ (resp. $FDG\Alg(R)_{0_*}$). Denote the opposite category to $Fc\Alg(R)_{00*}$ by $Fs\Aff(R)_{00*}$, etc.
\end{definition}

\begin{definition}
Given $\g \in Fs\hat{\cN}(R)$, we define $\bar{W}\g \in Fs\Aff(R)$ by 
$$
(\bar{W}\g)(B):= \bar{W}(\exp(\Hom_{F\Mod(R)}(  \g^{\vee}, (B)))) \in \bS
$$
for $B \in \Alg_A(R)$. Here, $\bar{W}$ is the classifying space functor of Definition \ref{barwdef},   and $\exp$ denotes exponentiation of a pro-nilpotent Lie algebra to give a pro-unipotent group.

Observe that this functor is continuous, and denote its left adjoint by $G\co Fs\Aff(R)\to Fs\hat{\cN}(R)$.
\end{definition}

\begin{definition}\label{barwg}
Define functors $\xymatrix@1{ Fdg\Aff(R) \ar@<1ex>[r]^G & Fdg\hat{\cN}(R) \ar@<1ex>[l]^{\bar{W}}_{\bot}}$ as follows.
For $\g \in Fdg\hat{\cN}(R)$, the Lie bracket gives a linear map $\bigwedge^2\g \to \g$. Write $\Delta$ for the dual $\Delta\co \g^{\vee} \to \bigwedge^2\g^{\vee}$, which respects the filtration.
This is equivalent to a  map $\Delta\co \g^{\vee}[-1] \to \Symm^2(\g^{\vee}[-1])$, and we define 
$$
O(\bar{W}\g):= \Symm(\g^{\vee}[-1])
$$
to be the graded polynomial ring on generators $\g^{\vee}[-1]$, with a derivation defined on generators by $D:=d +\Delta$. The Jacobi identities ensure that $D^2=0$.

We define $G$ by writing $\sigma B[1]$ for the brutal truncation (in non-negative degrees) of $B[1]$, and setting
$$
G(B)^{\vee}= \CoLie(\sigma B[1]),
$$
the free filtered graded Lie coalgebra over $\Ql$, with differential similarly defined on cogenerators by $D:=d +\mu$, $\mu$ here being the product on $B$.  Note also that $G(B)$ is cofibrant for all $B$.
\end{definition}

\begin{definition}
Define the category $Fs\cP(R)$ (resp. $Fdg\cP(R)$) to have the fibrant objects of $Fs\hat{\cN}(R)$ (resp. $Fdg\hat{\cN}(R)$), with morphisms given by
\begin{eqnarray*}
\Hom_{Fs\cP(R)}(\g,\fh)&=&\Hom_{\Ho(Fs\hat{\cN}(R))}(\g,\fh)\by^{\exp(\fh_0^R)}\prod_{x \in \Ob R}\exp(\pi_0\fh(x)),\\
\Hom_{Fdg\cP(R)}(\g,\fh)&=& \Hom_{\Ho(Fdg\hat{\cN}(R))}(\g,\fh)\by^{\exp(\fh^R_0)}\prod_{x \in \Ob R}\exp(\H_0\fh(x)),
\end{eqnarray*}
where $\fh_0^R$ is the Lie algebra $\Hom_{\Mod(R)}(\fh_0^{\vee}, \Ql)=\Hom_{F\Mod(R)}(\fh_0^{\vee}, \Ql)$, acting by conjugation on the set of homomorphisms.
\end{definition}

\begin{theorem}\label{fbigequiv}
There is the following commutative diagram of equivalences of categories:
$$
\xymatrix{
\Ho(Fdg\Aff(R)_*)_0 \ar@<1ex>[r]^{\Spec D} \ar@<-1ex>[d]_{\bar{G}}& Ho(Fs\Aff(R)_*)_0  \ar@<-1ex>[d]_{\bar{G}}\ar@<1ex>[l]^{\Spec \Th} \\
Fdg\cP(R) \ar@<-1ex>[u]_{\bar{W}}  & Fs\cP(R), \ar@<-1ex>[u]_{\bar{W}} \ar@<1ex>[l]^{N}
}
$$
where $N$ denotes normalisation,  $D$ is denormalisation, and  $\Th$ is the functor of Thom-Sullivan cochains.
\end{theorem}
\begin{proof}
The proof of  Theorem \ref{bigequiv}  carries over to this context, making use of Lemma \ref{fdgmodcof}, which implies that everything in the image of $\bar{W}$ is fibrant, as are all objects of $Fdg\hat{\cN}(R)$ and $Fs\hat{\cN}(R) $.
On objects, the functor $\bar{G}$ is defined by choosing, for any $X \in \Ho(Fs\Aff(R)_*)_0$ (resp. $X \in \Ho(Fdg\Aff(R)_*)_0$) a weakly equivalent object $X' \in Fs\Aff(R)_{00}$ (resp. $X' \in Fdg\Aff(R)_{00}$), and setting
$$
\bar{G}(X):= G(X'),
$$
for the functor $G$ from Definition \ref{barwg}.
\end{proof}

Although we do not have a precise analogue of this result for $ \Ho(Fdg\Aff_A(R))$ for general $A$, we do have the following:
\begin{lemma}\label{fpuny}
Given $X \in \Ho(Fdg\Aff(R)_*)_0$ and $\g \in Fdg\hat{\cN}(R)$, 
\begin{eqnarray*}
&&\Hom_{\Ho(Fdg\Aff_A(R)_*)}(X\ten A, \bar{W}\g \ten A) \\
&\cong& \Hom_{\Ho(Fdg\hat{\cN}_A(R))}(\bar{G}(X)\hat{\ten}A,\g\hat{\ten} A)\by^{\exp(\g^R_0\hat{\ten} A)} \prod_{x \in \Ob R}\exp(\H_0\g(x)). 
\end{eqnarray*}
\end{lemma}
\begin{proof}
The proof of \cite[Proposition \ref{htpy-wequiv}]{htpy}  adapts to this context.
\end{proof}

\begin{definition}\label{qf}
We say that a filtered cochain algebra $(B,J) \in FDG\Alg_A(R)$ is \emph{quasi-formal} if it is weakly equivalent in $FDG\Alg_A(R)$ to ${}_J\!\EE^{*,*}_1(B)$ (as in Definition \ref{efil}).
We say that a filtered homotopy type is quasi-formal if its associated cochain algebra is so.
\end{definition}

\subsubsection{Minimal models}\label{minimalsn}

Let $FDG\Rep(R)= FDG\Mod_{\Ql}(R)$ be the category of non-negatively graded filtered complexes of $R$-representations. 

\begin{definition}\label{fminimal}
We say that $M \in  FDG\Rep(R)$ is \emph{minimal} if $d(J_mM) \subset J_{m-1}M$ for all $m$.
\end{definition}

\begin{lemma}\label{frepminimal}
For any $V \in FDG\Rep(R)$, there exists a quasi-isomorphic filtered subobject $M \into V$, with $M$ minimal.
\end{lemma}
\begin{proof}
We prove this by induction on the filtration. Assume that we have constructed a filtered quasi-isomorphism $J_mf\co  J_mM \into J_mV$ (for $m=-1$, this is trivial). Pick a basis $v_{\alpha}$ for $\H^*(\gr_{m+1}^JV)$, and lift $v_{\alpha}$ to $v_{\alpha}' \in J_{m+1}V$. Thus $dv_{\alpha}' \in J_mV$, and $[dv_{\alpha}']=0 \in \H^*(J_mV/J_mM)=0$. This means that $dv_{\alpha}' \in J_mM + dJ_mV$. Choose $u_{\alpha} \in J_mV$ such that $dv_{\alpha}'-du_{\alpha} \in J_mM$, and set $\tilde{v}_{\alpha}:= v_{\alpha}'-u_{\alpha}$.

Now, $[\tilde{v}_{\alpha}] =v_{\alpha} \in \H^*(\gr_{m+1}^JV)$, so define
$$
J_{m+1}M:= J_mM \oplus \langle \tilde{v}_{\alpha}\rangle_{\alpha};
$$
this has the properties that $dJ_{m+1}M \subset J_mM$ and $\H^*(\gr_{m+1}^JM) \cong \H^*(\gr_{m+1}^JV)$, as required.
\end{proof}

\begin{definition}
We say that a cofibrant object $\m \in Fdg\hat{\cN}(R)$ (resp. $Fs\hat{\cN}(R)$) is \emph{minimal} if $(\m/[\m,\m])^{\vee}$ (resp. $N(\m/[\m,\m])^{\vee}$) is minimal in the sense of Definition \ref{fminimal}.
\end{definition}

\begin{proposition}[Minimal models]\label{fsdgminimal}
Every weak equivalence class in $Fdg\hat{\cN}(R)$ (resp. $s\hat{\cN}(R)$) has a minimal element $\m$, unique up to non-unique isomorphism.
\end{proposition}
\begin{proof}
The proof of  \cite[Proposition \ref{htpy-sminimal}]{htpy}  adapts to this context, using Lemma \ref{frepminimal} instead of the corresponding result for $DG\Rep(R)$.
\end{proof}

\subsubsection{Homotopy automorphisms}

\begin{definition}
Given $\fu \in Fs\hat{\cN}(R)$, let $G=\exp(\fu) \rtimes R$, and   define the \emph{group presheaf of filtered automorphisms} by
$$
\Aut^h_J(G)(A):= \{(f,\theta) \,:\, f\in \Aut(R)(A), \theta \in \Iso_{Fs\cP_A(R)}(\fu\hat{\ten}A, f^{\sharp}\fu\hat{\ten}A)\}.
$$ 
Define $\RAut_J(G):= \ker(\Aut^h_J(G) \to \Aut(R))$.
\end{definition}

\begin{definition}
Given $V \in \Rep(R)$ and $\g \in Fs\hat{\cN}(R)$, define the spectral sequence ${}_J\!\EE^{*,*}_*(R\ltimes \exp(\g),V)$ to be the cohomology spectral sequence of the filtered complex
$$
O(\bar{W}\g)\ten^RV,
$$
for $J_0V=V$. Thus ${}_J\!\EE_1^{a,b}(R\ltimes \exp(\g),V)= \H^{a+b}(\Gr^J_{-a}O(\bar{W}\g)\ten^RV)$.
\end{definition}

\begin{lemma}\label{frout}
Assume that $G$ is as above, and let $\m \in Fs\hat{\cN}(R)$ be a minimal model for $\Ru(G)$.  If $\H^i(G,V)$ is finite-dimensional for all $i$ and all finite-dimensional irreducible $R$-representations $V$, then the group presheaves
$$
\Aut_{Fs\hat{\cN}(R)}(\m)\by \prod_{x \in \Ob R}\exp(\pi_0\m(x)) \xra{\alpha} \RAut_J(G) \xra{\beta} \prod_{a,b}  \Aut_R( {}_J\!\EE_1^{a,b}(G,O(R)))
$$
are all pro-algebraic groups, the maps $\alpha$ and $\beta$ both have pro-unipotent kernels, and $\beta$ is surjective.
\end{lemma}
\begin{proof}
The proof of \cite[Theorem \ref{htpy-auto}]{htpy}  carries over.
\end{proof}

\subsubsection{Examples}

\begin{definition}\label{tau}
Given $B^{\bt} \in  DG\Alg_A(R)$, we define the \emph{good truncation} $\tau_*$  on $B$ by 
$$
(\tau_m B)^n:= \left\{ \begin{matrix} B^n & n< m\\ \z^m(B) & n=m\\ 0 & n> m. \end{matrix} \right. 
$$
Observe that $(B^{\bt}, \tau) \in FDG\Alg_A(R)$.
\end{definition}

\begin{definition}\label{tau''}
Given a bicosimplicial algebra $B^{\bt,\bt} \in  cc\Alg_A(R)$, we define the  associated filtered cosimplicial algebra  $(\tau''_0B \le \tau''_1B \le \ldots) \in Fc\Alg_A(R)$ by 
$$
(\tau_m'' B)^n =(D\tau_m\Th B^{n, \bt})^n,
$$
for $D,\Th$ as in Theorem \ref{affequiv}. Observe that   there is a canonical quasi-isomorphism $\diag B^{\bt,\bt} \to \tau''_{\infty}B^{\bt}$, where $\diag$ denotes the diagonal of a bicosimplicial complex.
 \end{definition}

In practice, the only filtered homotopy types which we will encounter come from morphisms of spaces: 

\begin{definition}\label{relgod}
Given an algebraic variety $X$ and an ind-constructible $\ell$-adic sheaf $\vv$ on $X$, recall (e.g. from \cite[Definition 2.3]{paper1}) that there is a natural cosimplicial complex
$$
\sC^{\bt}_{\et}(\vv)
$$
of $\ell$-adic sheaves on $\vv$, with the property that $\Gamma(X, \sC^{\bt}_{\et}(\vv))=\CC^{\bt}_{\et}(X,\vv)$, the Godement resolution (as in Remark \ref{ethtpy}). This construction respects tensor products.
\end{definition}

\begin{lemma}\label{jdef}
To any morphism $j\co  Y \to X$ of algebraic varieties, and any  $\Ql$-sheaf $\sS$ of algebras on $Y$ as in Definition \ref{relgod}, there is associated a canonical filtered homotopy type $\CC^{\bt}_{\et}(j,\sS)  \in  \Ho(Fc\Alg_{\Ql})$, with the property that
${}_J\!\EE^{*,*}_*\CC^{\bt}_{\et}(j,\sS)$ is the Leray spectral sequence
$$
{}_J\!\EE^{a,b}_1\CC^{\bt}_{\et}(j,\sS)=  \H^{2a+b}(X, \oR^{-a}j_*\sS)\abuts \H^{a+b}(Y, \sS).
$$
The associated unfiltered homotopy type is canonically weakly equivalent to 
$
\CC^{\bt}_{\et}(Y,\sS). 
$
\end{lemma}
\begin{proof}
We have a $\Ql$-sheaf $j_*\sC^{\bt}_{\et}(\sS)$ of cosimplicial algebras on $X$, and hence a bicosimplicial algebra
$$
\CC^{\bt}_{\et}(X, j_*\sC^{\bt}_{\et}(\sS)).
$$
Now, set
$$
J_n\CC^{\bt}_{\et}(j,\sS)= \tau''_n\CC^{\bt}_{\et}(X, j_*\sC^{\bt}_{\et}(\sS))= \diag \CC^{\bt}_{\et}(X, D\tau_n\Th j_*\sC^{\bt}_{\et}(\sS)),
$$
as in Definition \ref{tau''}, with $ \CC^{\bt}_{\et}(X, j_*\sC^{\bt}_{\et}(\sS)) \to J_{\infty}\CC^{\bt}_{\et}(j, \sS)$ a quasi-isomorphism.

Finally, observe that there is a quasi-isomorphism
$$
\CC^{\bt}_{\et}(Y,\sS) = \Gamma(X,j_*\sC^{\bt}_{\et}(\sS))\to \diag \CC^{\bt}_{\et}(X,j_*\sC^{\bt}_{\et}(\sS)),
$$
and that $\gr^{\tau}_nj_*\sC^{\bt}_{\et}(\sS)$ is quasi-isomorphic to $\oR^nj_*\sS$.
\end{proof}

\begin{remark}
There is a similar statement for filtrations on homotopy types coming from morphisms of topological spaces, using \v Cech resolutions instead of Godement resolutions.
\end{remark}

Since the construction above is functorial, for any point $y \in Y$, we have a morphism $\CC^{\bt}_{\et}(j,\sS) \to  \CC^{\bt}_{\et}(\id_y,\sS_y)$, where $\id_y$ is the identity map $\id_y \co  y \to y$. Now,
$$
J_n\CC^{\bt}_{\et}(\id_y,\sS_y)= \diag \CC^{\bt}_{\et}(y, D\tau_n\Th \sS_y).
$$
Since $\sS_y$ has constant simplicial structure, $\Th \sS_y=\sS_y$, so $J_n\CC^{\bt}_{\et}(\id_y,\sS_y)=\sS_y$ for all $n \ge 0$.

\begin{definition}\label{leraytype}
Given a morphism $j\co Y \to X$ of algebraic varieties and
a Zariski-dense continuous map
$$
\rho\co \widehat{\pi_f^{\et}(Y)} \to R(\Ql)
$$
define the filtered homotopy type $(Y^{\rho, \mal},j)$ to correspond to $\CC^{\bt}_{\et}(j,\bO(R))\in Fc\Alg(R)_*$, where the augmentation map is the canonical morphism 
$$
\CC^{\bt}_{\et}(j,\bO(R))\to \prod_{y \in Y}\CC^{\bt}_{\et}(\id_y,\bO(R))=\prod_{y \in Y} O(R)(y,-).
$$ 
\end{definition}

\section{Algebraic Galois actions}\label{galoisactions}

\subsection{Weight decompositions}\label{wgtdecomp}

By a weight decomposition, we will mean an algebraic action of the group $\bG_m$. A weight decomposition on a vector space $V$ is equivalent to a decomposition $V=\bigoplus_{n \in \Z}\cW_nV$, given by $\lambda \in \bG_m$ acting as $\lambda^n$ on $\cW_nV$.

Fix a prime $p$, which need not differ from $\ell$. Let $\Z^{\alg}$ be the  pro-algebraic group over $\Ql$  parametrising $\Z$-representations. Since $\Z$ is commutative, $\Z^{\alg}$ is commutative, so $\Z^{\alg}=\Z^{\red}\by \Ru(\Z^{\alg})$, where $\Z^{\red}$ is its reductive quotient. For any unipotent  algebraic group $U$, this means that $\Hom(\Ru(\Z^{\alg}), U)\cong \Hom(\Z, U(\Ql))= U(\Ql)$, so $\Ru(\Z^{\alg})= \bG_a$. Combining these observations gives $\Z^{\alg}= \bG_a \by \Z^{\red}$. 

Likewise, let $\hat{\Z}^{\alg}$ be the  pro-algebraic group over $\Ql$  parametrising continuous $\hat{\Z}$-representations. Since continuous $\hat{\Z}$-representations form a full subcategory of $\Z$-representations, $\hat{\Z}^{\alg}$ is a quotient of $\Z^{\alg}$. The reasoning above adapts to show that $\hat{\Z}^{\alg}= \bG_a \by \hat{\Z}^{\red}$. 

\begin{definition}\label{nq}
Given  $n \in \Z$ and a power $q$ of $p$, recall that an element $\alpha \in \bar{\Ql}$ is said to be \emph{pure of weight} $n$ if it is algebraic and for every  
embedding $\iota \co  \bar{\Ql}\into \Cx$ the element $\iota (\alpha )$ has complex absolute value $q^{n/2}$.

Let $M_q$ be the quotient of $\hat{\Z}^{\red}$ whose representations $\rho$ correspond to semisimple $\hat{\Z}$-representations for which  the eigenvalues of $\rho(1)$ are all of integer weight with respect to $q$. Such representations are called mixed.  
\end{definition}

Observe that every $M_q$-representation decomposes into ``pure''  representations, in which all eigenvalues have the same weight. There is thus a canonical map $\bG_m \to M_q$ given by $\lambda \in \bG_m$ acting as $\lambda^n$ on a pure representation of weight $n$.

\begin{definition}\label{pqmq}
Define $P_q$ to be the quotient of $M_q$ whose representations are pure of weight $0$, so $P_q=M_q/\bG_m$.
\end{definition}

\begin{definition}\label{inq}
Given  $n \in \Z$, an embedding $\iota\co \bar{\Ql}\to \Cx$  and a power $q$ of $p$, recall that an element $\alpha \in \bar{\Ql}$ is said to be $\iota$-pure of weight $n$ if $|\iota(\alpha)|=  q^{n/2}$. 

Let $M_{\iota,q}$ be the quotient of $\Z^{\red}$ whose representations $\rho$ correspond to semisimple $\Z$-representations for which  the eigenvalues of $\rho(1)$ are all of integer $\iota$-weight. Note that $M_q$ is a quotient of $M_{\iota,q}$.  
\end{definition}

Observe that  there is  a canonical map $\bG_m \to M_{\iota,q}$ given by $\lambda \in \bG_m$ acting as $\lambda^n$ on an $\iota$-pure representation of weight $n$, and that this induces the map $\bG_m \to M_q$ above.

\begin{definition}\label{ipqmq}
Define $P_{\iota,q}$ to be the quotient of $M_{\iota,q}$ whose representations are pure of $\iota$-weight $0$, so $P_{\iota,q}=M_{\iota,q}/\bG_m$.
\end{definition}

\begin{definition}
Given a pro-algebraic group $G$, let $G^0$ be the connected component of the identity; if $\hat{G}$ is the maximal pro-finite quotient of $G$ (parametrising representations with finite monodromy), then $G^0=\ker(G \to \hat{G})$.
\end{definition}

\begin{lemma}\label{zinfty}
 If $\Gamma$ is a pro-discrete group, then  we may make the identification 
$$
\Gamma^{\alg,0}= \lim_{\substack{ \lla \\ \Delta }} \Delta^{\alg},
$$ 
where $\Delta$ runs over $\Delta \lhd \Gamma$ open of finite index.

Thus the category of finite-dimensional $\Gamma^{\alg,0}$-representations is the direct limit $\LLim_{\Delta}\FD\Rep(\Delta)$ (over $\Delta$ as above) of the categories of finite-dimensional $\Delta$-representations. 
\end{lemma}
\begin{proof}
This is essentially \cite[Proposition 2]{magid}, which deals with the case when $\Gamma$ is discrete, and refers to $\LLim_{\Delta}\FD\Rep(\Delta)$ as the category of virtual $\Gamma$-representations.

First note that $\widehat{\Gamma^{\alg}}=\hat{\Gamma}$, where the pro-finite completion $\hat{\Gamma}$ of  $\Gamma$ is characterised by the property that $\Hom_{\pro(\gp)}(\Gamma, F) \cong \Hom_{\pro(\gp)}(\hat{\Gamma}, F)$ for all finite groups $F$. Thus $\hat{\Gamma}=\Gamma$ whenever $\Gamma$ is pro-finite.
 
The exact sequence $\Delta \to \Gamma \to  \Gamma/\Delta\to 1$ gives an exact sequence $(\Delta)^{\alg} \xra{\alpha} \Gamma^{\alg} \to \Gamma/\Delta\to 1$. It suffices to show that $\alpha$ is injective. This follows from the observation that every finite-dimensional $\Delta$-representation $V$ embeds into a finite-dimensional $\Gamma$-representation $\Ind_{\Delta}^{\Gamma}V$.
\end{proof}

Thus if $F$ is a generator for $\Z$, then representations of $\Z^{\alg,0}$ are sums of $F^r$-representations, with morphisms commuting locally with sufficiently high powers of $F$.

Observe that we have  commutative diagrams
$$
\begin{CD}
\hat{\Z} @>r>> \hat{\Z}\\
@VVV @VVV\\
M_{q^r} @>>> M_q.
\end{CD}
$$ 
Any $\hat{\Z}$-representation with finite monodromy is pure of weight $0$, giving a map
$
P_p \to \hat{\Z}.
$
Also note that $M_{q^r}=\ker(M_q \to \Z/r\Z)$. Combining these observations gives:

\begin{lemma}
$$
M_p^0= \Lim M_{p^r}, \quad P_p^0= \Lim P_{p^r};
$$ 
writing $M^0:=M_p^0$ and $P^0:= P_p^0$, 
there are quotient maps $\hat{\Z}^{\red,0}\onto M^0 \onto  P^0$. There are similar results for $M_{\iota}^0:= M_{\iota, p}^0$, $P_{\iota}^0:= P_{\iota, p}^0$.
\end{lemma}

\begin{definition}
We say that a representation of $\Z^{\alg,0}$ is \emph{mixed} (resp. \emph{pure of weight} $0$, resp. $\iota$-\emph{mixed with integral weights}, resp. $\iota$-\emph{pure}) if the action of $\Z^{\red,0}\lhd \Z^{\alg,0}$ factors through $M^0$ (resp.  $P^0$, resp. $M_{\iota}^0$, resp. $P_{\iota}^0$).\end{definition}

\begin{lemma}\label{inftywgt}
Observe that the canonical maps $\bG_m \to M_q$ are compatible, giving $\bG_m \to M^0$, with trivial image in $P^0$. Similarly, we have $\bG_m\to M_{\iota}^0$, with trivial image in $P_{\iota}^0$.
\end{lemma}

\subsubsection{Slope decompositions}

\begin{definition}\label{ginfty}
Define the pro-algebraic group $\widetilde{\bG_m}$ to be the inverse limit of the \'etale universal covering system of $\bG_m$. 
This is the  inverse system  $\{G_r\}_{r \in \N}$ with  $G_r=\bG_m$ and  morphisms $G_{sr} \xra{[s]} G_r$, for $s \in \N$. 
\end{definition}

\begin{lemma}
The category of $\widetilde{\bG_m}$-representations is canonically equivalent to the category of $\Q$-graded vector spaces. 
\end{lemma}
\begin{proof}
A representation of $\bG_m$ is equivalent to a $\Z$-grading. Given a finite-dimensional vector space $V$ with a $\Q$-grading $V=\bigoplus V_{\lambda}$, let $d$ be the lowest common multiple of the denominators of the set $\{\lambda \in \Q: V_{\lambda}\ne 0\}$.  Then $V=\bigoplus_{n \in \Z} V_{n/d}$, giving a $\bG_m$-action on $V$. If we regard this copy of $\bG_m$ as $G_d$, this defines a $\widetilde{\bG_m} $-action.

Now, for any pro-algebraic group $G$, arbitrary $G$-representations are nested unions of finite-dimensional $G$-subrepresentations. Likewise, every $\Q$-graded vector space can be expressed as a nested union of finite-dimensional  $\Q$-graded vector subspaces, so the two categories are equivalent. 
\end{proof}

Now assume that $p=\ell$.

\begin{definition}\label{slope}
Given a power $q$ of $p$, normalise the $p$-adic valuation $v$ on $\bar{\Q}_p$ by $v(q)=1$. Define the \emph{slope} of $\alpha\in  \bar{\Q}_p$ to be $v(\alpha) \in \Q$.
\end{definition}

\begin{lemma}\label{slopez}
There is a canonical  morphism $\widetilde{\bG_m} \to \Z^{\red}$, corresponding to the functor sending a $\Z$-representation $V$ to a slope decomposition $\bigoplus V_{\lambda}$. 
\end{lemma}
\begin{proof}
Let $F$ be the canonical generator for $\Z$. Given a finite-dimensional semisimple $\Z$-representation $V$, we may decompose $V\ten_{\Q_p}\bar{\Q}_p$ into $F$-eigenspaces, and hence take a decomposition by slopes of the eigenvalues. Since conjugates in $\bar{\Q}_p$ have the same slope, this descends to a slope decomposition $V = \bigoplus_{\lambda \in \Q} V_{\lambda}$, as required.
\end{proof}

\subsection{Potentially unramified actions}\label{pnr}

Fix a prime $p \ne \ell$, and take a local field $K$, with  finite residue field $k$ of characteristic $p$. 
Let $\cG:= \Gal(\bar{K}/K)^{\alg}$, the pro-algebraic completion of $\Gal(\bar{K}/K)$ over $\Ql$.
\begin{definition}
Say that a finite-dimensional continuous $\Ql$-representation of $\Gal(\bar{K}/K)$ is \emph{potentially unramified} if there exists a finite extension $K'/K$ for which the action of  $\Gal(\bar{K}/K')$ is unramified. Say that an arbitrary $\Ql$-representation of $\Gal(\bar{K}/K)$ is \emph{potentially unramified} if it is a  sum of  finite-dimensional potentially unramified  representations.

These form a neutral Tannakian category (see Definition \ref{tannaka}); let $\cG^{\pnr}$ be the corresponding pro-algebraic group. Since $\Rep(\cG^{\pnr})$ is a Tannakian subcategory of $\Rep(\cG)$, $\cG^{\pnr}$ is a quotient of $\cG$.
\end{definition}

\begin{lemma}
We can write $\cG^{\pnr}= \Gal(\bar{k}/k)^{\alg}\by_{\Gal(\bar{k}/k)}\Gal(\bar{K}/K)$, so $\cG^{\pnr,0}= \Gal(\bar{k}/k)^{\alg,0} \cong \hat{\Z}^{\alg,0}$.
\end{lemma}
\begin{proof}
A representation $\cG \to \GL(V)$ (for $V$ finite-dimensional) is potentially unramified if it annihilates $\ker(\Gal(\bar{K}/K')\to \Gal(\bar{k}/k'))$ for some finite Galois extension $K'/K$. In other words, it annihilates $\ker(\Gal(\bar{K}/K) \to\Gal(\bar{k}/k)\by_{\Gal(k'/k)} \Gal(K'/K))$, so is an algebraic  representation of $\Gal(\bar{k}/k)^{\alg}\by_{\Gal(k'/k)}\Gal(K'/K)$.
 Thus the category of finite-dimensional $\cG^{\pnr}$-representations is given by
\begin{eqnarray*}
\FD\Rep(\cG^{\pnr}) &=& \LLim_{K'}\FD\Rep(\Gal(\bar{k}/k)^{\alg}\by_{\Gal(k'/k)}\Gal(K'/K))\\
&=& \FD\Rep( \Lim_{K'}\Gal(\bar{k}/k)^{\alg}\by_{\Gal(k'/k)}\Gal(K'/K))\\
&=& \FD\Rep(\Gal(\bar{k}/k)^{\alg}\by_{\Gal(\bar{k}/k)}\Gal(\bar{K}/K)),
\end{eqnarray*}
as required.

The final statement is an immediate consequence of Lemma \ref{zinfty}.
\end{proof}

\begin{definition}\label{pnrmixed}
We say that a representation of $\cG^{\pnr} $ is \emph{mixed} (resp. \emph{pure of weight} $0$) if the resulting action of $\Z^{\alg,0}\onto \hat{\Z}^{\alg,0}$ is so.
\end{definition}

\subsection{Potentially crystalline actions}\label{pcris}

Now let $\ell=p$, and  take a local field $K$, with  finite residue field $k$ of order $q=p^f$. 
Let $\cG:= \Gal(\bar{K}/K)^{\alg}$, the pro-algebraic completion of $\Gal(\bar{K}/K)$ over $\Q_p$. Let $W:=W(k)$, with fraction field $K_0$, and let $\sigma$ denote the unique  lift of arithmetic Frobenius $\Phi \in \Gal(\bar{k}/\bF_p)$ to $\sigma \in \Gal(K_0^{\nr}/\Q_p)$, for  $K_0^{\nr}$  the maximal unramified extension of $K_0$. Note that the geometric Frobenius of the previous section is $F=\Phi^{-f}$.

\begin{definition}
Say that a finite-dimensional continuous $\Gal(\bar{K}/K)$-representation over $\Q_p$ is \emph{potentially crystalline} if there exists a finite extension $K'/K$ for which the action of  $\Gal(\bar{K}/K')$ is crystalline. Say that an arbitrary $\Q_p$-representation of $\Gal(\bar{K}/K) $ is \emph{potentially crystalline} if it is a  sum of  finite-dimensional potentially crystalline  representations. Note that since unramified representations are automatically crystalline,  all potentially unramified representations are potentially crystalline.

These form a neutral Tannakian category (see Definition \ref{tannaka}); let $\cG^{\pcris}$ be the corresponding pro-algebraic group. Since $\Rep(\cG^{\pcris})$ is a full subcategory of $\Rep(\cG)$ closed under subobjects, $\cG^{\pcris}$ is a quotient of $\cG$.
\end{definition}

\begin{definition}\label{bcrisdef}
In \cite[\S 4]{fontaineBT}, Fontaine defined a ring  $B_{\cris}:=B_{\cris}(V)$ of periods over $\Q_p$, equipped with a Hodge filtration and actions of $\Gal(\bar{K}/K)$ and Frobenius,  and used it to characterise crystalline representations (adapted in Proposition \ref{pcristest} below). 

In \cite[6.8]{olssonhodge}, Olsson defined a  localisation  $\tilde{B}_{\cris}(V)$ of $B_{\cris}(V)$   as follows. Fix a sequence $\tau_m$ of elements of $\bar{V}$ with $\tau_0=p$ and $\tau_{m+1}^p=\tau_m$ for all $m\ge 0$. Define $\lambda_{p^{-n}}$ to be the sequence $(\tau_{n+m})_{m\ge 0}$, and let $\delta_{p^{-n}}$ be the associated Teichm\"uller lifting. Set
$$
\tilde{B}_{\cris}(V):= B_{\cris}(V)[\delta_{p^{-n}}^{-1}]_{n \ge 0},
$$
noting that $(\delta_{p^{-n-1}})^p= \delta_{p^{-n}}$.
\end{definition}

\begin{definition}
Given a finite-dimensional $\Gal(\bar{K}/K)$-representation $U$, set  
\begin{eqnarray*}
D_{\cris,K}(U)&:=&  (U\ten_{\Q_p}B_{\cris})^{\Gal(\bar{K}/K)},\\
D_{\pcris}(U)&:=& \LLim D_{\cris,K'}(U)  ,\\
\tilde{D}_{\cris,K}(U)&:=&  (U\ten_{\Q_p}\tilde{B}_{\cris})^{\Gal(\bar{K}/K)},\\
\tilde{D}_{\pcris}(U)&:=& \LLim \tilde{D}_{\cris,K'}(U)  ,
\end{eqnarray*}
for $K'$ ranging over all finite extensions of $K$. For an arbitrary algebraic $\Gal(\bar{K}/K)$-representation $U$, set
$$
D_{\pcris}(U):= \varinjlim D_{\pcris}(U_{\alpha}),
$$
for $U_{\alpha}$ running over all finite-dimensional subrepresentations, and similarly for $\tilde{D}_{\pcris}$.
\end{definition}

Observe that $\Spec B_{\cris}$ is an affine $\cG$-scheme over $\Spec \Q_p$, and that the coarse quotient $(\Spec B_{\cris})/\cG^0$ is  $\Spec K_0^{\nr}$.

\begin{proposition}\label{pcristest}
An action of $\cG$ on an affine $\Q_p$-scheme $Y$ factors through $\cG^{\pcris}$ if and only if there exists an affine $K_0^{\nr}$-scheme $Z$, with
$$
Y \by_{\Q_p} \Spec \tilde{B}_{\cris} \cong Z\by_{ K_0^{\nr}} \Spec \tilde{B}_{\cris}
$$
a $\cG^0$-equivariant map (for trivial $\cG^0$-action on $Z$).

In that case, we necessarily have $\O_Z=D_{\pcris}(\O_Y)= \tilde{D}_{\pcris}(\O_Y)$. 
\end{proposition}
\begin{proof}
If we replace  potentially crystalline with crystalline, and $K_0^{\nr}$ with $K_0$, then this is just \cite[Theorem D.3]{olssonhodge}. Taking the direct limit over finite extensions of $K$ gives the first expression.

 Taking $\cG^0$-invariants gives  $\O_Z= \tilde{D}_{\pcris}(\O_Y)$, but then  \cite[Remark D.10]{olssonhodge}  shows that for potentially crystalline representations $U$,  $\tilde{D}_{\pcris}(U)= D_{\pcris}(U)$.
\end{proof}

\subsubsection{Frobenius actions}
Although we do not have a canonical map $\Z^{\alg,0} \to \cG^{\pcris}$, there is something nearly as strong:

\begin{lemma}\label{pcrisfrob}
There is a canonical morphism
$$
\Z^{\alg,0}\ten_{\Q_p}B_{\cris}^{\sigma} \to \cG^{\pcris}\ten_{\Q_p}B_{\cris}^{\sigma}
$$
of affine  group schemes over the $\sigma$-invariant subring $B_{\cris}^{\sigma}$ of $B_{\cris}$.
\end{lemma}
\begin{proof}
Given $U \in \FD\Rep(\cG^{\pcris})$, $U$ is crystalline over $K'$ for some finite extension $K'/K$ with residue field $k'$. If $|k'/k|=r$ and $q=p^f$, then $\phi^{fr}$ is a $K'_0$-linear endomorphism of $D_{\cris, K'}(U)$. This extends uniquely to give a  $K_0^{\nr}$-linear automorphism $F_r$  of $D_{\pcris}(U)$ (note that $F_r\ne \phi^{fr}$, the latter being $\sigma$-semilinear). 

Now, observe that $D_{\pcris}(U)$ is a  sum of finite-dimensional  $F_r$-representations over $\Q_p$, since $D_{\cris, K'}(U)$ is finite-dimensional over $K'$, and hence over $\Q_p$. This gives us a ${\sigma}$-equivariant $\Q_p$-linear action of $\Z^{0,\alg}$ on $D_{\pcris}(U)$, and hence a ${\sigma}$-equivariant $B_{\cris}^{\sigma} $-linear action on $D_{\pcris}(U)\ten_{K_0^{\nr}}B_{\cris}=U\ten_{\Q_p} B_{\cris}$. We now take the $\phi$-invariant subspace, giving a $\Z^{0,\alg}\ten_{\Q_p}B_{\cris}^{\sigma}$-action on $U\ten_{\Q_p} B_{\cris}^{\sigma}$. 

If we took a larger extension $K''/K$ with residue field $k''$, the we would have $|k''/k|=s$ with $r|s$. The corresponding $K_0^{\nr}$-linear automorphism $F_s$  of $D_{\pcris}(U)$ is given by $F_s= F_r^{s/r}$, so gives rise to the same $\Z^{0,\alg}$-action on $D_{\pcris}(U)$. This ensures that the action is functorial in $U$.

Given $U,V \in \FD\Rep(\cG^{\pcris})$, we have $D_{\pcris}(U\ten_{\Q_p}V)= D_{\pcris}(U)\ten_{ K_0^{\nr}}D_{\pcris}(V)$, compatible with $\phi$. Choosing $K'$ so  that $U,V$ are both crystalline over $K'$, we see that  $D_{\pcris}(U\ten_{\Q_p}V)$ is isomorphic to $D_{\pcris}(U)\ten_{ K_0^{\nr}}D_{\pcris}(V)$ as an $F_r$-representation.

Hence the $\Z^{0,\alg}\ten_{\Q_p}B_{\cris}^{\sigma}$-representation  $(U\ten_{\Q_p}V)\ten_{\Q_p} B_{\cris}^{\sigma}$ is isomorphic to $(U\ten_{\Q_p} B_{\cris}^{\sigma})\ten_{B_{\cris}^{\sigma}} (V\ten_{\Q_p}B_{\cris}^{\sigma}) $.

For a $\Q_p$-algebra $A$, Tannakian duality says that  giving an element  $g \in \cG^{\pcris}(A)$ is equivalent to giving $A$-linear automorphisms  $g_U$ of $U \ten A$  for all $\cG^{\pcris}$-representations $U$, functorial and compatible with tensor products and duals.  Therefore the $\Z^{0,\alg}\ten_{\Q_p}B_{\cris}^{\sigma}$-actions on the representations  $U\ten_{\Q_p} B_{\cris}^{\sigma}$ give  group homomorphisms $ \Z^{0,\alg}(C) \to \cG^{\pcris}(C)$, functorial in $B_{\cris}^{\sigma}$-algebras $C$, as required.
\end{proof}

\begin{definition}\label{crismixed}
We say that a potentially crystalline representation $U$ is \emph{mixed} (resp. \emph{pure}, resp. $\iota$-\emph{mixed with integral weights}, resp. $\iota$-\emph{pure}) if the action of $\Z^{\alg,0}\ten B_{\cris}^{\sigma}$ on $U \ten B_{\cris}^{\sigma}$ factors through $M_q$ (resp. $P_q$, resp. $M_{\iota, q}$, resp. $P_{\iota, q}$). This is equivalent to saying that the action of $\Z$ on $D_{\pcris}(U)$ is mixed (resp. pure, resp. $\iota$-mixed with integral weights, resp. $\iota$-pure).
\end{definition}

We have the following analogue of a slope decomposition:
\begin{lemma}
There is a canonical morphism $\widetilde{\bG_m} \to \cG^{\pcris}\ten_{\Q_p}B_{\cris}^{\sigma}$ of affine  group schemes over $B_{\cris}^{\sigma}$, for $\widetilde{\bG_m} $ as in definition \ref{ginfty}.
\end{lemma}
\begin{proof}
Combine Lemma \ref{slopez} with Lemma \ref{pcrisfrob}.
\end{proof}

\section{Varieties over finite fields}\label{finite}

Fix a variety $X_k$ over a finite field $k$, of order $q$ prime to $\ell$. Let $X:= X_k \ten_k \bar{k}$, for $\bar{k}$ the algebraic closure of $k$. There is a Galois action on $X$, and hence on the pro-simplicial set $X_{\et}$, and on its algebraisation $G(X_{\et})^{\alg}$. The purpose of this section is to describe this action as far as possible. 

\subsection{Algebraising the Weil groupoid}

The morphism $X \to X_k$ gives a map of groupoids $\alpha\co \pi^{\et}_fX \to \pi^{\et}_f(X_k)$. Similarly, there is a map $\pi^{\et}_fX_k \to \pi^{\et}_f(\Spec k) = \Gal(\bar{k}/k) \cong \hat{\Z}$. Denote the canonical generator of $\Gal(\bar{k}/k)$ by $F$, the geometric Frobenius automorphism.

In constructing fundamental groupoids and \'etale homotopy types, we may use the same set of geometric points for both $X_k$ and $X$, so assume that $\alpha$ is an isomorphism on objects. We then have
$$
\pi^{\et}_f(X)= \pi^{\et}_f(X_k)\by_{\hat{\Z}}0.
$$
 
\begin{definition}
Define the \emph{Weil groupoid} $W_f(X_k)$ by
$$
W_f(X_k):= \pi^{\et}_f(X_k)\by_{\hat{\Z}}\Z,
$$ 
noting that this is a pro-groupoid with discrete objects.
\end{definition}

For any scheme $Y$, note that finite-dimensional representations of $ \varpi^{\et}_f(Y):= \varpi_f(\widehat{Y_{\et}})$ correspond to smooth $\Ql$-sheaves on $Y$. We now introduce natural quotients of this groupoid.

\begin{definition}\label{walg}
Define $\w\varpi^{\et}_f(X)$ to be the image of  $ \varpi^{\et}_f(X)\to W_f(X_k)^{\alg}$, so $W_f(X_k)^{\alg}= \w\varpi^{\et}_f(X) \ltimes \Z^{\alg}$. 

Define ${}^{\Gal}\!\varpi^{\et}_f(X)$ to be the image of  $ \varpi^{\et}_f(X)\to \varpi^{\et}_f(X_k)$, so $\varpi^{\et}_f(X_k)= {}^{\Gal}\!\varpi^{\et}_f(X) \ltimes \hat{\Z}^{\alg}$. Note that ${}^{\Gal}\!\varpi^{\et}_f(X)$ is a Frobenius-equivariant quotient of $\w\varpi^{\et}_f(X)$ (it is in fact the quotient on which $\hat{\Z}$ acts continuously).
\end{definition}

In \cite{weight1}, ${}^W\!\algpia$ was defined to be the universal object classifying continuous $W(X_k,x)$-equivariant homomorphisms
$
\pi_1(X,\bar{x}) \to G(\Ql)
$ to algebraic groups. In the terminology of \cite[Definition 1.3]{weight1}, ${}^W\!\algpia$ is the maximal quotient of $\algpia$ on which Frobenius acts algebraically.

Note that these definitions are consistent by \cite[Lemma \ref{weight1-wf}]{weight1}, which proceeds by establishing an action of $\Z^{\alg}$ on ${}^W\!\algpia$ generated by Frobenius, then showing that the map  $ \Z^{\alg}\ltimes {}^W\!\algpia\to W(X_k,x)^{\alg}$ is an isomorphism.

It also implies that 
linear representations of $\w\varpi^{\et}_f(X)$ correspond to smooth $\Ql$-sheaves on $X$ arising as subsheaves of Weil sheaves, while  linear 
representations of ${}^{\Gal}\!\varpi^{\et}_f(X)$ correspond to smooth $\Ql$-sheaves on $X$ arising as subsheaves of pullbacks of  smooth $\Ql$-sheaves on $X_k$.

\begin{lemma}
The canonical action of $F$ on $\w\varpi^{\et}_f(X)$ factors through a morphism
$$
\Z^{\alg} \to \Aut(\w\varpi^{\et}_f(X))
$$
of group presheaves, for $\Z^{\alg}$ as in \S \ref{wgtdecomp}.
\end{lemma}
\begin{proof}
Write $G= \w\varpi^{\et}_f(X), H= W_f(X_k)^{\alg}$,
and observe that the orbits of $F$ in $\Ob G=\Ob H$ are finite, giving a map
$$
\hat{\Z} \to \Aut(\Ob H).
$$
Since $\hat{\Z}$ is pro-finite, we may regard it as the pro-algebraic group $\Z^{\alg}/\Z^{\alg,0}$.

Now, consider the group scheme 
$$
N:=\coprod_{f \in \Aut(\Ob(H))} \prod_{x \in \Ob(H)} H(x,fx),
$$
with multiplication given by 
$$
(f,\{h_x\})\cdot (f',\{h'_x\}) = (f\cdot f', \{h_{f'x}\cdot h_x\}).
$$
There is a morphism
 $N \to \Aut(\Ob(H))$ fibred in affine schemes. Thus
$$
\hat{\Z}\by_{\Aut(\Ob(H))}N
$$
is an affine scheme. 

Now, $F$ gives a collection of paths  $F(x) \in  W_f(X_k)(x,Fx)$, and thus a map
$$
\Z \to (\hat{\Z}\by_{\Aut(\Ob(H))}N)(\Ql).
$$
Since the latter is an affine group scheme, this extends to a map $\Z^{\alg} \to \hat{\Z}\by_{\Aut(\Ob(H))}N$. Finally, observe that the conjugation action of $H$ on $G$ gives a map
$$
N \to \Aut(G).
$$
\end{proof}

\begin{theorem}\label{laff}
If $X_k/k$ is normal, then the action of $\Z^{\red}$ on  $\w\varpi^{\et}_f(X)^{\red}$ factors through $P_q$ (see Definition \ref{pqmq}); in other words, the Frobenius representation $O(\w\varpi^{\et}_f(X)^{\red})$ is a sum of finite-dimensional Galois representations, pure of weight $0$.
 
Moreover $\w\varpi^{\et}_f(X)^{\red}= {}^{\Gal}\!\varpi^{\et}_f(X)^{\red}$, so the $\Z^{\red}$ action factors through its quotient $\hat{\Z}^{\red}$.
\end{theorem}
\begin{proof}
Since $\Z^{\alg}= \Z^{\red}\by \bG_a$ (\S \ref{wgtdecomp}), this  amounts to showing that the Frobenius action factors through $P_q \by \bG_a$. We adapt the proof of \cite[Theorem \ref{weight1-weilred}]{weight1}  (to which we refer the reader for details).

Let $T$ be the set of  all  isomorphism classes of irreducible representations $V$ of $\w\varpi^{\et}_f(X)^{\red}$ over $\bar{\Ql}$.
Since $\w\varpi^{\et}_f(X)^{\red}$ is reductive, there is an isomorphism of $\w\varpi^{\et}_f(X)^{\red}\by\w\varpi^{\et}_f(X)^{\red}$-representations given on objects $(x,y)$ by
$$
O( \w\varpi^{\et}_f(X)^{\red}(x,y))\ten_{\Ql}\bar{\Ql}\cong \bigoplus_{V \in T} \Hom(V_x, V_y).
$$
If $\vv$ is the smooth sheaf on $X$ corresponding to the representation $V$, then $\bigoplus_{V \in T} \Hom(V_x, V_y)$ corresponds to the smooth sheaf
$$
\bigoplus_{V \in T} \pr_1^{-1}\vv^{\vee}\ten\pr_2^{-1}\vv
$$
on $X\by X$.

Now, $V\in T$ is an irreducible representation of $\varpi^{\et}_f(X)^{\red} $ which is a subrepresentation of some $W_f(X_k)$-representation. This is the same as underlying a $W_f(X_{k'})$-representation for some finite extension $k'/k$, so $\vv$ underlies a smooth Weil sheaf on $X_{k'}$.

From Lafforgue's Theorem (\cite[Conjecture 1.2.10]{Weil2}, proved in \cite[Theorem VII.6 and Corollary VII.8]{La}), every  irreducible smooth Weil sheaf over $\bar{\Ql}$ is of the form
$$
\vv \cong P \ten \bar{\Ql}^{(b)},
$$
for some  mixed  sheaf $P$ on $X_{k'}$. By \cite[Theorem 3.4.1 (ii)]{Weil2}, 
every irreducible smooth $\iota$-mixed Weil sheaf is $\iota$-pure. Thus the mixed sheaf $P$ is $\iota$-pure for all $\iota$, and hence pure. 

Thus
$$
\pr_1^{-1}\vv^{\vee}\ten\pr_2^{-1}\vv\cong \pr_1^{-1}P^{\vee}\ten\pr_2^{-1}P,
$$
which is a smooth sheaf on $X_{k'}\by X_{k'}$,  pure   of weight $0$. 

Therefore $O( \w\varpi^{\et}_f(X)^{\red})\ten_{\Ql}\bar{\Ql}$, and hence $O( \w\varpi^{\et}_f(X)^{\red})$, is a  pure Galois representation of weight $0$. Thus the action of $\Z^{\alg}$ factors through   $P_q \by \bG_a$, and
the discrete Galois action on $\varpi^{\et}_f(X)^{\red} $ descends to a continuous action on $ \w\varpi^{\et}_f(X)^{\red}$, so
$$
 \w\varpi^{\et}_f(X)^{\red}= {}^{\Gal}\!\varpi^{\et}_f(X)^{\red}.
$$ 
\end{proof}

\subsection{Weight decompositions}
Now assume that $X$ is  either smooth or proper and normal.

\begin{definition}
  Define a \emph{weight decomposition} on  a multipointed homotopy type $G \in \Ho(s\cE(R)_*)$  to be a morphism
$$
\bG_m \to \RAut(G)
$$
of pro-algebraic groups.
\end{definition}
Compare this with   \cite[Definition \ref{htpy-wgtdef}]{htpy}, which considers weight decompositions on unpointed homotopy types, corresponding to outer automorphisms.

\begin{proposition}\label{wgtexists}
If we let $R$ be any Frobenius-equivariant quotient of $\w\varpi^{\et}_f(X)^{\red}$, then the Galois action  on   
$$
X_{\et}^{R, \mal}
$$
is mixed, giving a canonical weight decomposition. Furthermore, the Frobenius action extends canonically to a continuous algebraic $\Gal(\bar{k}/k)$-action.
\end{proposition}
\begin{proof}
By Theorem \ref{laff}, the Galois action on $R$ factors through the quotient $P_q\by \bG_a$ of $\Z^{\alg}$. By Corollary \ref{redalgebraic}, the $\Gal(\bar{k}/k)$ action on $X_{\et}^{R, \mal}$ is thus algebraic. Since $R$ is a $P_q\by \bG_a$-representation, the Weil sheaf  $\bigoplus_{x \in \Ob R} \bO(R)(x)$ is an arithmetic  sheaf of weight $0$. Deligne's Weil II theorems (\cite[Corollaries 3.3.4 -- 3.3.6]{Weil2})  then imply that $\bigoplus_{x \in X} \H^*(X,\bO(R)(x))$ is a mixed $\Gal(\bar{k}/k)$  representation (i.e. a representation of $M_q \by \bG_a$). By Lemma \ref{cohohelps}, we may therefore conclude that the  action of $\Z^{\red}$ on $X_{\et}^{R, \mal}$ factors through $M_q$, giving
$$
M_q \to \Aut^h(X_{\et}^{R, \mal}).
$$

Finally, use the map $\bG_m \to M_q$ (given after Definition \ref{nq}) to define the weight decomposition. Since $R$ is pure of weight zero, the $\bG_m$-action on $R$ is trivial, giving
$$
\bG_m \to \Aut^h(X_{\et}^{R, \mal}),
$$
as required.
\end{proof}

\begin{corollary}\label{wgtexistspin}
The Galois actions are mixed on the duals $\varpi_{n}(X_{\et}^{R,\mal},x)^{\vee}$ of the  homotopy groups for $n \ge 2$, and on the structure sheaves $O(\varpi_{f}(X_{\et}^{R,\mal}))(x,y)$. In particular, these objects have canonical weight decompositions.
\end{corollary}
\begin{proof}
This is just the observation that there are canonical  maps $\Aut^h(X_{\et}^{R, \mal}) \to \Aut(\varpi_{n}(X_{\et}^{R,\mal},x))$ and $\Aut^h(X_{\et}^{R, \mal}) \to \Aut(O(\varpi_{f}(X_{\et}^{R,\mal}))(x,y)  )$ of group-valued presheaves, so Proposition \ref{wgtexists} gives algebraic actions of $M_q \by \bG_a$ (and hence $\bG_m$) on the homotopy groups and fundamental groupoid. 
\end{proof}

\begin{remark}
We have shown that $\varpi_{n}(X_{\et}^{R,\mal})$ is a mixed $\pi_f^{\et}(X_k)$-representation. In particular, this means that $\varpi_{n}(X_{\et}^{R,\mal}, \bar{x})$ is a mixed $\widehat{\langle F_x\rangle }$-representation, so has a canonical weight decomposition.
\end{remark}

\begin{remark}\label{wgtexistspinet}
If the hypotheses of Theorem \ref{etpimal} hold and $\pi_1^{\et}(X,x)$ is $N$-good relative to $R$, then Corollary \ref{wgtexistspin} implies that the Galois actions on the $\pi_n^{\et}(X,x)\ten_{\hat{\Z}}\Ql$ are mixed for $n \le N$. 

Alternatively, if  it should happen that the Galois action on $\H^n(\pi_1^{\et}(X,x), V)$ is mixed for all $R$-representations $V$ underlying pure $\pi_1^{\et}(X_k,x)$-representations and all $n \le N$, then  Lemma \ref{cohohelps} (combined with the Adams spectral sequence of  \cite[Proposition \ref{htpy-spectralh}]{htpy})
 implies that the Galois actions on  $\varpi_n(B\pi_1(X,x))^{R, \mal}$ is  mixed for $n \le N$.  Provided the first two hypotheses of Theorem \ref{etpimal}  hold, the  exact sequence of that theorem would then imply that the Galois actions on $\pi_n^{\et}(X,x)\ten_{\hat{\Z}}\Ql$ are also mixed.
\end{remark}

\subsection{Formality}

Now assume that $X$ is smooth and proper. Deligne's Weil II theorems then imply that $\bigoplus_{x \in X} \H^n(X,\bO(R)(x))$ is pure of weight $n$. 

\begin{theorem}\label{formal}
For $R$ as in Proposition  \ref{wgtexists}, the Malcev homotopy type $X_{\et}^{R, \mal} \in s\cE(R)$ is formal, in the sense that it corresponds (under the equivalences of  Proposition \ref{meequiv} and Theorem \ref{bigequiv}) to the $R$-representation 
$$
\H^*_{\et}(X,\bO(R)) 
$$
in cochain algebras, equipped with the unique augmentation map $\Ql =\H^0(X, \bO(R))\to \prod_{x \in \Ob R}O(R)(x,-)$. This isomorphism is Galois equivariant.
\end{theorem}
\begin{proof}
We need  to construct an isomorphism $\theta\co  N\Ru(X_{\et}^{\rho, \mal}) \cong G  \H^*_{\et}(X,\bO(R))$ in $dg\cP(R)$ (for $G$ as in Definition \ref{barwg}), such that 
$\ad_{\theta}\co  \Aut^h( X_{\et}^{\rho, \mal}) \to \Aut^h ( G \Spec D\H^*_{\et}(X,\bO(R))\rtimes R)$ satisfies $\ad_{\theta}F=F$.

As in \S \ref{minimalsn}, take a minimal model $\m$ for $N\Ru(X_{\et}^{\rho, \mal})\in dg\hat{\cN}(R)$. This has the property that $\m_n/[\m,\m]_n \cong \H^{n+1}(X, \bO(R))^{\vee}$. 

From the proof of  Lemma \ref{auto}, we know that
$$
\prod_{x \in \Ob R}\exp(\H_0\m(x))\by  \Aut_{dg\hat{\cN}_A(R)}(\m\hat{\ten}A) \to \RAut(X_{\et}^{\rho, \mal })(A)
$$
is a pro-unipotent extension of pro-algebraic groups. 

Likewise,   the maps
\begin{eqnarray*}
&\Aut(R\ltimes \exp(\m)) \onto \Aut^h(X_{\et}^{\rho, \mal}) \to\\  
&\{(f,\alpha) \,:\, f\in \Aut(R), \alpha \in \Iso_{DG\Alg(R) }(\H^*_{\et}(X,\bO(R)),  f^{\sharp}\H^*_{\et}(X,\bO(R)))\}
\end{eqnarray*}
both have pro-unipotent kernels. 

We may therefore lift the map $\hat{\Z}^{\alg} \to \Aut^h(X_{\et}^{\rho, \mal})$ to give  $\hat{\Z}^{\alg} \to \Aut(R\ltimes \exp(\m))$. 
This gives a lift of  the weight decomposition $\bG_m \to \RAut(X_{\et}^{\rho, \mal})$  to $ \bG_m \to  \Aut(R\ltimes \exp(\m))$. Since $\m$ is of strictly negative weights, we may adapt \cite[Corollary 1.21]{weight1}  by observing that $O(  R\ltimes \exp(\m))/O(R)$ is of strictly positive weights, and that the weight $0$ part $\cW_0O(  R\ltimes \exp(\m))$ is just O(R), so we have a $\hat{\Z}^{\alg}$-equivariant decomposition
$$
O(  R\ltimes \exp(\m))= O(R) \oplus \cW_+ O(  R\ltimes \exp(\m)).
$$
This  amounts to giving a $\hat{\Z}^{\alg}$-equivariant section of $R\ltimes \exp(\m) \to R$, or equivalently a $\hat{\Z}^{\alg}$-equivariant Levi decomposition, so we may assume that the $\hat{\Z}^{\alg}$ action on $R\ltimes \exp(\m)$ consists of actions on $R$ and on $\m$.

Let $V_n:=\cW_{-n-1}\m_n$, for $\cW$ as in \S \ref{wgtdecomp}; since cohomology is pure, we deduce that $V_n \to \H^{n+1}(X, \bO(R))^{\vee}$ is an isomorphism, and that $\m$ is freely generated as a Lie algebra by the spaces $V_n$. The differential $d$ on $\m$ is then determined by $d\co V_n \to \m_{n-1}$, and  weight considerations show that the only non-zero contribution is $V_n \to \prod_{a+b=n-1}[V_a,V_b]$. This is isomorphic to $d\co \m/[\m,\m] \to [\m,\m]/[\m,[\m,\m]]$, so must be dual to the cup product. 

Therefore, the choice of lift $\hat{\Z}^{\alg} \to \Aut(R\ltimes \exp(\m))$ has determined an isomorphism $R\ltimes \exp(\m)\cong R \ltimes \exp(G \H^*_{\et}(X,\bO(R)))$, and this is automatically compatible with the Galois action $\hat{\Z}^{\alg} \to \Aut^h(R\ltimes \exp(\m))$.
\end{proof}

\begin{corollary}\label{formalrk}
If we let $R$ be any Frobenius-equivariant quotient of $\w\varpi^{\et}_f(X)^{\red}$, then the relative Malcev homotopy groups $\varpi^{\et}_n(X^{R, \mal}, x)$ can be described in terms of  cohomology as 
\[
  \varpi^{\et}_n(X^{R, \mal}, x)\cong  \H_{n-1}(G \H^*(X, \bO(R))),
\]
for $G$ as in Definition \ref{barwg}. This description is Galois-equivariant. If the conditions of Theorem \ref{etpimal} hold (including goodness), then this also calculates $\pi_n^{\et}(X,x)\ten_{\hat{\Z}} \Ql$ as a Galois representation. 
\end{corollary}

\subsection{Quasi-formality}

Let $j\co X \into\bar{X}$ be an open immersion of varieties over $\bar{k}$, such that  locally for the \'etale topology, the pair $(X,\bar{X})$ is isomorphic to   $(\bA^m\by\prod_i(\bA^{c_i}-\{0\}),  \bA^d)$, for some $d=m +\sum c_i$. Note that this is satisfied when $\bar{X}-X$ is a normal crossings divisor (corresponding to the case $c_i=1$ for all $i$).
It also includes all geometric fibrations over $\bar{k}$ in the sense of \cite[Definition 11.4]{fried}. 

\begin{definition}
For $X, \bar{X}$ as above, let $T=\bar{X}-X$, and let $D$ be the closed subscheme of $T$ of codimension $1$ in $\bar{X}$. Note that $\pi^{\et}_f(X)\to \pi^{\et}_f(\bar{X}-D)$ is an isomorphism, and   define $\pi_f^t(X):= \pi_f^t(\bar{X}-D)$ to be the \emph{tame fundamental groupoid} (as in \cite[XIII.2.1.3]{sga1}).

Define $\pi_f^t(X_k)$ similarly, with  the \emph{tame  Weil groupoid} $W_f^t(X_k)$ given by
$$
W_f^t(X_k):= \pi_f^t(X_k)\by_{\hat{\Z}}\Z.
$$ 

Let $\varpi^t_f(X):=\pi^t_f(X)^{\alg}$, and   define $\w\varpi_f^t(X)$ to be the image of  $ \varpi_f^t(X)\to W_f^t(X_k)^{\alg}$.

Given a local system $\vv$ on $X$, observe that the direct image $i_*\vv$ of $\vv$ under the inclusion  $i\co X \into \bar{X}-D$ is also a local system. We say that $\vv$ is  tamely ramified along the divisor if $i_*\vv$ is tamely ramified along  $D$ in the sense of \cite[Definition XIII.2.1.1]{sga1}.
\end{definition}

\begin{lemma}\label{lerayworks}
Take $j$ as above. If $\vv$ is a pure smooth Weil sheaf on $Y$ of weight zero, tamely ramified along  the divisor, then $\oR^{\nu}j_*\vv$ is pure of weight $2\nu$ (in the sense of \cite[Lemma-Definition II.12.7]{KW}).
\end{lemma}
\begin{proof}
This is a consequence of the following statements: 
\begin{enumerate}
\item $\oR^{\nu}j_*\vv$ is pointwise pure of weight $2\nu$;
\item the canonical map $(\oR^{\nu}j_*\vv)^{\vee} \to \oR\hom_{\bar{X}}(\oR^{\nu}j_*\vv, \Ql)$  is an isomorphism.
\end{enumerate}
If $0\to \vv' \to \vv \to \vv''\to 0$ is an exact sequence, with the statements holding for $\vv$ and $\vv''$, then observe that they also hold for $\vv$, since the long exact sequence must degenerate.  

The statements are local on $\bar{X}$.  \'Etale-locally, the pair $(X, \bar{X})$ is isomorphic to $(U,U')=(\bA^m\by\prod_i(\bA^{c_i}-\{0\}),  \bA^d)$, for $d=m +\sum c_i$. 
 We may then reduce to the case when  $\vv$ is irreducible on $U$, and so $\vv= \vv_m \boxtimes \bigotimes_i \vv_i$, for $\vv_i$ irreducible on $\bA^{c_i}-\{0\}$. By the K\"unneth formula, we now need only consider the pair $(\bA^c-\{0\}, \bA^c)$. 

If $\vv$ is constant, then the statements follow from the cohomological purity theorem (\cite[VI.5.1]{Mi}). Since the scheme $\bA^{c}-\{0\}$ is simply connected for $c>1$, this leaves only the case $c=1$. \cite[Lemma I.9.1]{KW}  shows that $j_*\vv$ is pure, and local calculations give $\oR^ij_*\vv=0$ for $i>0$ (since $\vv$ is  tamely ramified, and is non-constant irreducible).
\end{proof}

\begin{proposition}\label{divisorwgts}
Assume that $j\co X_k \into\bar{X}_k$ is a morphism over $k$, with $j\ten \bar{k}$ as in Lemma \ref{lerayworks}, for $\bar{X}_k$ proper. If $\vv$ is a pure smooth Weil sheaf on $X$ of weight zero,  tamely ramified along the divisor, then $\H^i(\bar{X}, \oR^{\nu}j_*\vv)$ is pure of weight $i+2\nu$, for $j\co X \to \bar{X}$  the compactification map.
\end{proposition}
\begin{proof}
By \cite[Corollary 3.3.4]{Weil2}, we know that $\H^i(\bar{X}, \oR^{\nu}j_*\vv)$ is mixed of weights $\le i+2\nu$, since $\oR^{\nu}j_*\vv$ is pure of weight $2\nu$. Now, Poincar\'e duality (\cite[Corollary II.7.3]{KW}) implies that
$$
\H^i(\bar{X}, \oR^{\nu}j_*\vv)^{\vee}\cong \H^{2d-i}(\bar{X}, (\oR^{\nu}j_*\vv)^{\vee})(2d), 
$$
which is mixed of weight $\le-i-2\nu$, using the isomorphism $(\oR^{\nu}j_*\vv)^{\vee} \cong \oR\hom_{\bar{X}}(\oR^{\nu}j_*\vv, \Ql)$ of Lemma \ref{lerayworks}. 
\end{proof}

\begin{corollary}\label{qformal}
For $X$ as above, and  $\rho\co \varpi^{\et}_fX \to R$ any Frobenius-equivariant quotient of $\w\varpi_f^t(X)^{\red}$, the filtered homotopy type $ (X^{\rho, \mal}, j)$ of Definition \ref{leraytype} is quasi-formal (in the sense of Definition \ref{qf}). The formality quasi-isomorphism is equivariant with respect to the Galois action.
\end{corollary}
\begin{proof}
This is largely the same as  Theorem \ref{formal}. 
Use the equivalences of Theorem \ref{fbigequiv} to take a filtered minimal model $(\m,J) \in Fs\hat{\cN}(R)$ for $ (X^{\rho, \mal}, j)$. The increasing filtration $J_*$ on $\m^{\vee}$ gives  a decreasing filtration $J^*$ on $\m$, with $J^r\m_n$ the annihilator of $J_{r-1}(\m^{\vee})$. Note that $[J^a\m,J^b\m] \subset J^{a+b}\m$ and $J^0\m=\m$.

If we write $\Aut_J(R\ltimes \exp(\m))$ for the group of filtered automorphisms of $R\ltimes \exp(\m)$, then similarly to Lemma \ref{frout}, the maps
\begin{eqnarray*}
&\Aut_J(R\ltimes \exp(\m)) \onto \Aut^h_J(X_{\et}^{\rho, \mal}) \to  \\
&\{(f,\alpha) \,:\, f\in \Aut(R), \alpha \in \Iso_{FDG\Alg(R) }(\H^*_{\et}(\bar{X},\oR^*j_* \bO(R)),  f^{\sharp}\H^*_{\et}(\bar{X},\oR^*j_* \bO(R)))\}
\end{eqnarray*}
both have pro-unipotent kernels. 

We may therefore lift the Galois action  $\hat{\Z}^{\alg} \to \Aut_J^h(X_{\et}^{\rho, \mal})$ to a filtered  automorphism of $R\ltimes \exp(\m)$. 
This gives a lift of  the weight decomposition $\bG_m \to \RAut_J(X_{\et}^{\rho, \mal})$, a unique  Galois-equivariant Levi decomposition of $R\ltimes \exp(\m) $, and a  weight decomposition  $\bG_m \to \Aut_J(\m)$. 

Now, $(\m^{\ab}_n)^{\vee} \cong  \bigoplus_{a+b=n+1} \H^{a}(\bar{X}, \oR^{b}j_*\bO(R))=:E^{n+1}$, on which $J_r$ is the subspace of weights $\le n+r+1$. Thus $J^r(\m^{\ab}_n)$ is the subspace of weights $\le -(n+r+1)$. 

Let  $\Gamma^r\m$ be the lower central series on $\m$, so $\Gamma^1\m=\m$ and $\Gamma^{r+1}\m= [\m, \Gamma^r\m]$.  
The weight restrictions on $\m^{\ab}$ show that $J^r(\gr^s_{\Gamma}\m)_n = J^r(\Lie_s(\m^{\ab}))_n$, which is of weights  $\le -(n+r+s)$. This implies that $J^r(\Gamma_s\m)_n$ is of weights $\le -(n+r+s)$. 

We now make a canonical choice of generators by setting
$$
\cW_{-(n+r+1)}V_n:= \cW_{-(n+r+1)}J^r\m_n.
$$
Set $V:= \prod_i \cW_iV$; the weight conditions above show that this has no intersection with $\Gamma_s\m$ for $s>1$, so the composition $V \to \m \to \m^{\ab}$ is injective. Since $\cW_{-(n+r+1)}(\m^{\ab})_n = \cW_{-(n+r+1)}J^r(\m^{\ab})_n$, the composition is also surjective, so $V$ is a space of generators for $\m$.

The structure of $\m$ is now determined by the differentials $d\co V_n \to \m_{n-1}$. As $\m= \Lie(V)=V\by \bigwedge^2V \by \Gamma_3\m$, weight and filtration considerations show that  we must have the projection $d\co V_n \to  (\Gamma_3\m)_{n-1}$ being $0$. The non-zero contributions to $d$ are $V_n \to V_{n-1}$, which is dual to $d_1$ on $E$, and $V_n \to \prod_{a+b=n-1}[V_a,V_b]$, which must be dual to the cup product. Thus $\m = G(E)$, and so  $R\ltimes \exp(\m) = R\ltimes \exp(G(E)) $,  as required.
\end{proof}

\begin{corollary}\label{htpyleray}\label{qformalrk}
For $X$ and $R$ as above, we can describe the relative Malcev homotopy groups $\varpi_n^{\et}(X^{R, \mal},x)$ explicitly in terms of the Leray spectral sequence as 
\[
\H_{n-1}(G( {}_J\!\EE^{*,*}_1)),
\]
for
\[
        {}_J\!\EE^{a,b}_1= \H^{2a+b}(\bar{X}, \oR^{-a}j_*\bO(\w\varpi^{\et}_f(X)^{L,\red}))
\]
as in Definition \ref{efil}, and  $G$ as in Definition \ref{barwg}. If the conditions of Theorem \ref{etpimal} hold (including goodness), then this also calculates $\pi_n^{\et}(X,x)\ten_{\hat{\Z}} \Ql$ as a Galois representation. 
\end{corollary}

\section{Varieties over local fields}\label{local} 

\subsection{Potentially good reduction, $\ell \ne p$}
Let $V'$ be a complete discrete valuation ring, with residue field $k'$ (finite, of characteristic $p \ne \ell$), and fraction field $K'$ (of characteristic $0$). Let $\bar{k}, \bar{K}$ be the algebraic closures of $k,K'$ respectively, and $\bar{V}$ the algebraic closure of $V'$ in $\bar{K}$. 

Let $X_{V'}= \bar{X}_{V'} -T_{V'}$ be a geometric fibration over $V'$ (in the sense of \cite[Definition 11.4]{fried}). 
Assume that we have a subfield $K \subset K'$ and a scheme $X_K/K$ such that $X_K\ten_KK'\cong X_{V'}\ten_{V'}K'$.
We wish to study the $\Gal(\bar{K}/K)$-action on the homotopy type $X_{\bar{K},\et}$.

Recall from \cite[ Theorem X.2.1]{sga1}  that the map $\pi^{\et}_f(\bar{X}_{k'}) \to\pi^{\et}_f(\bar{X}_{V'})$ is an equivalence. By ibid. \S XIII.2.10, this generalises to an equivalence $\pi_f^t(X_{k'}) \to\pi_f^t(X_{V'})$. Meanwhile, ibid. Corollary XIII.2.8 implies that $\pi_f^t(X_{\bar{K}}) \to \pi_f^t(X_{\bar{V}})$ is an epimorphism, and ibid. Corollary XIII.2.9 shows that $\pi^{\et}_f(X_{\bar{K}})^{\wedge_L} \to \pi^{\et}_f(X_{\bar{V}})^{\wedge_L}$ is an equivalence, where $L$ is any set of prime numbers excluding $p$.

\begin{proposition}\label{specn}
If $\vv$ is an $\ell$-adic local system on  $X_{\bar{V}}$,  tamely ramified along the divisor (i.e. coming from a representation of $\pi_f^t(X_{\bar{V}})$), then the maps
\begin{eqnarray*}
i_{\eta}^*\co \H^*(X_{\bar{V}}, \vv) &\to& \H^*(X_{\bar{K}},i_{\eta}^*\vv )\\
i_{s}^*\co \H^*(X_{\bar{V}}, \vv) &\to& \H^*(X_{\bar{k}},i_s^*\vv )\\
\end{eqnarray*}
are isomorphisms.
\end{proposition}
\begin{proof}
In \cite[Theorem 11.5]{fried}, this is proved for  $\pi^{\et}_f(X_{\bar{V}})^{\wedge_L}$-representations, for $p\notin L$. The same proof carries over to $\pi_f^t(X_{\bar{V}})$-representations, since the pro-$L$ hypothesis is only used to restrict the monodromy around the divisor.
\end{proof}

\begin{definition}
Since $\pi^{\et}_1(\Spec V')\cong \Gal(\bar{k}/k')$, we may  define $\w\varpi_f^t(X_{\bar{V}})$ analogously to Definition \ref{walg} as the maximal quotient of $\varpi_f^t(X_{\bar{V}}):= \pi_f^t(X_{\bar{V}})^{\alg}$ on which the Frobenius action is algebraic. Define ${}^{\pnr}\!\varpi_f^t(X_{\bar{K}})$ to be the image of $\varpi_f^t(X_{\bar{K}})\to \w\varpi_f^t(X_{\bar{V}})$, noting that this is a quotient of $\varpi_f^t(X_{\bar{K}})$ on which the $\Gal(\bar{K}/K)$-action is potentially unramified.
\end{definition}

Note that these definitions are independent of the choice of extension  $V'/V$, in the sense that  a finite extension $V''/ V'$ would give the same construction.

\begin{theorem}\label{locqformal}  
Let  $R$ be   any Frobenius-equivariant reductive quotient of ${}^{\pnr}\!\varpi_f^t(X_{\bar{K}})$.
Then the  $\Gal(\bar{K}/K)$-action  on the homotopy type
$$
X_{\bar{K}, \et}^{R, \mal}
$$
is algebraic, potentially unramified (in the sense of \S \ref{pnr}) and mixed (Definition \ref{pnrmixed}), 
giving a  canonical Galois-equivariant weight decomposition. It is also quasi-formal, corresponding to the $E_2$-term
$$
\bigoplus_{a,b} \H^a(\bar{X}_{\bar{K}}, \oR^bj_*\bO(R)) \in FDG\Alg(R),
$$
of the Leray spectral sequence for the immersion $j\co  X \to \bar{X}$. The formality quasi-isomorphism is equivariant with respect to the $\Gal(\bar{K}/K)$-action.
\end{theorem}
\begin{proof}
We know that the homotopy type is given by 
$$
\CC^{\bt}_{\et}(X_{\bar{K}}, \bO(R))\in c\Alg(R).
$$
From the definition of ${}^{\pnr}\!\varpi_f^t(X_{\bar{K}})$, we know that $\bO(R)$ is the pullback of a local system  on $X_{\bar{V}}$, so $i_{\eta*}\bO(R)$ is a local system and $i_{\eta}^*i_{\eta*}\bO(R)=\bO(R)$.

The equivalences of Proposition \ref{specn} now give quasi-isomorphisms
$$
\CC^{\bt}_{\et}(X_{\bar{K}}, \bO(R))= \CC^{\bt}_{\et}(X_{\bar{K}}, i_{\eta}^*i_{\eta*}\bO(R))\la \CC^{\bt}_{\et}(X_{\bar{V}}, i_{\eta*}\bO(R)) \to \CC^{\bt}_{\et}(X_{\bar{k}}, i_{s}^*i_{\eta*}\bO(R)),
$$
compatible with the basepoint augmentation maps.

We may assume that $K \subset K'$ is a Galois extension, then observe that the equivalences above imply that action of $\Gal(\bar{K}/K')$ is unramified, so the $\Gal(\bar{K}/K)^{\alg}$ action factors through $\Gal(\bar{K}/K)\by_{\Gal(\bar{k}/k')} \Gal(\bar{k}/k')^{\alg}$. In fact, Proposition \ref{wgtexists} implies that the action factors through $\Gal(\bar{K}/K)\by_{\Gal(\bar{k}/k')}M_{q'}$, where $q'= |k'|$, so the morphism $\bG_m \to M_{q'}^0= \ker(M_{q'} \to\Gal(\bar{k}/k'))$ provides the weight decomposition. This is compatible with the Galois action since $M_{q'}$ is commutative (being a quotient of $\Z^{\alg}$), so $\bG_m$ lies in the centre of  $\Gal(\bar{K}/K)\by_{\Gal(\bar{k}/k')}M_{q'}$.

We may now adapt Corollary \ref{qformal} to see that this is quasi-formal,  noting that all of the quasi-isomorphisms above extend naturally to the filtered algebras of Corollary \ref{qformal}. 
\end{proof}

\begin{corollary}\label{locqformalb1}
 Let  $X$ and $R$ be   as above. Then  the homotopy groups $\varpi^{\et}_n(X_{\bar{K}})$ are potentially unramified and mixed as Galois representations, giving them a canonical weight decomposition. They may also be recovered from the Leray spectral sequence, as in Corollary \ref{htpyleray}. 
\end{corollary}

\begin{corollary}\label{locqformalb2}
If $L$ is  a set of primes including $\ell$, and:
\begin{enumerate}
\item $\pi^{\et}_f(X)^{\wedge_L}$ is $(N+1)$-good relative to ${}^{\pnr}\!\varpi_f^t(X^{\wedge_L}_{\bar{K}})$,
\item $\pi^{\et}_n(X^{\wedge_L})\ten_{\hat{\Z}}\Ql$ is finite-dimensional for all $1<n\le N$, and

\item the action of $\ker(\pi^{\et}_f(X_{\bar{K}})^{\wedge_L}\to \pi_f^t(X_{\bar{V}})^{\wedge_L})$ on $\pi^{\et}_n(X^{\wedge_L}_{\bar{K}})\ten_{\hat{\Z}}\Ql$ is unipotent for all $1<n\le N$,
\end{enumerate} 
then the Galois action on  $\pi^{\et}_n(X^{\wedge_L}_{\bar{K}})\ten_{\hat{\Z}}\Ql$ is potentially unramified and mixed, giving it a canonical weight decomposition. It may also be recovered from the Leray spectral sequence.
\end{corollary}
\begin{proof}
Substitute $R=\pi_f^t(X_{\bar{V}})^{L, \red}$ into Corollary \ref{locqformalb1}, Corollary \ref{htpyleray} and Theorem   \ref{etpimal}.
\end{proof}
%
Note that  if $L$ does not contain $p$, then the third condition of the Corollary is vacuous.  

\subsection{Potentially good reduction,  $\ell=p$}\label{crissn}
%
%
\subsubsection{Convergent isocrystals}

Let $X,\bar{X},V',K,K',k'$ etc. be as in the previous section, but with $\ell=p$. Let $W'=W(k')$, the ring of Witt vectors over $k'$, and $K_0'$ the fraction field of $W'$; let $W^{\nr}:= W(\bar{k})$, with $K_0^{\nr}$ its fraction field. Choose a homomorphism $\sigma\co K' \to K'$ extending the natural action of the Frobenius operator $\phi$ on $W(k') \subset K'$. Assume moreover that $T_{V'}=D_{V'}$, a normal crossings divisor, or more generally that $D_{V'}$ corresponds to a log structure.

\begin{definition}
 Let $\mathrm{MF}_{(\bar{X}_{V'},D_{V'})/K'}^{\nabla}$ be the category of filtered convergent $F$-isocrystals on $(\bar{X}_{V'},D_{V'})$, as in \cite[\S 1]{tsujicaltech}  (or \cite[6.9]{olssonhodge}  when $K'$ is unramified, noting that the construction extends to ramified rings, as mentioned at the end of  \cite[1.14]{olssonhodge}).

Roughly speaking, an object  of $\mathrm{MF}_{(\bar{X}_{V'},D_{V'})/{K'}}^{\nabla}$ consists of an $F$-isocrystal $(E, \phi_E)$ on $(\bar{X}_k,D_k)/W$, together with a filtration $\Fil^i\sE$ of $\sE$ satisfying Griffiths transversality with respect to $\nabla_{\sE}$, where $(\sE, \nabla_{\sE})$ is the module with logarithmic connection on $(\bar{X}_{K'},D_{K'}) $ obtained by base change from the evaluation of $E$ on the $p$-adic completion of $(\bar{X}_{V'},D_{V'})$.
\end{definition}

\subsubsection{Crystalline \'etale sheaves}



We now introduce crystalline \'etale sheaves, as in \cite[V(f)]{Hop}  or \cite{andreattaiovita}.


\begin{definition}
We define the \emph{category of associations} on $(\bar{X}_{V'}, D_{V'})$ to consist of triples $(\vv, \iota, E)$, where
\begin{enumerate}
\item $\vv$ is a  smooth $\Q_p$-sheaf   on $X_{K'}$, 
\item $E \in \mathrm{MF}_{(\bar{X}_{V'}, D_{V'})}^{\nabla}(\Phi)$,
\item $\iota$  is an association isomorphism (\cite[\S 6.13]{olssonhodge}), i.e. 
 a collection of  isomorphisms 
$$
\iota_U\co \vv\ten_{\Q_p}B_{\cris}(\hat{U})\to E(B_{\cris}(\hat{U})) 
$$
for $U \to X_{V'}$ \'etale, compatible with the filtrations and semi-linear Frobenius automorphisms, and  with morphisms over $X$, so that $\iota$ becomes an isomorphism of \'etale presheaves. Here, $B_{\cris}(\hat{U})$ is  formed by applying Fontaine's construction to the $p$-adic completion $\hat{U}$ of $U$.
 \end{enumerate}

A morphism $f\co (\vv, \iota, E)\to (\vv', \iota', E')$ in the category of associations consists of a morphism $f^{\et}\co \vv \to \vv'$ and a morphism $f^{\cris}\co E \to E'$ such that $f^{\cris} \circ \iota= \iota' \circ f^{\et}\co \vv\ten_{\Q_p}B_{\cris}(\hat{U})\to E'(B_{\cris}(\hat{U}))$ for all $U$.
\end{definition}

The following lemma is a counterpart to \cite[Lemma 5.5]{Hop}, which gives the corresponding statements for the forgetful functor from associations to $\mathrm{MF}_{(X_{V'},D_{V'})}^{\nabla}$.
\begin{proposition}\label{assockey}
The forgetful functor $(\vv, \iota, E) \mapsto \vv$ from the category of associations to the category of smooth $\Q_p$-sheaves on $X_{K'}$ is full and faithful. Its essential image  is stable under extensions and subquotients.
\end{proposition}
\begin{proof}
Given associations  $(\vv, \iota, E)$ and  $(\vv', \iota', E')$, note that $(\vv^{\vee}\ten \vv, (\iota^{\vee})^{-1}\ten \iota', E^{\vee}\ten E')$ is another association. Giving a morphism $f^{\et}\co \vv\to \vv'$ amounts to giving an element of $\H^0(X_K, \vv^{\vee}\ten \vv')$, or equivalently a Galois-invariant  element of $\H^0(X_{\bar{K}}, \vv^{\vee}\ten \vv')$. By \cite[5.6]{Hop}, the map
$$
(\iota^{\vee})^{-1}\ten \iota' \co  \H^*(X_{\bar{K}}, \vv^{\vee}\ten \vv')\ten_{\Q_p}B_{\cris} \to \H_{\cris}^*(X_k/W, E^{\vee}\ten E')\ten_{{K'}_0}B_{\cris}
$$
is an isomorphism. Taking Galois-invariant and Frobenius-invariant elements in $\Fil^0$, this gives an isomorphism
$$
(\iota^{\vee})^{-1}\ten \iota' \co \H^0(X_{\bar{K}}, \vv^{\vee}\ten \vv')^{\Gal(\bar{K}/{K'})}\to \Fil^0\H_{\cris}^0(X_k/W, E^{\vee}\ten E')^{\phi},
$$
so there is a unique Frobenius-equivariant morphism $f^{\cris}\co E\to E'$ preserving the Hodge filtration such that the diagrams
$$
\begin{CD} 
\vv\ten_{\Q_p}B_{\cris}(\hat{U})@>{\iota}>> E(B_{\cris}(\hat{U}))\\
@V{f^{\et}\ten_{\Q_p}B_{\cris}}VV @VV{f^{\cris}(B_{\cris}(\hat{U}))}V\\
\vv'\ten_{\Q_p}B_{\cris}(\hat{U})@>{\iota'}>> E'(B_{\cris}(\hat{U}))
\end{CD}
$$
commute. This shows that the forgetful functor is full and faithful.

To see that the essential image is stable under extensions, observe that  extensions     of $\vv$ by $\vv'$ are parametrised by elements $a$ of 
$\H^1(X_{K'}, \vv^{\vee}\ten \vv')$. The isomorphisms above then show that $((\iota^{\vee})^{-1}\ten \iota')(a)$ is a Frobenius-equivariant element of $\Fil^0\H_{\cris}^1(X_k/W, E^{\vee}\ten E')$, so gives a unique  extension of $(\vv, \iota, E)$ by $(\vv', \iota', E')$ in the category of associations.

Finally, note that the subquotient of an extension is  an extension of subquotients, so it suffices to show that the essential image contains subquotients of semisimple objects. Since such a subquotient  $\vv'$ of $\vv$ is isomorphic to a direct summand, we have an idempotent endomorphism  $\pi$ of  $\vv$ with $\ker \pi \cong \vv'$. Since the forgetful functor is full, $\pi$ lifts to an idempotent endomorphism $\tilde{\pi}$ of $(\vv, \iota, E)$, so $\vv'$ underlies $\ker \tilde{\pi}$. 
\end{proof}

\begin{definition}
Say that a smooth $\Q_p$-sheaf $\vv$   on $X_{K'}$ is \emph{crystalline} if it lies in the essential image of the forgetful functor from the category of associations.
\end{definition}

\begin{proposition}
The fibre functors $(\vv, \iota, E) \mapsto \vv_{\bar{x}}$ make the category of associations into a multifibred Tannakian category. The corresponding pro-algebraic groupoid $\varpi_f^{\et}(X_{K'})^{\cris}$ is a quotient of $\varpi_f^{\et}(X_{K'})$. Moreover, $\varpi_f^{\et}(X_{K'})^{\cris}$ is the Malcev completion of $\pi_f^{\et}(X_{K'})$ with respect to the reductive quotient $\varpi_f^{\et}(X_{K'})^{\cris,\red}$. 
\end{proposition}
\begin{proof}
Associations form a $\Q_p$-linear rigid abelian tensor category, with $(\vv, \iota, E)\ten (\vv', \iota', E')= (\vv\ten_{\Q_p}\vv', \iota\ten \iota', E\ten_{\O_{X_k, \cris}}E')$ and $(\vv, \iota, E)^{\vee}= (\vv^{\vee}, (\iota^{-1})^{\vee}, E^{\vee})$.

By Proposition \ref{assockey}, associations are equivalent to the Tannakian subcategory of crystalline \'etale sheaves in $\Rep(\varpi_f^{\et}(X_{K'}))$. Thus the forgetful functor from associations to smooth $\Q_p$-sheaves corresponds to a surjection $\varpi_f^{\et}(X_{K'})\to \varpi_f^{\et}(X_{K'})^{\cris}$ of pro-algebraic groupoids (with the same object set).

For $\rho\co \pi_f^{\et}(X_{K'})\to \varpi_f^{\et}(X_{K'})^{\cris,\red}$,   representations of $\varpi_f^{\et}(X_{K'})^{\rho, \mal}$ are smooth $\Q_p$-sheaves on $X_{K'}$  which are Artinian extensions of semisimple crystalline \'etale sheaves. By Proposition \ref{assockey}, this is equivalent to the category $\Rep(\varpi_f^{\et}(X_{K'})^{\cris}) $ of associations. 
\end{proof}

\begin{definition}
Say that a smooth $\Q_p$-sheaf $\vv$   on $X_{K}$ is \emph{potentially crystalline} if $\vv|_{X_{K''}}$ is crystalline for some finite extension $K' \subset K''$.
\end{definition}

\subsubsection{Equivariant pro-algebraic fundamental groups}

\begin{definition}\label{galalg}
Define ${}^{\cris,{K'}}\!\varpi^{\et}_f(X_{\bar{K}})$ to be the image of $\varpi^{\et}_f(X_{\bar{K}})\to \varpi^{\et}_f(X_{K'})^{\cris}$.
 \end{definition}

Note that we can also   characterise ${}^{\cris,{K'}}\!\varpi^{\et}_f(X_{\bar{K}})$  as 
\[
        \ker(\varpi^{\et}_f(X_{K'})^{\cris}\to  \Gal(\bar{K}/{K'})^{\cris}=\varpi^{\et}_f(\Spec {K'})^{\cris}),
\]
 using the right-exactness  of pro-algebraic completion. Thus 
$$
\varpi^{\et}_f(X_{K'})^{\cris}= {}^{\cris,{K'}}\!\varpi^{\et}_f(X_{\bar{K}}) \rtimes \Gal(\bar{K}/{K'})^{\cris},
$$
so 
representations of ${}^{\cris,{K'}}\!\varpi^{\et}_f(X_{\bar{K}}) $ correspond to smooth $\Q_p$-sheaves on $X_{\bar{K}}$ arising as subsheaves of pullbacks of  crystalline \'etale $\Q_p$-sheaves on $X_{K'}$.

\begin{definition}
Define 
$$
{}^{\pcris}\!\varpi^{\et}_f(X_{\bar{K}}):= \Lim_{K''}{}^{\cris,K''}\!\varpi^{\et}_f(X_{\bar{K}}),
$$
where the limit is taken over all finite Galois extensions $K' \subset K''$.
\end{definition}

Finite-dimensional representations of ${}^{\pcris}\!\varpi^{\et}_f(X_{\bar{K}}) $ thus correspond to smooth $\Q_p$-sheaves on $X_{\bar{K}}$ arising as subsheaves of pullbacks of  potentially crystalline smooth $\Q_p$-sheaves on $X_{K}$. 

Since $\cG= \Lim_{K''}(\Gal(\bar{K}/K'')^{\cris}\by_{\Gal(\bar{K}/K'')}\Gal(\bar{K}/K))$, 
this gives an isomorphism 
$$
\Lim_{K''}(\Gal(\bar{K}/K)\by_{\Gal(\bar{K}/K'')  } \varpi^{\et}_f(X_{K''})^{\cris})\cong {}^{\pcris}\!\varpi^{\et}_f(X_{\bar{K}}) \rtimes \cG^{\pcris},
$$
so the Galois action on ${}^{\pcris}\!\varpi^{\et}_f(X_{\bar{K}})$ is algebraic and potentially crystalline. 
 
\begin{lemma}\label{katzlattice}
The map $\varpi^{\et}_f(X_{\bar{K}}) \onto {}^{\pnr}\!\varpi_f(\bar{X}_{\bar{K}})$ factors through ${}^{\pcris}\!\varpi^{\et}_f(X_{\bar{K}})$.
\end{lemma}
\begin{proof}
Since 
$$
\Gal(\bar{K}/K)^{\pnr} \ltimes {}^{\pnr}\!\varpi_f(\bar{X}_{\bar{K}}) = \Lim_{K''}  \Gal(\bar{K}/K)\by_{\Gal(\bar{k}/k'')}\varpi_f(\bar{X}_{k''}), 
$$  
it suffices to show that the map $\varpi_f(X_{K''}) \to \varpi_f(\bar{X}_{k''})$ factors through $\varpi_f^{\et}(X_{K''})^{\cris}$. By looking at representations, this is equivalent to saying that every smooth $\Q_p$-sheaf on $\bar{X}_{k''}$ pulls back to give a crystalline \'etale sheaf on $X_{V''}$. This now follows from \cite[4.1.1]{katzlattice}, which shows that smooth $\Q_p$-sheaves on $\bar{X}_{k''}$ correspond to unit-root $F$-lattices on $X_{V''}$.
\end{proof}

\begin{definition}
Any field extension $K' \to K''$ gives a pullback functor $ \mathrm{MF}_{(\bar{X}_{V'},D_{V'})/K'}^{\nabla}\to \mathrm{MF}_{(\bar{X}_{V''},D_{V''})/K''}^{\nabla}$, and we set
$$
\mathrm{MF}_{(\bar{X}_{\bar{V}},D_{\bar{V}})/\bar{K}}^{\nabla}:= \LLim_{K''} \mathrm{MF}_{(\bar{X}_{V''},D_{V''})/K''}^{\nabla},
$$
where $K''$ ranges over all finite field extensions $K' \subset K''$.
\end{definition}

Representations of $\Gal(\bar{K}/K)^{\pcris,0} \ltimes {}^{\pcris}\!\varpi^{\et}_f(X_{\bar{K}})$ are just representations of $ \Lim_{K''} \varpi^{\et}_f(X_{K''})^{\cris}$, so the category of finite-dimensional representations is $\LLim_{K''} \FD\Rep( \varpi^{\et}_f(X_{K''})^{\cris})$. 

\begin{definition}\label{dpcrisx}
Making use of the forgetful functor from associations to filtered convergent $F$-isocrystals, the observation above gives us a $\Q_p$-linear functor
$$
D_{\pcris}^X\co   \FD\Rep(\Gal(\bar{K}/K)^{\pcris,0} \ltimes {}^{\pcris}\!\varpi^{\et}_f(X_{\bar{K}})) \to \mathrm{MF}_{(\bar{X}_{\bar{V}},D_{\bar{V}})/\bar{K}}^{\nabla}.
$$
Say that an object of $\mathrm{MF}_{(\bar{X}_{\bar{V}},D_{\bar{V}})/\bar{K}}^{\nabla} $ is \emph{potentially admissible} if it lies in the essential image of $D_{\pcris}^X $.
\end{definition}
Note that $D_{\pcris}^{\Spec K} =D_{\pcris}$.

\begin{definition}
Given a $\cG^0$-equivariant affine scheme $Y$ over $\Q_p$, define the affine scheme $D_{\pcris}(Y)$ over $K_0^{\nr}$ by 
$$
 D_{\pcris}(Y)= \Spec D_{\pcris}O(Y).
$$

Observe that $O(Y)$ is therefore an ind-object of (i.e. a sum of objects in) the  category  $\mathrm{MF}_{(\Spec \bar{V},\emptyset)/\bar{K}}^{\nabla}$.
\end{definition}

\begin{proposition}\label{assocgpds}
The category of finite-dimensional $D_{\pcris}({}^{\pcris}\!\varpi^{\et}_f(X_{\bar{K}}))$-representations in potentially admissible objects of $\mathrm{MF}_{(\Spec \bar{V},\emptyset)/\bar{K}}^{\nabla}$   is equivalent to the category of finite-dimensional $ \cG^{\pcris,0} \ltimes {}^{\pcris}\!\varpi^{\et}_f(X_{\bar{K}})$-representations, which in turn is equivalent to the category of potentially admissible objects of  $\mathrm{MF}_{(\bar{X}_{\bar{V}},D_{\bar{V}})/\bar{K}}^{\nabla} $. 

For any point $x \in X_{\bar{V}}(\bar{K})$, the associated fibre functor from $D_{\pcris}({}^{\pcris}\!\varpi^{\et}_f(X_{\bar{K}})) $-representations to $\mathrm{MF}_{(\Spec \bar{V},\emptyset)/\bar{K}}^{\nabla}$ corresponds under this equivalence to the pullback 
$$
x^*\co \mathrm{MF}_{(\bar{X}_{\bar{V}},D_{\bar{V}})/\bar{K}}^{\nabla}\to \mathrm{MF}_{(\Spec \bar{V},\emptyset)/\bar{K}}^{\nabla}. 
$$
\end{proposition}
\begin{proof}
A $D_{\pcris}({}^{\pcris}\!\varpi^{\et}_f(X_{\bar{K}}))$-representation $V$ in potentially admissible objects of $\mathrm{MF}_{(\Spec \bar{V},\emptyset)/\bar{K}}^{\nabla}$ consists of potentially admissible objects $V(x) \in \mathrm{MF}_{(\Spec \bar{V},\emptyset)/\bar{K}}^{\nabla}$ for all $x \in \Ob ({}^{\pcris}\!\varpi^{\et}_f(X_{\bar{K}}))$, together with coassociative morphisms
$$
V(y) \to V(x) \ten D_{\pcris}O( {}^{\pcris}\!\varpi^{\et}_f(X_{\bar{K}})(x,y))
$$
in $\mathrm{MF}_{(\Spec \bar{V},\emptyset)/\bar{K}}^{\nabla}$.

Since $D_{\pcris}$ gives an equivalence between $\cG^{\pcris,0}$-representations and potentially admissible objects of $\mathrm{MF}_{(\Spec \bar{V},\emptyset)/\bar{K}}^{\nabla}$, the description above shows that it defines the required equivalence from  $ \cG^{\pcris,0} \ltimes {}^{\pcris}\!\varpi^{\et}_f(X_{\bar{K}})$-representations. 

Now, $ \cG^{\pcris,0} \ltimes {}^{\pcris}\!\varpi^{\et}_f(X_{\bar{K}}) \cong \Lim_{K''} \varpi^{\et}_f(\bar{X}_{K''})^{\cris}$, so we may apply the functor $D_{\pcris}^X$ from Definition \ref{dpcrisx}, mapping to  potentially admissible objects in  $\mathrm{MF}_{(\bar{X}_{\bar{V}},D_{\bar{V}})/\bar{K}}^{\nabla} $. By \cite[Lemma 5.5]{Hop}, this functor is full and faithful, so gives us the second  equivalence required.
\end{proof}

\begin{definition}
Define
$$
\Isoc((\bar{X}_{\bar{k}}, D_{\bar{k}})/K_0^{\nr}):= \LLim_{K''}\Isoc((\bar{X}_{k''}, D_{k''})/K'')
$$
to be the category of isocrystals on $\Lim_{K''} (\bar{X}_{k''}, D_{k''})/K''$, where the limit is taken over finite extensions $K' \subset K''$.
\end{definition}

\begin{proposition}\label{crisgpdreps}
The category of finite-dimensional  $D_{\pcris}({}^{\pcris}\!\varpi^{\et}_f(X_{\bar{K}}))$-representations over $K_0^{\nr}$ is equivalent to a full subcategory of  $\Isoc((\bar{X}_{\bar{k}}, D_{\bar{k}})/K_0^{\nr} )$.
 This subcategory is the smallest  full abelian subcategory containing  the potentially admissible  objects of $\mathrm{MF}_{(\bar{X}_{\bar{V}},D_{\bar{V}})/\bar{K}}^{\nabla} $. 
\end{proposition}
\begin{proof}
Write $G:={}^{\pcris}\!\varpi^{\et}_f(X_{\bar{K}})$, and let $\bO(G)$ be the universal $G$-representation in smooth $\Q_p$-sheaves on $X_{\bar{K}}$, as defined in Definition \ref{OR}. Following through the proof of Proposition \ref{assocgpds}, the functor from $D_{\pcris}(G)$-representations in potentially admissible objects of $\mathrm{MF}_{(\Spec \bar{V},\emptyset)/\bar{K}}^{\nabla}$ to $\mathrm{MF}_{(\bar{X}_{\bar{V}},D_{\bar{V}})/\bar{K}}^{\nabla} $ is given by
$$
F(A) := A\ten^{D_{\pcris}(G)} D_{\pcris}^X\bO(G), 
$$
while its inverse is
$$
F_*(\sA) := \LLim_{K''}\H^0_{\cris}((\bar{X}_{k''}, D_{k''}), \sA\ten D_{\pcris}^X\bO(G)).
$$

The same formulae define left exact functors $F, F_*$ between the category of finite-dimensional $D_{\pcris}(G)$-representations and  $\Isoc((\bar{X}_{\bar{k}}, D_{\bar{k}})/K_0^{\nr} )$. For any point $x \in X(\bar{K})$,
$$
F(A)_x= A\ten^{D_{\pcris}(G)} D_{\pcris}(\bO(G)_x)= A\ten^{D_{\pcris}(G)} D_{\pcris}(O(G)(x,-))=A(x),
$$
so $F$ is exact.

For any $D_{\pcris}(G)$-representation $A$,
 \begin{eqnarray*}
F_*F(A)&=& A\ten^{D_{\pcris}(G)} \LLim_{K''}\H^0_{\cris}((\bar{X}_{k''}, D_{k''}),   D_{\pcris}^X\bO(G)\ten D_{\pcris}^X\bO(G))\\
&=& A\ten^{D_{\pcris}(G)}D_{\pcris}O(G)\\
&=& A.
\end{eqnarray*}
Moreover, $F_*$ is right adjoint to $F$, since a morphism $A \to F_*(\sA')$ is equivalent to a $G$-equivariant morphism $A\ten \O_{X, \cris} \to \sA'\ten D_{\pcris}^X\bO(G)$ of isocrystals, which is equivalent to a $G$-equivariant $D_{\pcris}^X\bO(G)$-linear morphism $A\ten D_{\pcris}^X\bO(G) \to \sA'\ten D_{\pcris}^X\bO(G)  $, which (taking $G$-invariants) is just a morphism $F(A) \to \sA'$. These two statements combine to show that $F$ is full and faithful.

Since $F$ is exact, its essential image is an abelian subcategory. Proposition \ref{assocgpds} ensures that it contains all  potentially admissible objects of  $\mathrm{MF}_{(\bar{X}_{\bar{V}},D_{\bar{V}})/\bar{K}}^{\nabla} $, so we need only show that anything in the image of $F$ is in the abelian subcategory generated by these potentially admissible objects.  

Given any $D_{\pcris}(G)$-representation $A$, we have a canonical embedding $A \into A\ten D_{\pcris}(O(G))$, which is a sum of objects of $\mathrm{MF}_{(\Spec \bar{V},\emptyset)/\bar{K}}^{\nabla}$. Thus for some finite-dimensional subobject $U$, we  have an embedding $A \into U$. Replacing $A$ with $U/A$, we get an embedding $U/A \into U'$, so $A= \ker(U \to U')$, and hence $F(A)= \ker(F(U) \to F(U'))$. Since $F(U)$ and $F(U')$ are potentially admissible objects of  $\mathrm{MF}_{(\bar{X}_{\bar{V}},D_{\bar{V}})/\bar{K}}^{\nabla} $, this completes the proof.
\end{proof}

\subsubsection{Crystalline homotopy types}

Fix  a Galois-equivariant quotient $R$ of ${}^{\pcris}\!\varpi^{\et}_f(X_{\bar{K}})^{\red}$, or rather of its full subgroupoid on objects $X(\bar{K})$

\begin{definition}
Let  $\sF \to \sC_{\cris}^{\bt}(\sF)$ be a choice of functor from isocrystals to cosimplicial sheaves on the log-crystalline site, with the property that $\sC_{\cris}^{\bt}(\sF)$ is a resolution of $\sF$, compatible with tensor products, and acyclic for log-crystalline cohomology. Examples of such a functor are given in \cite[p.\pageref{gal-logcrishtpy}]{gal}, or by denormalising the construction $DR$ of \cite[4.29.2]{olssonhodge}. In both cases, the resolution is given by first choosing a resolution which is acyclic for the derived functor between crystalline and  Zariski sites (such as denormalisation of the de Rham complex), then taking a \v Cech resolution.

Define
$$
\CC_{\cris}^{\bt}(Y,\sF):= \Gamma(Y,\sC_{\cris}^{\bt}(\sF)),
$$  
observing that this construction will also be compatible with tensor products.
\end{definition}


\begin{definition}
Define the \emph{relative crystalline homotopy type} $X_{\bar{k},\cris}^{D_{\pcris}(R), \mal}$ \emph{over} $D_{\pcris}R$ to be the pro-algebraic homotopy type in $\Ho(s\cE(D_{\pcris}R)_*)$ (over $K_0^{\nr}$) corresponding under Theorem \ref{bigequiv} to the $D_{\pcris}(R)$-representation 
$$
\CC_{\cris}^{\bt}((\bar{X}_{\bar{k}}, D_{\bar{k}}), D_{\pcris}^X\bO(R))
$$
in cosimplicial $K_0^{\nr}$-algebras, equipped with its natural augmentations to $D_{\pcris}O(R)(x,-) = \CC_{\cris}^{\bt}(\Spec K_0^{\nr}, x^*D_{\pcris}^X\bO(R))$ coming from elements  $x \in X(\bar{V})$.
\end{definition}

\begin{lemma}
There is a canonical equivalence between representations of $\varpi_f(X_{\bar{k}}/K_0^{\nr})_{\cris}^{D_{\pcris}(R), \mal}$ and  a full  subcategory  of  $\Isoc((\bar{X}_{\bar{k}}, D_{\bar{k}})/K_0^{\nr} )$. Objects of this category are Artinian extensions of those  isocrystals corresponding under Proposition \ref{crisgpdreps} to $D_{\pcris}(R)$-representations.
\end{lemma}
\begin{proof}
This is \cite[Theorem 2.28]{olssonhtpy}. An alternative approach would be to note that the proof of \cite[Theorem 2.9]{gal} 
 carries over to non-nilpotent torsors.
\end{proof}

\begin{definition}
For a topos $\cT$, if $\sC^{\bt}_{\cT}(\sS)$ is a canonical cosimplicial ${\cT}$-resolution of a sheaf $\sS$ of algebras on $X$, with $\CC^{\bt}_{\cT}(X, \sS):= \Gamma(X, \sC^{\bt}_{\cT}(\sS))$, then for any morphism $f\co X \to Y$ we have a bicosimplicial algebra $\CC^{\bt}_{\cT}(Y, f_*\sC^{\bt}_{\cT}(\sS))$, and we define
$$
\CC^{\bt}_{\cT}(f, \sS):= \tau''\CC^{\bt}_{\cT}(Y, f_*\sC^{\bt}_{\cT}(\sS))\in Fc\Alg,
$$
defined as in Definition \ref{tau''}.
\end{definition}

\begin{definition}\label{leraytypecris}
If we write $j$ for the embedding $X \into \bar{X}$, define the \emph{ filtered relative crystalline homotopy type} $(X_{\bar{k},\cris},j_{\bar{k},\cris})^{D_{\pcris}(R), \mal}$ \emph{over} $D_{\pcris}R$
to be the filtered pro-algebraic homotopy type in $\Ho(s\cE(D_{\pcris}R)_*)$ (over $K_0^{\nr}$) corresponding under Theorem \ref{fbigequiv} to the filtered $D_{\pcris}(R)$-representation 
$$
\CC_{\cris}^{\bt}(j_{\bar{k},\cris}, D_{\pcris}^X\bO(R))
$$
in cosimplicial $K_0^{\nr}$-algebras,  equipped with its natural augmentations to $D_{\pcris}O(R)(x,-) = \CC_{\cris}^{\bt}(\Spec K_0^{\nr}, x^*D_{\pcris}^X\bO(R))$ coming from elements  $x \in X(\bar{V})$.
\end{definition}

\subsubsection{Comparison of homotopy types}

From now on, let $B:=B_{\cris}(V)$ and $\tilde{B}:= \tilde{B}_{\cris}(V)$, from Definition \ref{bcrisdef}.


\begin{proposition}\label{crisequiv}
For any
Galois-equivariant quotient $R$ of ${}^{\pcris}\!\varpi^{\et}_f(X_{\bar{K}})^{\red}$, 
there is a chain of  $(\phi,\cG^0)$-equivariant quasi-isomorphisms 
$$
 X_{\bar{K},\et}^{R, \mal}\ten_{\Q_p}\tilde{B} \sim X_{\bar{k},\cris}^{ D_{\pcris}R, \mal}\ten_{K_0^{\nr}}\tilde{B} 
$$
in $s\Aff_{\tilde{B}}(R)_*$.
\end{proposition}
\begin{proof}
This amounts to establishing a chain of quasi-isomorphisms
$$
\CC^{\bt}_{\et}(X_{\bar{K}}, \bO(R))\ten_{\Q_p}\tilde{B} \sim \CC^{\bt}_{\cris}(X_{\bar{k}}/K_0^{\nr}, D_{\pcris}^X\bO(R))\ten_{K_0^{\nr}}\tilde{B}
$$
in $c\Alg_{\tilde{B}}(R)_* $

In the notation of \cite[4.29 and 5.21]{olssonhodge}, $\CC^{\bt}_{\cris}(X_k/K_0^{\nr}, D_{\pcris}^X\bO(R))$ and $\CC^{\bt}_{\et}(X_{\bar{K}}, \bO(R)) $ are quasi-isomorphic to the denormalisations of     $\oR \Gamma_{\cris}(D_{\pcris}^X\bO(R) )$ and $GC(\bO(R), X(\bar{K}))$,  since denormalisation and Thom-Sullivan are quasi-inverse up to homotopy (as in Remark \ref{cfolsson}). 

Since the affine group schemes $R/\Q_p$ and $D_{\pcris}(R)/K_0^{\nr}$ are associated by an isomorphism
$$
B\ten_{\Q_p}O(R) \cong B\ten_{K_0^{\nr}}D_{\pcris}O(R),
$$ 
the required result is then ibid. 6.15.1, combined with the observation in ibid. Proposition 6.19 that pullback preserves associations, thus ensuring that these associations are compatible with the augmentation maps coming from basepoints.

The proof of ibid. 6.15.1 proceeds by adapting the isomorphisms on cohomology groups from  \cite[5.6]{Hop}  to quasi-isomorphisms of DG algebras. Since the latter proves that the cohomological isomorphisms   respect cup products, an alternative  approach would be to  extend the isomorphisms to  quasi-isomorphisms of the  minimal $E_{\infty}$-algebras they underlie. Remark \ref{einfty} would then imply that the corresponding objects in $dg\hat{\cN}(R)$ are weakly equivalent. 
\end{proof}

\begin{remark}
When $L$ is a crystalline \'etale sheaf on $X_K$ and $R$ is the Zariski closure of the image of $\pi_1^{\et}(X_{\bar{K}}, \bar{x}) \to \GL(L_{\bar{x}})$ with nilpotent monodromy around each component of the  divisor, then Proposition \ref{crisequiv} is effectively \cite[Theorem 1.7]{olssonhodge}  (replacing ``crystalline'' with ``potentially crystalline'' throughout). The nilpotent hypothesis was needed for Tannakian considerations, which in our case are obviated by Proposition \ref{assockey}. 
\end{remark}

\begin{theorem}\label{outercris}
Given a Galois-equivariant quotient $R$ of ${}^{\pcris}\!\varpi^{\et}_f(X_{\bar{K}})$,  the  Galois action on $X_{\bar{K},\et}^{R \mal}$ is algebraic and potentially crystalline.
\end{theorem}
\begin{proof}
In the notation of \S \ref{pcris}, we need to show that the map $\cG \to \Aut^h( X_{\bar{K},\et}^{R, \mal})$ factors through $\cG^{\pcris}$. 
Apply Proposition \ref{pcristest} to Proposition \ref{crisequiv}, taking
$$
Y=\Aut^h(X_{\bar{K},\et}^{R, \mal})\by_{\Aut(R)} \cG^{\pcris,0}
$$
with the $\cG^0$ action on $Y$ given by left multiplication.

Now, note that $D_{\pcris}(\cG^{\pcris,0}\by R) = D_{\pcris}(\cG^{\pcris,0})\by R$, giving a $K_0^{\nr}$-linear map $f\co D_{\pcris}(\cG^{\pcris,0})\by R \to D_{\pcris}R$. In fact, $D_{\pcris}(\cG^{\pcris,0})= \Spec B^{\ker(\cG^0 \to \cG^{\pcris,0})}$, so this map just comes from the isomorphism $(D_{\pcris}O(R))\ten_{K_0^{\nr}}B \cong O(R)\ten_{\Q_p}B$.

We now define $Z$ over  $D_{\pcris}\cG^{\pcris,0}$ to be the affine scheme given by
$$
Z(A) =  \Iso_{\Ho(dg\Aff_A(R)_*)}( X_{\bar{K},\et}^{R, \mal}\ten_{\Q_p}A, f^{\sharp}(X_{\cris}^{D_{\pcris}R, \mal}\ten_{K_0^{\nr}}A) ),
$$
for  $D_{\pcris}O(\cG^{\pcris,0})$-algebras $A$.

Since $\cG^{\pcris,0}$ is potentially crystalline, we have an isomorphism $\alpha\co \cG^{\pcris,0}\by \Spec \tilde{B}\to  (D_{\pcris}\cG^{\pcris,0}) \by_{\Spec K_0^{\nr}} \Spec \tilde{B}$, so the scheme $Z\by_{\Spec K_0^{\nr}} \Spec \tilde{B}$ can be regarded as a scheme over $\cG^{\pcris,0}\by \Spec \tilde{B}$

The $\cG^0$-equivariant isomorphism of Proposition  \ref{crisequiv} then gives,  for any $D_{\pcris}O(\cG^{\pcris,0})\ten_{K_0^{\nr}}\tilde{B}$-algebra $A$, a $\cG^0$-equivariant   isomorphism
$$
Z(A) \cong \Iso_{\Ho(dg\Aff_A(R)_*)}( X_{\bar{K},\et}^{R, \mal}\ten_{\Q_p}A, \alpha^{\sharp} X_{\bar{K},\et}^{R, \mal}\ten_{\Q_p}A),
$$
but the right-hand side is just $Y(A)$, giving a $\cG^0$-equivariant isomorphism
$$
Z\by_{ K_0^{\nr}} \Spec \tilde{B}_{\cris} \cong Y \by_{\Q_p} \Spec \tilde{B}_{\cris}, 
$$
as required.
\end{proof}

\begin{corollary}
For $x, y \in X(\bar{K})$, the $\cG^0$-actions on 
$$
\varpi_n(X_{\bar{K},\et}^{R, \mal},x) \quad \text{ and }\quad \varpi_f(X_{\bar{K},\et}^{R, \mal})(x,y)
$$
are potentially crystalline.
\end{corollary}
\begin{proof}
This is just the observation that the map $\Aut(X_{\bar{K},\et}) \to \Aut(\varpi_n(X_{\bar{K},\et}^{R, \mal},x) )(\Q_p)$ factors through $\Aut^h(X_{\bar{K},\et}^{R, \mal})$.
\end{proof}
Note that if we set $R=1$ and look at the fundamental group, this recovers the comparison theorem of \cite{shihokey} and \cite{vologodsky} between pro-unipotent \'etale and crystalline fundamental groups.

In fact, we may extend Proposition \ref{crisequiv} to a filtered version:
\begin{proposition}\label{fcrisequiv}
For any
Galois-equivariant quotient $R$ of ${}^{\pcris}\!\varpi^{\et}_f(X_{\bar{K}})^{\red}$ and for 
$j\co X \to \bar{X}$,  there is a chain of canonical $(\phi,\cG^0)$-equivariant  quasi-isomorphisms 
$$
 (X_{\bar{K},\et}, j_{\bar{K}, \et})^{R, \mal}\ten_{\Q_p}\tilde{B} \sim (X_{\bar{k},\cris},j_{\bar{k},\cris})^{D_{\pcris}(R), \mal}\ten_{K_0^{\nr}}\tilde{B} 
$$
in $Fs\Aff_{\tilde{B}}(R)_*$.
\end{proposition}
\begin{proof}
The proof of Proposition \ref{crisequiv} adapts.
\end{proof}

Lacking a suitable $p$-adic analogue of Lafforgue's Theorem (although \cite[Theorem 6.3.4]{kedlaya} might provide a viable replacement in some cases), we now impose a purity hypothesis.
\begin{assumption}\label{assume}
Assume that $D_{\pcris}^X\bO(R)$ is an ind-object in the category of $\iota$-pure overconvergent $F$-isocrystals. Like Definition \ref{walg}, this is equivalent to saying that for every $R$-representation  $V$, the corresponding sheaf $\vv$ on $X_{\bar{K}}$ can be embedded in the pullback of a crystalline \'etale  sheaf $\bU$ on $X_{K''}$, associated to an $\iota$-pure overconvergent $F$-isocrystal on $(\bar{X}_{k''}, D_{k''})/K''$, for some finite extension $K' \subset K''$. Also note that this implies that the Frobenius action on $D_{\pcris}O(R)$ is $\iota$-pure.
\end{assumption}

\begin{example}
To see how the hypotheses of Assumption \ref{assume} arise naturally, assume that $f\co  Y_K \to X_K$ is a geometric fibration (in the sense of 
\cite[Definition 11.4]{fried}, for instance any smooth proper morphism) with   connected components, for $Y$ of potentially good reduction. Let $G(\bar{x}, \bar{z})$ be the Zariski closure of the map
$$
\pi_f^{\et}(X_{\bar{K}})(\bar{x}, \bar{z}) \to \prod_n \Iso( (\oR^n f_{\bar{K},*}^{\et} \Q_p)_{\bar{x}}, (\oR^n f_{\bar{K},*}^{\et} \Q_p)_{\bar{z}}),
$$
so $G$ is a pro-algebraic groupoid on objects $X(\bar{K})$, and then set $R= G^{\red}$. By \cite{Hop}, $\oR^n f_{\bar{K},*}^{\et} \Q_p$ is associated to $\oR^n f_{\bar{k},*}^{\cris} \O_{Y_{\bar{k}}, \cris}$, which 
by \cite[Theorem 6.6.2]{kedlaya}  is $\iota$-pure (or if $f$ is not proper, globally $\iota$-mixed). Thus the semisimplifications of the $G$-representations $\bar{x} \mapsto (\oR^n f_{\bar{K},*}^{\et} \Q_p)_{\bar{x}}$ are direct sums of $\iota$-pure representations. Since these generate the Tannakian category of $R$-representations, the hypotheses are satisfied.

For $\bar{x} \in X(\bar{K})$, we may write $F:= Y\by_{f,X, \bar{x}} \Spec \bar{K}$, and Theorem \ref{lfibrations} then shows that the homotopy fibre of
$$
(Y_{\bar{K}}^{\et})^{R, \mal} \to (X_{\bar{K}}^{\et})^{R, \mal}
$$
over $\bar{x}$ is $(F_{\bar{K}}^{\et})^{1, \mal}$.
\end{example}

\begin{example}
A more comprehensive example would be to let $G(\bar{x}, \bar{z})$ be the Zariski closure of the map $\pi_f^{\et}(X_{\bar{K}})(\bar{x}, \bar{z}) \to \prod_{n,f} \Iso( (\oR^n f_{\bar{K},*}^{\et} \Q_p)_{\bar{x}}, (\oR^n f_{\bar{K},*}^{\et} \Q_p)_{\bar{z}})$, where $f$ ranges over all geometric fibrations of potentially good reduction with connected components, and then to set $R:= G^{\red}$. The resulting homotopy type $(X_{\bar{K}}^{\et})^{R, \mal}$ would be very close to possible conceptions of a pro-algebraic motivic homotopy type.
\end{example}

\begin{theorem}\label{criswgtexists}
Given a Galois-equivariant   quotient $R$ of ${}^{\pcris}\!\varpi^{\et}_f(X_{\bar{K}})$ satisfying Assumption \ref{assume},   the Galois action on $X_{\bar{K},\et}^{R, \mal}$ is $\iota$-mixed in the sense of Definition \ref{crismixed},   giving a canonical weight decomposition on 
$X_{\bar{K},\et}^{R, \mal}\ten B^{\sigma}$.
\end{theorem}
\begin{proof}
This is essentially the same as Proposition \ref{wgtexists}. Frobenius gives a canonical element of  $\Aut^h(X_{\cris}^{D_{\pcris}R, \mal} )$. We first show that this is $\iota$-mixed of integral weights. By Lemma \ref{cohohelps}, we need only consider the Frobenius action on cohomology
$$
\H^*_{\cris}((\bar{X}_{\bar{k}},D_{\bar{k}}), D_{\pcris}^X\bO(R)).
$$
The Leray spectral sequence gives 
$$
 \H^{2a+b}_{\cris}(\bar{X}_{\bar{k}}, \oR^{-a}_{\cris}j_*D_{\pcris}^X\bO(R))\abuts \H^{a+b}_{\cris}((\bar{X}_{\bar{k}},D_{\bar{k}}), D_{\pcris}^X\bO(R)).
$$
If we write $D^{(n)}$ for the normalisation of the $n$-fold intersection of the local components of $D$, and $i_n\co  D^{(n)}\to \bar{X}$ for the embedding, then as in  \cite[3.2.4.1]{Hodge2}, there is an isomorphism
$$
\H^{2a+b}_{\cris}(\bar{X}_{\bar{k}}, \oR^{-a}_{\cris}j_*D_{\pcris}^X\bO(R))\cong \H^{2a+b}_{\cris}(D^{(-a)}_{\bar{k}}, i_n^*j_*D_{\pcris}^X\bO(R)(a)),
$$
since $j_*D_{\pcris}^X\bO(R)$ is associated to a locally constant sheaf on $X$. 


Now, 
\cite[Theorem 6.6.2]{kedlaya}  combined with Poincar\'e duality proves that  $\H^{2a+b}_{\cris}(D^{(-a)}_{\bar{k}}, i_n^*j_*D_{\pcris}^X\bO(R)(a))$ is $\iota$-pure of weight $b$. Thus Lemma \ref{cohohelps} implies that the Frobenius element of $\Aut^h(X_{\cris}^{D_{\pcris}R, \mal} )$ is $\iota$-mixed of integral weights.

We need to show that  the composite morphism 
$$
\Z^{\alg,0} \to \cG^{\pcris}\ten_{\Q_p}B^{\sigma}\to \Aut^h(X_{\bar{K},\et}^{R, \mal})\ten_{\Q_p}B^{\sigma}
$$
factors through $M_{\iota}^0$. By Proposition \ref{crisequiv}, 
$$
\Aut^h(X_{\bar{K},\et}^{R, \mal})\ten_{\Q_p}\tilde{B}^{\sigma}\cong \Aut^h(X_{\cris}^{D_{\pcris}R, \mal} )\ten_{K_0^{\nr}}\tilde{B}^{\sigma},
$$
so the map
$$
\Z^{\alg,0} \to \cG^{\pcris}\ten_{\Q_p}B^{\sigma}\to \Aut^h(X_{\bar{K},\et}^{R, \mal})\ten_{\Q_p}\tilde{B}^{\sigma}
$$
factors through $M_{\iota}^0$. Since $B^{\sigma} \subset \tilde{B}^{\sigma}$, this completes the proof.
\end{proof}

\begin{theorem}\label{crisqformal}
For $R$ as in Theorem \ref{criswgtexists}, the filtered homotopy type $(X_{\bar{K},\et}, j_{\bar{K},\et})^{R, \mal}\ten B^{\sigma}$  is  quasi-formal, corresponding to the $E_2$-term
$$
{}_J\!\EE_1^{a,b}( X_{\bar{K},\et}^{R, \mal})\ten B^{\sigma} =\bigoplus_{a,b} \H^{2a+b}(\bar{X}_{\bar{K}}, \oR^{-b}j_*\bO(R))\ten B^{\sigma} \in FDG\Alg_{ B^{\sigma}}(R),
$$
of the Leray spectral sequence for the immersion $j\co  X \to \bar{X}$, and the formality isomorphism is  equivariant with respect to the  Galois action.

The filtered homotopy type $(X_{\bar{K},\et}, j_{\bar{K},\et})^{R, \mal}$ is also quasi-formal, but the formality isomorphism is not in general Galois-equivariant or canonical.
\end{theorem}
\begin{proof}
Since the Galois action is $\iota$-mixed in the sense of Definition \ref{crismixed}, there is a Galois-equivariant weight  decomposition $\bG_m \to \RAut_J(X_{\bar{K},\et}^{R, \mal}\ten {B}^{\sigma})$, using Lemma \ref{pcrisfrob} and the observation after Definition \ref{inq}. The argument of Corollary \ref{qformal} now adapts to show that $X_{\bar{K},\et}^{R, \mal}\ten {B}^{\sigma}$  is quasi-formal, with  the formality quasi-isomorphism  equivariant under the Galois action, proving the first part.

In particular this implies that
$$
\RAut_J(X_{\bar{K},\et}^{R, \mal} )({B}^{\sigma}) \to \Aut({}_J\!\EE_1^{*,*}( X_{\bar{K}}^{R, \mal}))( {B}^{\sigma})
$$
is a pro-unipotent extension. Thus the corresponding morphism of pro-algebraic groups is surjective,  allowing us to lift the weight decomposition on $\EE_1^{*,*}( X_{\bar{K}}^{R, \mal})$ non-canonically to $X_{\bar{K},\et}$. This decomposition need not be compatible with the canonical decomposition on $X_{\bar{K},\et}^{R, \mal}\ten {B}^{\sigma}$. 
The argument of Corollary \ref{qformal} adapted to this decomposition now  shows that $X_{\bar{K},\et}^{R, \mal}$  is quasi-formal. 
\end{proof}

\begin{corollary}\label{crisqformalrk}
For $X$ and $R$ as above, we can describe the homotopy groups  $\varpi_n^{\et}(X_{\bar{K}}^{R, \mal},x)\by_{\Spec \Q_p} \Spec B^{\sigma}$ explicitly in terms of the Leray spectral sequence as 
\[
\varpi_n^{\et}(X_{\bar{K}}^{R, \mal},x)^{\vee}\ten_{ \Q_p} B^{\sigma} =\H^{n-1}(G( {}_J\!\EE^{*,*}_1( X_{\bar{K},\et}^{R, \mal}))^{\vee})\ten_{ \Q_p} B^{\sigma},
\]
for
 $G$ as in Definition \ref{barwg}. Of course, if the conditions of Theorem \ref{etpimal} hold (including goodness), then this also calculates $\pi_n^{\et}(X_{\bar{K}},x)\ten_{\hat{\Z}} B^{\sigma}$ as a Galois representation. 
 \end{corollary}

\begin{remarks}\label{finalrks}
\begin{enumerate}
\item
In the case when $X$ is projective and $R$ is a quotient of ${}^{\Gal}\!\varpi_f(X_{\bar{k}})$, this is essentially the main formality result of \cite[\S 4]{olssonhtpy}, which has since been extended to the general projective case in \cite[Theorem 7.22]{olssonhodge}, although Frobenius-equivariance is not made explicit there. The proofs also differ in that they work with minimal algebras, rather than minimal Lie algebras. 

\item
Although at first sight Theorem \ref{crisqformal} is weaker than Theorem \ref{locqformal}, it is more satisfactory in one important respect. Theorem \ref{locqformal} effectively shows that relative Malcev $\ell$-adic homotopy types carry no more information than cohomology, whereas to recover a relative Malcev $p$-adic homotopy type from Theorem \ref{crisqformal}, we still need to identify  $(X_{\bar{K},\et}, j_{\bar{K},\et})^{R, \mal} \subset (X_{\bar{K},\et}, j_{\bar{K},\et})^{R, \mal}\ten {B}^{\sigma}$. This must be done by describing the Hodge filtration on $(X_{\cris}^{ D_{\pcris}R, \mal},j_{\bar{k},\cris})$, which is not determined by cohomology (since it is not Frobenius-equivariant). Thus the Hodge filtration is the only really new structure on the relative Malcev homotopy type. This phenomenon is  similar to the formality results for mixed Hodge structures in  \cite[\S \ref{mhs-real}]{mhs}.
\end{enumerate}
\end{remarks}

\bibliographystyle{alphanum}
\bibliography{references}

\begin{thebibliography}{DMOS}

\bibitem[AI]{andreattaiovita}
Fabrizio Andreatta and Adrian Iovita.
\newblock Comparison isomorphisms for smooth formal schemes.
\newblock www.mathstat.concordia.ca/faculty/iovita/research.html, 2009.

\bibitem[AM]{arma}
M.~Artin and B.~Mazur.
\newblock {\em Etale homotopy}.
\newblock Lecture Notes in Mathematics, No. 100. Springer-Verlag, Berlin, 1969.

\bibitem[BK]{bousfieldkan}
A.~K. Bousfield and D.~M. Kan.
\newblock {\em Homotopy limits, completions and localizations}.
\newblock Lecture Notes in Mathematics, Vol. 304. Springer-Verlag, Berlin,
  1972.

\bibitem[Del1]{Hodge2}
Pierre Deligne.
\newblock Th{\'e}orie de {H}odge. {II}.
\newblock {\em Inst. Hautes {\'E}tudes Sci. Publ. Math.}, (40):5--57, 1971.

\bibitem[Del2]{Weil2}
Pierre Deligne.
\newblock La conjecture de {W}eil. {II}.
\newblock {\em Inst. Hautes {\'E}tudes Sci. Publ. Math.}, (52):137--252, 1980.

\bibitem[DK]{pathgpd}
W.~G. Dwyer and D.~M. Kan.
\newblock Homotopy theory and simplicial groupoids.
\newblock {\em Nederl. Akad. Wetensch. Indag. Math.}, 46(4):379--385, 1984.

\bibitem[DMOS]{tannaka}
Pierre Deligne, James~S. Milne, Arthur Ogus, and {Kuang-yen} Shih.
\newblock {\em Hodge cycles, motives, and {S}himura varieties}, volume 900 of
  {\em Lecture Notes in Mathematics}.
\newblock Springer-Verlag, Berlin, 1982.

\bibitem[Fal]{Hop}
Gerd Faltings.
\newblock Crystalline cohomology and {$p$}-adic {G}alois-representations.
\newblock In {\em Algebraic analysis, geometry, and number theory (Baltimore,
  MD, 1988)}, pages 25--80. Johns Hopkins Univ. Press, Baltimore, MD, 1989.

\bibitem[Fon]{fontaineBT}
Jean-Marc Fontaine.
\newblock Sur certains types de repr{\'e}sentations {$p$}-adiques du groupe de
  {G}alois d'un corps local;\ construction d'un anneau de {B}arsotti-{T}ate.
\newblock {\em Ann. of Math. (2)}, 115(3):529--577, 1982.

\bibitem[Fri]{fried}
Eric~M. Friedlander.
\newblock {\em {\'E}tale homotopy of simplicial schemes}, volume 104 of {\em
  Annals of Mathematics Studies}.
\newblock Princeton University Press, Princeton, N.J., 1982.

\bibitem[GJ]{sht}
Paul~G. Goerss and John~F. Jardine.
\newblock {\em Simplicial homotopy theory}, volume 174 of {\em Progress in
  Mathematics}.
\newblock Birkh{\"a}user Verlag, Basel, 1999.

\bibitem[Hai]{hainrelative}
Richard~M. Hain.
\newblock The {H}odge de {R}ham theory of relative {M}alcev completion.
\newblock {\em Ann. Sci. {\'E}cole Norm. Sup. (4)}, 31(1):47--92, 1998.

\bibitem[Hir]{Hirschhorn}
Philip~S. Hirschhorn.
\newblock {\em Model categories and their localizations}, volume~99 of {\em
  Mathematical Surveys and Monographs}.
\newblock American Mathematical Society, Providence, RI, 2003.

\bibitem[HM1]{hainmatrelative}
Richard Hain and Makoto Matsumoto.
\newblock Relative pro-{$l$} completions of mapping class groups.
\newblock {\em J. Algebra}, 321(11):3335--3374, 2009.

\bibitem[HM2]{Levi}
G.~Hochschild and G.~D. Mostow.
\newblock Pro-affine algebraic groups.
\newblock {\em Amer. J. Math.}, 91:1127--1140, 1969.

\bibitem[Hov]{Hovey}
Mark Hovey.
\newblock {\em Model categories}, volume~63 of {\em Mathematical Surveys and
  Monographs}.
\newblock American Mathematical Society, Providence, RI, 1999.

\bibitem[HS]{HinSch}
V.~A. Hinich and V.~V. Schechtman.
\newblock On homotopy limit of homotopy algebras.
\newblock In {\em {$K$}-theory, arithmetic and geometry ({M}oscow,
  1984--1986)}, volume 1289 of {\em Lecture Notes in Math.}, pages 240--264.
  Springer, Berlin, 1987.

\bibitem[Isa]{isaksen}
Daniel~C. Isaksen.
\newblock A model structure on the category of pro-simplicial sets.
\newblock {\em Trans. Amer. Math. Soc.}, 353(7):2805--2841 (electronic), 2001.

\bibitem[Jan]{jancts}
Uwe Jannsen.
\newblock Continuous {\'e}tale cohomology.
\newblock {\em Math. Ann.}, 280(2):207--245, 1988.

\bibitem[Kat]{katzlattice}
Nicholas~M. Katz.
\newblock {$p$}-adic properties of modular schemes and modular forms.
\newblock In {\em Modular functions of one variable, III (Proc. Internat.
  Summer School, Univ. Antwerp, Antwerp, 1972)}, pages 69--190. Lecture Notes
  in Mathematics, Vol. 350. Springer, Berlin, 1973.

\bibitem[Ked]{kedlaya}
Kiran~S. Kedlaya.
\newblock Fourier transforms and {$p$}-adic `{W}eil {II}'.
\newblock {\em Compos. Math.}, 142(6):1426--1450, 2006.

\bibitem[Kon]{Kon}
Maxim Kontsevich.
\newblock Topics in algebra --- deformation theory.
\newblock Lecture Notes, available at
  http://www.math.brown.edu/$\sim$abrmovic/kontsdef.ps, 1994.

\bibitem[KPT1]{KTP}
L.~Katzarkov, T.~Pantev, and B.~To{\"e}n.
\newblock Schematic homotopy types and non-abelian {H}odge theory.
\newblock {\em Compos. Math.}, 144(3):582--632, 2008.
\newblock arXiv math.AG/0107129 v5.

\bibitem[KPT2]{schematicv2}
L.~Katzarkov, T.~Pantev, and B.~To{\"e}n.
\newblock Algebraic and topological aspects of the schematization functor.
\newblock {\em Compos. Math.}, 145(3):633--686, 2009.
\newblock arXiv math.AG/0503418 v2.

\bibitem[KW]{KW}
Reinhardt Kiehl and Rainer Weissauer.
\newblock {\em Weil conjectures, perverse sheaves and {$l$}'adic {F}ourier
  transform}, volume~42 of {\em Ergebnisse der Mathematik und ihrer
  Grenzgebiete. 3. Folge. A Series of Modern Surveys in Mathematics [Results in
  Mathematics and Related Areas. 3rd Series. A Series of Modern Surveys in
  Mathematics]}.
\newblock Springer-Verlag, Berlin, 2001.

\bibitem[Laf]{La}
Laurent Lafforgue.
\newblock Chtoucas de {D}rinfeld et correspondance de {L}anglands.
\newblock {\em Invent. Math.}, 147(1):1--241, 2002.

\bibitem[Mac]{mac}
Saunders MacLane.
\newblock {\em Categories for the working mathematician}.
\newblock Springer-Verlag, New York, 1971.
\newblock Graduate Texts in Mathematics, Vol. 5.

\bibitem[Mag]{magid}
Andy~R. Magid.
\newblock On the proalgebraic completion of a finitely generated group.
\newblock In {\em Combinatorial and geometric group theory ({N}ew {Y}ork,
  2000/{H}oboken, {NJ}, 2001)}, volume 296 of {\em Contemp. Math.}, pages
  171--181. Amer. Math. Soc., Providence, RI, 2002.

\bibitem[Mil]{Mi}
James~S. Milne.
\newblock {\em {\'E}tale cohomology}.
\newblock Princeton University Press, Princeton, N.J., 1980.

\bibitem[Mor]{Morgan}
John~W. Morgan.
\newblock The algebraic topology of smooth algebraic varieties.
\newblock {\em Inst. Hautes {\'E}tudes Sci. Publ. Math.}, (48):137--204, 1978.

\bibitem[Ols1]{olssonhodge}
Martin Olsson.
\newblock Towards non-abelian $p$-adic {H}odge theory in the good reduction
  case.
\newblock {\em Mem. Amer. Math. Soc.}, To appear.

\bibitem[Ols2]{olssonhtpy}
Martin~C. Olsson.
\newblock {$F$}-isocrystals and homotopy types.
\newblock {\em J. Pure Appl. Algebra}, 210(3):591--638, 2007.

\bibitem[Pri1]{gal}
J.~P. Pridham.
\newblock Galois actions on the pro-$l$-unipotent fundamental group.
\newblock arXiv math.AG/0404314 v4, 2004.

\bibitem[Pri2]{paper1}
J.~P. Pridham.
\newblock Deforming {$l$}-adic representations of the fundamental group of a
  smooth variety.
\newblock {\em J. Algebraic Geom.}, 15(3):415--442, 2006.

\bibitem[Pri3]{htpy}
J.~P. Pridham.
\newblock Pro-algebraic homotopy types.
\newblock {\em Proc. London Math. Soc.}, 97(2):273--338, 2008.
\newblock arXiv math.AT/0606107 v8.

\bibitem[Pri4]{weight1}
J.~P. Pridham.
\newblock Weight decompositions on {\'e}tale fundamental groups.
\newblock {\em Amer. J. Math.}, 131(3):869--891, 2009.
\newblock arXiv math.AG/0510245 v5.

\bibitem[Pri5]{mhs}
J.~P. Pridham.
\newblock Formality and splitting of real non-abelian mixed {H}odge structures.
\newblock arXiv: 0902.0770v2 [math.AG], submitted, 2010.

\bibitem[Pri6]{monad}
J.~P. Pridham.
\newblock The homotopy theory of strong homotopy algebras and bialgebras.
\newblock {\em Homology, Homotopy Appl.}, 12(2):39--108, 2010.
\newblock arXiv:0908.0116v2 [math.AG].

\bibitem[Pri7]{ddt1}
J.~P. Pridham.
\newblock Unifying derived deformation theories.
\newblock {\em Adv. Math.}, 224(3):772--826, 2010.
\newblock arXiv:0705.0344v5 [math.AG].

\bibitem[Pri8]{heid}
J.~P. Pridham.
\newblock On $\ell$-adic pro-algebraic and relative pro-$\ell$ fundamental
  groups.
\newblock In {\em The arithmetic of fundamental groups ({PIA} 2010)}. Springer,
  to appear.

\bibitem[Qui1]{quick}
Gereon Quick.
\newblock Profinite homotopy theory.
\newblock {\em Doc. Math.}, 13:585--612, 2008.

\bibitem[Qui2]{QRat}
Daniel Quillen.
\newblock Rational homotopy theory.
\newblock {\em Ann. of Math. (2)}, 90:205--295, 1969.

\bibitem[Qui3]{Qpf}
Daniel~G. Quillen.
\newblock An application of simplicial profinite groups.
\newblock {\em Comment. Math. Helv.}, 44:45--60, 1969.

\bibitem[Sch]{alexschmidt}
Alexander Schmidt.
\newblock Extensions with restricted ramification and duality for arithmetic
  schemes.
\newblock {\em Compositio Math.}, 100(2):233--245, 1996.

\bibitem[Ser1]{Se}
Jean-Pierre Serre.
\newblock {\em Lie algebras and {L}ie groups}.
\newblock Springer-Verlag, Berlin, second edition, 1992.
\newblock 1964 lectures given at Harvard University.

\bibitem[Ser2]{galoisienne}
Jean-Pierre Serre.
\newblock {\em Cohomologie galoisienne}, volume~5 of {\em Lecture Notes in
  Mathematics}.
\newblock Springer-Verlag, Berlin, fifth edition, 1994.

\bibitem[SGA]{sga1}
{\em Rev{\^e}tements {\'e}tales et groupe fondamental}.
\newblock Springer-Verlag, Berlin, 1971.
\newblock S{\'e}minaire de G{\'e}om{\'e}trie Alg{\'e}brique du Bois Marie
  1960--1961 (SGA 1), Dirig{\'e} par Alexandre Grothendieck. Augment{\'e} de
  deux expos{\'e}s de M. Raynaud, Lecture Notes in Mathematics, Vol. 224.

\bibitem[Shi]{shihokey}
Atsushi Shiho.
\newblock Crystalline fundamental groups and {$p$}-adic {H}odge theory.
\newblock In {\em The arithmetic and geometry of algebraic cycles ({B}anff,
  {AB}, 1998)}, volume~24 of {\em CRM Proc. Lecture Notes}, pages 381--398.
  Amer. Math. Soc., Providence, RI, 2000.

\bibitem[To{\"e}]{chaff}
Bertrand To{\"e}n.
\newblock Champs affines.
\newblock {\em Selecta Math. (N.S.)}, 12(1):39--135, 2006.
\newblock arXiv math.AG/0012219.

\bibitem[Tsu]{tsujicaltech}
Takeshi Tsuji.
\newblock Crystalline sheaves, syntomic cohomology and $p$-adic polylogarithms.
\newblock Caltech seminar notes,
  modular.math.washington.edu/swc/notes/files/DLSTsuji2.pdf, 2001.

\bibitem[Vol]{vologodsky}
Vadim Vologodsky.
\newblock Hodge structure on the fundamental group and its application to
  {$p$}-adic integration.
\newblock {\em Mosc. Math. J.}, 3(1):205--247, 260, 2003.

\bibitem[Wei]{W}
Charles~A. Weibel.
\newblock {\em An introduction to homological algebra}.
\newblock Cambridge University Press, Cambridge, 1994.

\end{thebibliography}
\end{document}